\documentclass[11pt,reqno]{amsart}

\makeatletter
\renewcommand{\email}[2][]{%
  \ifx\emails\@empty\relax\else{\g@addto@macro\emails{,\space}}\fi%
  \@ifnotempty{#1}{\g@addto@macro\emails{\textrm{(#1)}\space}}%
  \g@addto@macro\emails{#2}%
}
\makeatother

\usepackage[affil-it]{authblk}
\usepackage[utf8]{inputenc}
\usepackage{amsmath,amsthm,amsfonts,amssymb,bm}
\usepackage{fancyhdr,amscd}
\usepackage{anysize}
\usepackage{graphicx,epsfig}
\usepackage{amsthm,latexsym,amsmath,amssymb}
\usepackage{mathrsfs}
\usepackage{indentfirst}
\usepackage{xcolor}
\usepackage[titletoc]{appendix}
\usepackage{comment}
\usepackage[left=1.5cm,right=1.5cm,top=2cm,bottom=2cm]{geometry}
\usepackage{enumerate}

\usepackage{amsmath, nccmath}

\numberwithin{equation}{section}

\def\curl{{\rm curl}}

\DeclareMathOperator{\dv}{div}

\newcommand{\beq}{\begin{equation}}
\newcommand{\eeq}{\end{equation}}
\newcommand{\ben}{\begin{eqnarray}}
\newcommand{\een}{\end{eqnarray}}
\newcommand{\beno}{\begin{eqnarray*}}
\newcommand{\eeno}{\end{eqnarray*}}

\newcommand{\pa}{\partial}

\newcommand{\R}{\mathcal{R}}

\newcommand{\psca}[1]{\left\langle#1\right\rangle}

\newtheorem{theorem}{\textbf Theorem}[section]
\newtheorem{lemma}{\textbf Lemma}[section]
\newtheorem{rem}{\textbf Remark}[section]

\newtheorem{prop}{\textbf Proposition}[section]

\newcounter{remark}
{\par \stepcounter{remark} {\it Remark
\arabic{section}.\arabic{remark}.}~}{\rm \end Proof\par}

\title{On the role of the displacement current and the Cattaneo's law on boundary layers of plasma}

\usepackage{amsaddr}

\author[F. De Anna, N. AARACH, M. Paicu, N. Zhu]{Nacer Aarach$^{1}$,\,Francesco De Anna$\,^2$, Marius Paicu$^3$ and Ning Zhu$^{3}$}

\address{$^{1,\,3}$ Bordeaux Institute of Mathematics, University of Bordeaux, France\\
$^{1}$ nacer.aarach@math.u-bordeaux.fr\\
$^{3}$  marius.paicu@math.u-bordeaux.fr}

\address{$\,^2$ Institute of Mathematics, University of Würzburg, Germany\\
   francesco.deanna@mathematik.uni-wuerzburg.de}
\address{$^4$ School of Mathematics, Shandong University, 
P.R. China\\
ning.zhu@sdu.edu.cn}   

\begin{document}

\begin{abstract}  

In the present paper, we aim to mathematically analyse the role of the displacement current and the Cattaneo's law on the boundary-layer theory of plasma, when the corresponding characteristic speed is relativistic. We restrict our analysis to two-dimensional flows and we study the asymptotic limit of the Navier-Stokes-Maxwell equations with Cattaneo's law near a bounding flat line, when the Hartmann, Reynolds and magnetic Reynolds numbers proportionally diverge to $\infty$. 

\noindent 
The goal of this paper is twofold. We first show that the extended version of the Navier-Stokes-Maxwell equations leads to a new family of boundary layers, which are hyperbolic both on the momentum equation and the Ampere's law.  
Secondly, we address the well-posedness of the derived equations and show the existence of global-in-time analytic solutions for small initial data. 

\noindent 
Our modelling highlights which conditions on the dimensionless parameters allow to interpret the proposed system as boundary layers with thickness typical of Prandtl or Hartmann. Furthermore, our development shows that the conditions related to Hartmann might be more physically acceptable. 

\noindent 
Finally, our analysis suggests that the Cattaneo's law and the displacement current might indeed stabilise the derived system in terms of existence of global-in-time analytic solutions. 

\end{abstract} 

\maketitle

\noindent 
\textbf{2010 Mathematics Subject Classification:} 35Q35, 35Q61, 35Q85, 76W05, 76F40.

\noindent 
\textbf{Keywords:} Boundary layers, Navier-Stokes-Maxwell equations, Cattaneo's law, displacement current, analytic solutions.

\section{Introduction}

\noindent 
The mathematical study of electrically conducting fluids and hot plasma has received for many years numerous investigations. The understanding of the underlying equations (MHD equations or Navier-Stokes-Maxwell equations) has significantly provided a fascinating number of implications, both on technological processes and physical experiments. Among the remarkable variety of applications, MHD flows are ubiquitous in contexts like astronomy (hydrodynamics of plasma in neutron stars and white dwarfs), nuclear fusion reactors (self-cooled liquid metal blankets) and metallurgic (liquid metal stirring). 

\noindent 
In this paper we are interested in deriving and analysing a family of partial differential equations that mathematically account for boundary layers of plasma and electrically conducting fluids, when the corresponding characteristic speed is of relativistic order. This specific hydrodynamics near a wall surface has been a topic of constant interest in astrophysics. For instance, these boundary layers are expected to be dominant sources of X-ray production in neutron stars \cite{Popham_2001},  gravitational radiation \cite{PhysRevD.64.044009} and magnetic reconnection \cite{Voros}.

\noindent 
The mathematical treatment of boundary layers in electrically conducting fluids has a long history, which dates back to the pioneering work of Hartmann \cite{hartmann1937theory}. Hartmann studied a duct flow of a viscous electrically-conducting fluid under the influence of a transverse magnetic field. Oriented at the right angle, the magnetic field produced additional viscosity, separating the channel into two main regions, boundary-layer region (Hartmann layers) and central core region. 

\noindent
Afterwards, many theoretical investigations and experiments have been developed around this theory, the most of them under the assumption that the electromagnetic variations of plasma are non-relativistic (i.e.~the characteristic speed of plasma has magnitude consistently lower than the speed of light). 
This hypothesis relaxes several terms of the Maxwell's equations, in particular, it neglects the so-called displacement current in the Ampere's law (a source of the magnetic field related to the ratio between the characteristic speed of plasma and the speed of light). 

\noindent 
The lack of the displacement current may or may not be satisfactory, depending on the modelling context. In neutron stars, for instance, strong time-dependent electric field could develop, when the plasma density falls below a critical value \cite{10.1093/mnras/staa774}. Thus, the associated displacement current makes up for the deficit of the plasma density and plays a major role in the evolution of the magnetic field.

\noindent
To the best of our knowledge, it remains still an open problem to mathematically understand boundary layers of plasma, whose magnetic field in the Navier-Stokes-Maxwell's equations is affected by the displacement current. This paper is therefore a first mathematical attempt to address this issue. In details, we derive and analyse the following system of PDEs (written in dimensionless form):
\begin{equation}
	\label{eq:main-system}
	\left\{\;
	\begin{aligned}
		& 
		\mathbb J
		\partial_{t}^2u + 
		\pa_tu+u\pa_x u+v\pa_yu -\pa_y^2u+\pa_xp= \mathbb{H}^2  \big(b_1  b_2 v  - u  b_2^2 - b_2e\big) 
		\hspace{0.4cm}(t,x,y)\in 
		&&  (0,T)\times \mathbb{R}\times  (0,1),\\
		& \pa_y p= \mathbb{H}^2  \big( b_1b_2 u - b_1^2 v + b_1 e \big) 
		&& (0,T)\times \mathbb{R}\times  (0,1),\\
		&  \frac{\kappa}{\rm Pr_m} \pa_{t }^2 b_1 + \pa_tb_1+u\pa_x b_1+v\pa_yb_1 -\frac{1}{\rm Pr_m} \pa_y^2b_1 = 
		 b_1 \partial_x u + 
            b_2 \partial_y u 
		&& (0,T)\times \mathbb{R}\times  (0,1) ,\\
		&  \frac{\kappa}{\rm Pr_m}  \pa_{t }^2 b_2 + \pa_tb_2+u\pa_x b_2+v\pa_yb_2 -
		 \frac{1}{\rm Pr_m}
		 \pa_y^2b_2 =
		   	b_1 \partial_x v + b_2 \partial_y v 
		&& (0,T)\times \mathbb{R}\times  (0,1),\\
		& \partial_t b_1 + \partial_y e = 0 \quad \text{and}\quad 
        \partial_t b_2 - \partial_x e  = 0  
        && (0,T)\times \mathbb{R}\times  (0,1),\\
		& \pa_xu+\pa_yv=0 \hspace{0.43cm}\text{and}\quad  \pa_xb_1+\pa_yb_2=0  
		&& (0,T)\times \mathbb{R}\times  (0,1),
	\end{aligned}
	\right.
\end{equation}
coupled with initial and boundary conditions
\begin{equation}\label{boundary-condition-our-main-eq}
    {\rm (IC)}
    \left\{
	\begin{aligned}
	    & (u,b_1,b_2)|_{t=0}=(\bar{u},\bar{b}_1,\bar{b}_2)  
		&& \mathbb{R}\times  (0,1),\\
		& (\partial_t u,\pa_tb_1, \pa_t b_2 )|_{t=0}=(\tilde{u},\tilde{b}_1,\tilde{b}_2)
		\hspace{-0.2cm}
		&&  \mathbb{R}\times  (0,1),
	\end{aligned}
	\right. 
	{\rm (BC)}
    \left\{
	\begin{aligned}
	    & (u,b_1,b_2,e)|_{y=1}=(0,0,0,0)  
		&& (0,T)\times \mathbb{R},\\
		& (u,b_1,b_2,e)|_{y=1}=(0, \mathbf b_1,\mathbf b_2, \mathbf e) 
		&& (0,T)\times \mathbb{R}.
	\end{aligned}
	\right.
\end{equation}

\subsection{Overview of System \eqref{eq:main-system}}$\,$

\noindent
For the sake of a clear presentation, we consider here a rather simple geometry, as well as simple boundary conditions. We assume indeed that the conducting fluid is restricted to the whole half space, in other words System  \eqref{eq:main-system} represents the behaviour of the fluid on a thin layer near a (flat) boundary line $(y = 0)$.
The velocity field $(u,v)^T$ satisfies no-slip boundary conditions, while the magnetic field $(b_1, b_2)^T$ and the electric-field intensity $e$ are assumed constant on the boundary 
(a scenario which is typical when the surrounding medium is an insulator).

\noindent 
All state variables $(u,v,b_1,b_2,e)$ in \eqref{eq:main-system} depend on time $t\in (0,T)$ and space $(x,y)\in \mathbb{R}\times (0,1)$. The vector fields $(u,v)^T \in \mathbb{R}^2$ and $(b_1,b_2)^T\in \mathbb{R}^2$ are divergence free and stand for the velocity and magnetic fields of plasma, respectively. The electric field assumes size $e\in \mathbb R$ and is perpendicular to the plane containing the plasma.  

\noindent
All constants $\mathbb{H}$, $\kappa$, ${\rm Pr_m}$ and $\mathbb J$ are positive and depend on standard dimensionless parameters of magnetohydrodynamics.
More precisely, $\mathbb{H}$ stands for the asymptotic of the ratio between the Hartmann number ${\rm Ha}>0$ and the Reynolds number $\rm Re$, as ${\rm Re}\to \infty$.  The magnetic Prandtl number ${\rm Pr_m}>0$ is assumed in this work constant and represents the ratio between the Reynolds number $\rm Re$ and the magnetic Reynolds number $\rm Re_m$. 
Hence viscous and magnetic diffusions are proportional, a regime typical of heavier white dwarfs, in which Reynolds and magnetic Reynolds numbers range between $10^{14}$ to $10^{15}$ (cf.~Section $2$ in \cite{Isern_2017}). Furthermore, the proportionality between ${\rm Ha}$ and ${\rm Re_m}$ reflects a threshold for the initiation of magnetic advection and subsequent reconnection (cf.~for instance Section $3.3$ in \cite{10.1111/j.1365-2966.2004.07898.x}, for the binary star AE Acquarii).

\subsection{Novelties of the model}$\,$

\noindent
Although System \eqref{eq:main-system} differs intrinsically from previous models (such as the Prandtl-MHD equations, cf.~\eqref{eq:MHD-Prandtl-literature}), the major novelties reside in particular within the terms ${\kappa}/{\rm Pr_m} \partial_{tt}^2 b_1 $ and ${\kappa}/{\rm Pr_m}\partial_{tt}^2 b_2$ for the equations of $b_1$ and $b_2$ (due to the displacement current \cite{MR3490904}), as well as within the term $\mathbb J 	\partial_{t}^2u$ in the equation for $u$. 
The role of the underlying constant $\kappa>0$ is exploited in details in Section \ref{sec:modelling} (cf.~Theorem \ref{thm:main-modelling-sec2} and Theorem \ref{thm:main-modelling2Hartmann}), it relates however to the ratio $(U_0/c)^2$ between the characteristic speed of plasma $U_0>0$ and the speed of light ``$c$'', for high value of ${\rm Re }\gg 1$. 

\noindent
The constant $\mathbb J\geq 0$ together with the second time derivative $\partial_t^2 u$ are derived from a well-known hyperbolic extension of the Navier-Stokes equations, a model which is known as Navier-Stokes with Cattaneo's law (cf.~\cite{MR3942552,MR2045417,Coulaud,MR4429384,MR2404054,MR3085226}). This extension was first proposed in fluid-dynamics by Carrassi and Morro \cite{carrassi1972AMN} (inspired by the original work of Cattaneo in the study of heat diffusion \cite{MR0032898,MR95680}). As most compelling reason to introduce this term, a positive value of $\mathbb J>0$ avoids indeed an infinite speed of propagation of $u$, which would be quite unnatural when considering the evolution of fluids at large scale. 

\smallskip
\noindent
The general concern of this paper is twofold. First, we aim to derive model \eqref{eq:main-system} from suitable asymptotics of the full Navier-Stokes-Maxwell's equations under Cattaneo's law (cf.~Section \ref{intro:modelling}). 
 
\noindent 
Secondly we aim to study the underlying well-posedness theory and prove the existence of global-in-time smooth solutions for system \eqref{eq:main-system}, by considering initial data that are small and highly regular (more precisely analytic in the horizontal coordinate $x\in \mathbb R$, cf.~Section \ref{intro:analysis}).




\subsection{A brief overview of the analysis of MHD boundary layers}$\,$

\noindent 
The analysis of boundary layers in magnetohydrodynamics have received from the mathematical community numerous investigations during the past decades. To the best of our knowledge, however, the derivation of system \eqref{eq:main-system} is new, hence there has not been related analytical results, up to now. The majority of the results concerns classical MHD-equations and the underlying boundary-layer theory, in which the displacement current is indeed neglected.  
In this paragraph we shall hence focus on the various contributions that dealt with certain equations that shear at least similarities with system \eqref{eq:main-system}. 


\noindent 
One of the system that has majorly received attention was provided by G\'erard-Varet and Prestipino in \cite{MR3657241}. Omitting the notation of the several dimensionless parameters, the equations read as follows:
\begin{equation}
	\label{eq:MHD-Prandtl-literature}
	\left\{\;
	\begin{aligned}
		& \pa_tu+u\pa_x u+v\pa_yu -\pa_y^2u+\pa_xp= b_1 \partial_x b_1 + b_2 \partial_y b_1,\\
		&\partial_y p = 0,\\
		&\partial_t b_1 + u \partial_x b_1 + v \partial_y b_1 - \partial_{yy}^2 b_1= b_1 \partial_x u + b_2 \partial_y u, \\
		&\partial_t b_2 + u \partial_x b_2 + v \partial_y b_2 - \partial_{yy}^2 b_2= b_1 \partial_x v + b_2 \partial_y v, \\
		&\partial_x u + \partial_y v = 0,\quad \partial_x b_1 +\partial_y b_2 = 0.
	\end{aligned}
	\right.
\end{equation}
The authors derived this system as boundary asymptotic of the classical MHD equations, under a stringent regime of the coupling parameters (which we also assume in this paper): the Hartmann number ${\rm Ha}$, the Reynolds number  ${\rm Re}$ and the magnetic Reynolds number  ${\rm Re_m}$ were all  proportional and assumed high values. 

\noindent 
The major differences of systems \eqref{eq:main-system} and \eqref{eq:MHD-Prandtl-literature}   can be recognised on the forcing term of the momentum equation in $u$,
a non-constant pressure in the vertical variable due to $\partial_y p \neq 0$ and the second time derivatives $\partial_{tt}^2u$, $\partial_{tt}^2b_1$ and $\partial_{tt}^2b_2$ on the equations for the magnetic field (which are due to the displacement current, cf.~also Remark \ref{rem:differences-between-MHD-Prandtl-and-ours}). 

\noindent
System \eqref{eq:MHD-Prandtl-literature} retains most terms of the original MHD equations and it reduces to the widespread Prandtl equations for purely hydrodynamic flows, when the magnetic field $(b_1,b_2)$ is null. Among the mathematical community, there has been hence an increasing interest to  transfer well-known analytical results of the Prandtl theory to the corresponding Prandtl-MHD equations \cite{MR3657241,MR3864769,MR4213671,MR4328431}:  existence and uniqueness of solutions, regularity analysis (Sobolev, analytic, Gevrey), stability of certain equilibria (such as shear flows).

\noindent 
At a first glance, one may think that system \eqref{eq:MHD-Prandtl-literature} 
shall satisfy reduced (or at least similar) properties than classical Prandtl. 
However, in \cite{MR3657241}, G\'erard-Varet and Prestipino  overturned this statement, when dealing with the stability of \eqref{eq:MHD-Prandtl-literature} around certain equilibria. The authors showed indeed how system \eqref{eq:MHD-Prandtl-literature} is stable to a suitable family of shear flows in which both plasma's velocity and magnetic field are parallel to the bounding flat surface: 
$$(u(t,x,y),v(t,x,y),b_1(t,x,y),b_2(t,x,y)) = (U_s(y), 0, 1, 0),$$ 
This stability holds already at the level of Sobolev regularities, a fact that is in sharp contrast with the Sobolev instability of the classical Prandtl equations (cf.~\cite{MR2952715,MR2601044}). Hence, in terms of Sobolev stability, Prandtl-MHD has enhanced properties than only Prandtl. To better understand this unusual characteristic, one shall first recall that the major difficulties of the Prandtl equations reside in the convective term $v\partial_y u$ (this vertical component $v$ has a lower regularity in the tangential variable than $u$ and is determined by the divergence free condition $\partial_x u +\partial_y v = 0$). However, when dealing with the Prandtl-MHD equations \eqref{eq:MHD-Prandtl-literature}, one can get rid of this 
``bad term'', by introducing a new modified velocity field $\tilde u = u + U'_s \phi$, where $\phi$ stands for the potential generating the magnetic field (i.e.~ $b_1 =\partial_y \phi$ and $b_2 = - \partial_x \phi$). This mathematical artifact has clarified certain observations in physics,  in particular the fact that a magnetic field has a stabilising effect on the flow of plasma and provides therefore a mechanism of containment.


\noindent 
Away from shear flows, Liu, Xie and Tong proposed in \cite{MR3975147} a generalisation of this Sobolev stability, when dealing with the full nonlinear version of equations \eqref{eq:MHD-Prandtl-literature}. The authors showed indeed that a modified velocity similar to $\tilde u$ could still be defined, as long as the tangential component $b_1$ of the magnetic field never vanishes (a condition known as 
``nondegeneracy of $b_1$''). This result was local in time and was extended globally by Liu and Zhang in \cite{MR4213671}, under a smallness condition on the initial data. 

\noindent 
After \cite{MR3975147}, a remaining open problem was to understand if the nondegeneracy of $b_1$ was somehow necessary in order to recover the mentioned Sobolev stability. The same authors Liu, Xie and Tong in \cite{MR3864769} provided a surprisingly positive answer to this dilemma: when linearising equations \eqref{eq:MHD-Prandtl-literature} around a family of shear flows of the form
$$(u(t,x,y),v(t,x,y),b_1(t,x,y),b_2(t,x,y)) = (U_s(t,y), 0, B_s(t,y), 0),$$ 
in which $B_s$ vanishes (together with some of its derivatives), it was shown that the corresponding system is indeed ill-posed in Sobolev spaces. 

\noindent 
This result in \cite{MR3864769} opened a further variety of questions, which regarded in particular the following aspects:
\begin{itemize}
    \item If the asymptotic limit ($\rm Re,\,Re_m\to \infty$) of the Prandtl-MHD equations \eqref{eq:MHD-Prandtl-literature} is well posed in Sobolev spaces, does a formal mathematical expansion reveal the corresponding boundary layers within the solutions of the original MHD equations 
    (for high values of $\rm Re \gg 1$ and $\rm Re_m \gg 1$)?
    \item If the nondegeneracy condition $b_1 \neq 0$ is necessary for the Sobolev stability, can one consider higher regularities (such as Gevrey), in order to relax this constraint?
\end{itemize}
Liu, Xie and Yang in \cite{MR3975147} provided a positive answer to the first question, as long as the tangential magnetic component $b_1 \geq \delta >0$ remains strictly positive. With some additional technical condition on the initial data, the authors showed that the differences between the smooth solutions of the original MHD equations and the boundary layers (which depend on $\rm Re>0$ and $\rm Re_m>0$) converge to the smooth solution of the limit case ${\rm Re},\, {\rm Re}_m\to \infty$. The convergence is indeed uniform in $L^\infty$ (both in time and space) because of the Sobolev stability.

\noindent 
For what concerns higher regularities, Li and Yang addressed in \cite{MR4270479} solutions with Gevrey regularities in ``$x$'' of order $3/2$ (a function space between analytic and Sobolev functions). Gevrey functions have significantly influenced and impacted the analysis of the Prandtl equations, since they still allow  to cope with smooth test functions (in contrast with analytic regularity). Without any assumption on the tangent component $b_1$ of the magnetic field, Li and Yang showed in \cite{MR4270479} that initial data in Gevrey $3/2$ generate local-in-time smooth solutions of \eqref{eq:MHD-Prandtl-literature}. It remains still an interesting open problem to establish if this result is optimal or if one can further enlarge the regularity (for instance towards Gevrey $2$, the optimal value of Prandtl, cf.~\cite{MR3925144}).



\subsection{Statement of our modelling results}\label{intro:modelling}$\,$

\noindent
The well-posedness results of classical   Navier-Stokes-Maxwell  equations without Cattaneo’s law can be found in \cite{Masmoudi2D,Masmoudi3D,Gallagher2D3D,Ogawa}. 
The approach of our modelling is first to introduce a suitable form of the
Navier-Stokes-Maxwell equations with Cattaneo’s law (cf.~System \eqref{eq:Maxwell+Navier-Stokes} and its dimensionless form \eqref{eq:Maxwell+Navier-Stokes2-dimensionless}) and secondly to examine the asymptotic of some related 
dimensionless parameters (that are indeed popular in MHD). 
More precisely, we derive System \eqref{eq:Maxwell+Navier-Stokes} for high values of 
the Reynolds number ${\rm Re}$, the magnetic Reynolds numbers ${\rm Re_m}$ and the Hartmann number ${\rm Ha}$ (for details, cf.~Sections \ref{sec:navier-stokes-maxwell-dimensionless}-\ref{sec:Hartmann}, as well as Theorem \ref{thm:main-modelling} and Theorem \ref{thm:main-modelling2}, below). 

\noindent
The nature of the derived equations depends on certain hypotheses on the characteristic speed $U_0>0$ of Plasma (we refer to \eqref{eq:Maxwell+Navier-Stokes2-dimensionless} for the explicit relation between the starting equations and $U_0$). We establish indeed boundary layers with thickness of two types: Prandtl or Hartmann.

\noindent 
To clarify our result, we shall first recall that in magnetohydrodynamics several types of boundary layers can occur, depending on the angle of orientation of the magnetic field boundary. Among the most relevant, Prandtl and Hartmann layers stand out, since they also differ on their thickness. Conventional Prandtl layers are indeed purely hydrodynamic and are characterised by a length, which is proportional to $1/\sqrt{\rm Re}$. On the other hand, Hartmann layers are mainly attributable to the magnetic field, having a thickness inversely proportional to the Hartmann number {\rm Ha}. Depending on the magnitude of an imposed magnetic field, Hartmann layers may be as thin as one desires, thus the velocity field of plasma  usually increases much more rapidly over a short distance from the boundary.


\noindent
Our first result in Theorem \ref{thm:main-modelling} shows that, when the characteristic speed of Plasma $U_0$ is proportional to the Reynolds number ${\rm Re}$, then 
System \eqref{eq:main-system} stands for a boundary layer within a region of thickness $1/\sqrt{Re}$ (Prandtl).
\begin{theorem}\label{thm:main-modelling}
    Consider the Navier-Stokes-Maxwell equations with Cattaneo's law in \eqref{eq:Maxwell+Navier-Stokes2-dimensionless}.
    Assume that the following relations between the dimensionless parameters in system \eqref{eq:Maxwell+Navier-Stokes2-dimensionless} are satisfied:
    \begin{equation*}
        \lim_{{\rm Re}\to +\infty} \frac{{\rm Ha}}{{\rm Re}}= \mathbb{H}\in \mathbb{R},
        \quad  
        {\rm Pr_m}:= \frac{{\rm Re_m}}{{\rm Re}}\text{ is fixed,}
        \quad \lim_{{\rm Re}\to +\infty} 
        \left( \frac{U_0}{c}\right)^2 \frac{1}{{\rm Re}} = \kappa \in \mathbb{R},
        \quad 
         \mathbb J := \kappa \mathcal{J}\text{ is fixed}.
    \end{equation*}
    Furthermore, assume that the initial data $ \mathbf{B}' = (\mathbf{B}_1',\mathbf{B}_2')^T\in \mathbb R^2$ and $\mathbf{E}'\in \mathbb R$ are such that 
    \begin{equation*}
       \mathbf{B}_1'= \mathbf{b}_1\in \mathbb{R} \text{ is fixed, while}\quad 
       \mathbf{B}_2' = \frac{1}{\sqrt{{\rm Re}}} \mathbf{b}_2,\quad 
       \mathbf{E}' = \frac{1}{\sqrt{{\rm Re}}} \mathbf{e},
    \end{equation*}
    for some $ \mathbf{b}_2\in \mathbb{R}$ and $ \mathbf{e}\in \mathbb{R}$.
    Then system \eqref{eq:main-system} appears as boundary layer of equations  
    \eqref{eq:Maxwell+Navier-Stokes2-dimensionless} 
    in the region
    \begin{equation*}
        (t',x',y') \in (0,T) \times \mathbb{R}\times  \left(0,\frac{1}{\sqrt{\rm Re}}\right),
    \end{equation*}
    when ${\rm Re}\to +\infty$ (and thus also when ${\rm Re}$, ${\rm Ha}\to +\infty$).
\end{theorem}

\noindent 
Since $(t,x,y)$ are the variable of the boundary layer in \eqref{eq:main-system}, we clarify that $(t',x',y')$ are now the variables of the starting Navier-Stokes-Maxwell equations \eqref{eq:Maxwell+Navier-Stokes2-dimensionless}.

\noindent
The relation $U_0^2/c^2{\rm Re}\to \kappa>0$ is of course questionable, 
since it would imply that the characteristic speed $U_0$ assumes values that are much higher than the speed of light $``c"$. In case $U_0/c$ is fixed (we treat this case in Theorem \ref{thm:main-modelling2}), the contribution of the displacement current would indeed vanish,
thus an ansatz typical of the Prandtl theory would probably lead to the standard Prandtl-MHD equations \eqref{eq:main-system} (as derived in \cite{MR3657241}) rather than system \eqref{eq:main-system}. 

\noindent
Theorem \ref{thm:main-modelling} would seem therefore to suggest that when the characteristic speed of Plasma is constant, System \eqref{eq:main-system} is ineffective at the limit ${\rm Re}\to \infty$. 
With the next Theorem \ref{thm:main-modelling2}, we counteract this statement, by showing that System \eqref{eq:main-system} remains an accurate boundary layer, when considering a different scaling (and thus a different region of the layer).
\begin{theorem}\label{thm:main-modelling2}
    Consider the Navier-Stokes-Maxwell equations with Cattaneo's law in \eqref{eq:Maxwell+Navier-Stokes2-dimensionless}.
        Assume that the following relations between the dimensionless parameters in system \eqref{eq:Maxwell+Navier-Stokes2-dimensionless} are satisfied:
    \begin{equation*} 
        \lim_{{\rm Re}\to +\infty} \frac{{\rm Ha}}{{\rm Re}}= \mathbb{H}\in \mathbb{R} 
        \quad  
        \text{and}
        \quad
        {\rm Pr_m}:=\frac{{\rm Re_m}}{{\rm Re}},
        \;  \kappa:= \left( \frac{U_0}{c}\right)^2,
        \;
         \mathbb J := \kappa \mathcal{J}\text{ are all fixed}.
    \end{equation*}
    Furthermore, assume that the initial data $ \mathbf{B}' = (\mathbf{B}_1',\mathbf{B}_2')^T$ are such that 
    \begin{equation*}
       \mathbf{B}_2'= \mathbf{b}_2\in\mathbb R  \text{ is fixed, while} 
       \quad  
       \mathbf{B}_1' = \sqrt{\frac{\rm Ha}{\mathbb H}} \mathbf{b}_1,
       \quad 
       \mathbf{E}' =  \sqrt{\frac{\rm Ha}{\mathbb H}} \mathbf{e},
    \end{equation*}
    for some $ \mathbf{b}_1\in \mathbb{R}$ and $ \mathbf{e}\in \mathbb{R}$. Then system \eqref{eq:main-system} appears as boundary layer in the region
    \begin{equation*}
        (t',x',y') \in \left(0, \frac{\mathbb H}{\rm Ha}T'\right) \times \mathbb{R}\times  \left(0,\frac{\mathbb H}{\rm Ha}
        \right),
    \end{equation*}
    when ${\rm Re}\to +\infty$ (and thus also when ${\rm Re}$, ${\rm Ha}\to +\infty$).
\end{theorem}
\noindent 

\begin{rem}
The domain of the boundary layer in Theorem \ref{thm:main-modelling2} is not only close to $y' = 0$ but also to the time origin $t' = 0$. This is not surprising from a mathematical point of view. 
Indeed, at the asymptotic limit ${\rm Re}\to \infty$, the Navier-Stokes-Maxwell's equations with Cattaneo's law \eqref{eq:Maxwell+Navier-Stokes2-dimensionless} switch from hyperbolic to parabolic in the variables $(t,y)$. This leads to a loss of initial data on the time derivative of the velocity field and the magnetic field, thus the appearance of boundary layers near the origin in time.

\noindent 
From a physical point of view, one would wonder if this domain is rather an artifact of the equations, since such a short time would not be observed in real applications. We counter this statement through the following remark: although $ t'\in (0,\delta T)$ reflects a short range of time, the solution $(u,v,b_1,b_2,e)$ of \eqref{eq:main-system} contributes 
to dynamics of Plasma $(U_1', U_2', B_1', B_2', E')$ with a rescaled magnitude of order $1/\delta$ (for more details cf.~transformation \eqref{new-functions-Hartmann}). This property would suggest that the solution $(u,v,b_1,b_2,e)$ might still impact  the evolution of Plasma near bounding surface for larger time $t'>\delta T$. The formal proof of this statement is however above the interest of this paper. 
\end{rem}

\noindent 
We conclude this introduction with a short overview of the sections concerning our modelling. In Section \ref{sec:navier-stokes-maxwell} and Section \ref{sec:navier-stokes-maxwell-dimensionless} we introduce the Navier-Stokes-Maxwell equations with Cattaneo's law on a suitable dimensionless form. Next, in Section \ref{sec:modelling-Prandtl}, we prove Theorem \ref{thm:main-modelling} and derive System \eqref{eq:main-system} as a boundary layer with thickness of Prandtl type. Finally, we address the Hartmann origin of System \eqref{eq:main-system}, by proving Theorem \ref{thm:main-modelling2} in Section \ref{sec:Hartmann}.

\subsection{Statement of our analytic results}\label{intro:analysis}$\,$

\noindent
Once concluded the modelling of Theorem \ref{thm:main-modelling} and Theorem \ref{thm:main-modelling2}, we pass to investigate the well-posedness problem of the derived system. Since the considered model is an extension of the standard Prandtl equations, it presents similar analytical challenges, in particular the lack of diffusion (and thus of regularising effects) on the variable $x\in \mathbb R$. The major nonlinearities can indeed generate strong instabilities in the horizontal direction, specifically under the occurrence of high oscillations of the solutions (the contribution of the high frequencies in $x\in \mathbb R$).  It is rather common in the scientific community to address  the analysis of boundary layers by considering therefore highly regular initial data, such as analytic in $x\in \mathbb R$ \cite{MR4362378,MR4271962,MR4125518}. We postpone the precise definition of this functional framework to Section \ref{sec:analysis}, we shall however mention that these are functions whose frequencies $\xi\in \mathbb R$ under Fourier transform decay like $\exp(-\tau |\xi|)$, for some $\tau >0$ known as radius of analyticity. Our analytical result asserts that if the initial data are indeed analytic and are sufficiently small, then there exists a global-in-time analytic solution of equations \eqref{eq:main-system}, whose radius of analyticity decays exponentially in time.
\begin{theorem} \label{intro-thm:existence-analytic-solutions}
Assume homogeneous boundary conditions in \eqref{boundary-condition-our-main-eq}: $(\mathbf b_1,\mathbf b_2, \mathbf e)=(0,0,0)$.
For any $s>2$, there exists a sufficiently small positive constant $\varepsilon_s\in [0,1)$ (which depends uniquely upon $s$),  such that the following result holds true. 
Let $\bar u$, $\bar b_{1}$ and $\tilde b_1$ be initial data of \eqref{eq:main-system}
that are  analytic in the variable $x\in \mathbb{R}$ with radius of analyticity $\tau_0>0$:
\begin{equation}\label{intro-Sobolev-spaces}
\begin{alignedat}{4}
    e^{\tau_0 (1+|D_x|)} \bar u     
    \quad\text{and}\quad
    e^{\tau_0 (1+|D_x|)} \bar b_1    
    \quad 
    &\text{belong to}
    \quad 
    H^{s+1,1}(\mathbb R \times (0,1)),\\
    e^{\tau_0 (1+|D_x|)} \tilde u   
    \quad\text{and}\quad
    e^{\tau_0 (1+|D_x|)} \tilde b_1  
    \quad 
    &\text{belong to}
    \quad 
    H^{s,0}(\mathbb R \times (0,1)).\\
\end{alignedat}
\end{equation} 
If the following smallness condition on the initial data holds true
\begin{equation}\label{intro-thm:small-condition}
\begin{aligned}
        \|  
            &
            e^{\tau_0(1+|D_x|)} \bar u       
        \|_{H^{s+1,0}}
        +
        \|  
            e^{\tau_0(1+|D_x|)} 
            \partial_y \bar u       
        \|_{H^{s,0}}
        +
        \|  
            e^{\tau_0(1+|D_x|)} \tilde u       
        \|_{H^{s,0}}
        +
        \|  
            e^{\tau_0(1+|D_x|)}  
            \bar b_1     
        \|_{H^{s+1,0}}
        +\\
        &+\!
        \|  
            e^{\tau_0(1+|D_x|)} \partial_y \bar b_1     
        \|_{H^{s,0}}
       \!+\!
        \|  e^{\tau_0(1+|D_x|)}  \tilde b_1   
        \|_{H^{s,0}}
        \leq 
        \bigg(
        \frac{\min \{  1, \mathbb J, \kappa/{\rm Pr}_m\}^\frac{3}{2}}
        {\max \{  1, \mathbb J, \kappa/{\rm Pr}_m\}^\frac{5}{2}}
        \frac{
        \min\{\tau_0,\tau_0^{-1}\}\frac{3}{2}}{\max\{1, \mathbb{H}^2\}\max\{{\rm Pr_m}^{-1}, {\rm Pr}_m\}^\frac{1}{2}}
        \bigg)
        \varepsilon_s,
\end{aligned}
\end{equation}
then there exists a global-in-time analytic solution $(u, b_1)$ of \eqref{eq:main-system}, which has a decaying radius of analyticity $\tau:\mathbb R_+ \to (0,\tau_0]$ given by
\begin{equation}\label{intro-radius-of-analyticity-thm:existence-of-solutions}
\begin{aligned}
    \tau(t)  := 
        \tau_0
            \exp\Big\{ 
                -\frac{t}{16 \max\{1, \mathbb J, \kappa/{\rm Pr}_m\} }
        \Big\}>0.
\end{aligned}
\end{equation}
Furthermore,  the analytic norms of the solution decay exponentially in time $t\in \mathbb R_+$ as follows:
\begin{equation}\label{intro-norms-solution}
\begin{aligned}
    \| &e^{\tau(t)(1+|D_x|)} u(t)             \|_{H^{s+1,0}}^2 + 
    \| e^{\tau(t)(1+|D_x|)} \partial_t u(t)   \|_{H^{s,0}}^2+ 
    \| e^{\tau(t)(1+|D_x|)} \partial_y u(t)   \|_{H^{s,0}}^2 +
    \\
    &+
    \| e^{\tau(t)(1+|D_x|)} b_1(t)              \|_{H^{s+1,0}}^2 + 
    \| e^{\tau(t)(1+|D_x|)} \partial_t b_1(t)   \|_{H^{s,0}}^2+ 
    \| e^{\tau(t)(1+|D_x|)} \partial_y b_1(t)   \|_{H^{s,0}}^2 
    \\
    &
    \leq 
    C({\mathbb J, \kappa, {\rm Pr_m},\tau_0})
   \bigg\{
    \| e^{\tau_0(1+|D_x|)} \bar u               \|_{H^{s+1,0}}^2 + 
    \| e^{\tau_0(1+|D_x|)}  \tilde  u           \|_{H^{s,0}}^2+ 
    \| e^{\tau_0(1+|D_x|)} \partial_y\bar  u    \|_{H^{s,0}}^2 +
    \\
    &+
    \| e^{\tau_0(1+|D_x|)} \bar b               \|_{H^{s+1,0}}^2 + 
    \| e^{\tau_0(1+|D_x|)}  \tilde  b           \|_{H^{s,0}}^2+ 
    \| e^{\tau_0(1+|D_x|)} \partial_y\bar  b    \|_{H^{s,0}}^2 
    \bigg\}
    \exp\bigg\{-\frac{t}{8\max\{1, \mathbb J, \kappa/{\rm Pr_m}\}}
    \bigg\}
    .
\end{aligned}    
\end{equation}
where the constant $ C({\mathbb J, \kappa, {\rm Pr_m},\tau_0})$ is defined by
\begin{equation*}
    C({\mathbb J, \kappa, {\rm Pr_m},\tau_0})
    =
    4^3
    \frac{
        \max\{1,\mathbb J, \kappa/{\rm Pr_m}\}^3
    }
    {
        \min\{1,\mathbb J, \kappa/{\rm Pr_m}\}^3
    }
    \max\big\{{\rm Pr_m}, {\rm Pr_m^{-1}}\big\}
    \max \{\tau_0, \tau_0^{-1} \}^2.
\end{equation*}
\end{theorem}
\noindent 
Some remarks are here in order. The statement consider uniquely the state variables $(u,b_1)$. Indeed all others variables $(v,b_2, e)$ are  determined by the  divergence-free conditions, the Faraday's law and the homogeneous boundary conditions  :
\begin{equation*}
    v(t,x,y) = -\int_0^y \partial_x u(t,x,z)dz,\quad 
    b_2(t,x,y) = -\int_0^y \partial_x b_1(t,x,z)dz\quad 
    e(t,x,y) =  -\int_0^y \partial_t b_1(t,x,z)dz.
\end{equation*}
We refer to Section \ref{sec:reduced-system} for more details, and to \eqref{eq:MHD-Prandtl-without-e} for an explicit form of the equations in terms of $(u,b_1)$.

 \noindent
 The function space described by \eqref{intro-Sobolev-spaces} is analytic in the variable $x\in \mathbb R$, since $e^{\tau_0(1+|D_x|)}$  is a Fourier multiplier  that enforces an exponential decay on the frequencies of the initial data (more details in Section \ref{sec:analytic-fcts}). The range $s>2$ of the related norms is however a pure artifact of our analysis rather than a real restriction. This condition simplifies indeed certain estimates (cf.~for instance Lemma \ref{lemma:technical-lemma-product-law-Hs0} together with \eqref{est1-where-we-need-s>2} where $\sigma_2-1/2 =s-2>0$). Nevertheless we could also consider $s\in [0,2)$, since $e^{\tau_0(1+|D_x|)} \bar u$ (and similarly all the other initial data) can always be recasted as $e^{-\epsilon (1+|D_x|)} e^{(\tau_0+\epsilon)(1+|D_x|)} \bar u $, where the Fourier multiplier $e^{-\epsilon (1+|D_x|)}$ is a regularising operator (of course we would need a slightly higher radius of analiticity). We do not pursue this approach just for the sake of a short presentation.

\noindent
The small parameter $\varepsilon_s$ in \eqref{intro-thm:small-condition} depends uniquely on $s>2$ and can be explicitly defined as (cf.~\eqref{def-esp_s-T*})
\begin{equation*}
     \varepsilon_s := 
      \frac{s-2}{2^{2s+14}}
      \frac{1}{1+\frac{s-2}{\sqrt{s-1}}}
     .
\end{equation*}
Moreover the right hand-side of \eqref{intro-thm:small-condition} and the size of the initial data decrease proportionally to $\tau_0^{3/2}$ (when the radius of analyticity $\tau_0<1$). This aspect is revealed and supported by the function framework in \eqref{intro-Sobolev-spaces}, which converges to standard Sobolev spaces when $\tau_0$ vanishes (our model may still be ill-posed in Sobolev, as Prandtl).

\smallskip
\noindent 
Let us now comment on the relation between the smallness condition \eqref{intro-thm:small-condition} and the different physical parameters of system \eqref{eq:main-system}. When $\mathbb H\gg 1$ the nonlinearities at the right-hand side of the $u$-equation becomes predominant, hence a more restrictive condition on the initial data is natural in order to achieve analytic stability. 
Moreover, also the constants $\mathbb J$  and $\kappa /{\rm Pr}_m$ play a major role in \eqref{intro-thm:small-condition}, since they inherently decrease the size of the initial data, when they converge towards $0$. At a first glance, this property would seem questionable, since we would expect that when these constants vanish (i.e.~both the inertial term of the velocity field and the displacement current are neglected) we should recover similar equations to Prandtl-MHD in \eqref{eq:MHD-Prandtl-literature}. Consequently, the equations would switch from hyperbolic on the $y$-direction (with damping mechanisms) to parabolic, a setting which is usually more stable. This observation is however inaccurate, since our model \eqref{eq:main-system} owns a more involved structure than Prandtl-MHD \eqref{eq:MHD-Prandtl-literature}, which can be  highlighted in the following aspects:
\begin{itemize}
       \item The right-hand side of the momentum equation (first equation in \eqref{eq:main-system}) has terms that are trilinear in the solution (contrary to the bilinear ones in \eqref{eq:MHD-Prandtl-literature}). These terms further increase the instabilities of our model, when compared with the ones of boundary layers in MHD.
       \item The contribution of the pressure is not trivial as in Prandtl-MHD, since $\partial_y p$ is not identically null and it encompasses further trilinear terms. These are indeed the more challenging terms to estimate.  
\end{itemize}
Our analysis and our smallness condition \eqref{intro-thm:small-condition} therefore suggest that the contributions $\mathbb J $, $\kappa/{\rm Pr_m}$ of the displacement current and the Cattaneo's law  have not only a role as derivation of our model, but they may rather have a stabilising effect on the underlying solutions (at least at the level of analytic regularities).

\smallskip 
\noindent
The part of this paper concerning the analysis of system \eqref{eq:main-system} is organized as follows. In Section~\ref{sec:analytic-fcts} we define the function spaces of the analytic solutions, while in Section~\ref{sec:reduced-system} we provide a compact formulation of System \eqref{eq:main-system}. Section \ref{sec:proof-of-main-thm} is devoted to the sketch of the proof of Theorem \ref{intro-thm:existence-analytic-solutions}, postponing the more technical parts. Our approach is indeed based on a suitable estimate of the norms of the solutions (cf.~Proposition \ref{prop:the-overall-final-estimate-in-eta}, whose proof is formally developed in Section \ref{sec:proof-proposition-est}).

\section{Derivation of the model}\label{sec:modelling}

\noindent 
This section is devoted to prove Theorem \ref{thm:main-modelling} and Theorem \ref{thm:main-modelling2}. To this end, in Section \ref{sec:navier-stokes-maxwell} we introduce a suitable form of the Navier-Stokes-Maxwell's equations with Cattaneo's law, that we recast in their dimensionless form in Section \ref{sec:navier-stokes-maxwell-dimensionless}. Section \ref{sec:modelling-Prandtl} is hence devoted to prove Theorem \ref{thm:main-modelling}, while in Section \ref{sec:Hartmann} we deal with  Theorem \ref{thm:main-modelling2}.


\subsection{The Navier-Stokes-Maxwell equations with Cattaneo's law}\label{sec:navier-stokes-maxwell}$\,$

\noindent
We begin with by recalling the widespread form of the two-dimensional Navier-Stokes-Maxwell system with Cattaneo's law:
\begin{equation}\label{eq:Maxwell+Navier-Stokes}
    \begin{cases}
        \frac{\rho \nu \mathcal{J}}{c^2}
        \partial_{\tau }^2 \overrightarrow U
        +
       \rho 
       \big(
            \partial_\tau \overrightarrow  U + \overrightarrow U \cdot \nabla \overrightarrow  U 
       \big)- \rho  \nu  \Delta \overrightarrow  U + \nabla P = \overrightarrow  J\times  \overrightarrow  B
       \qquad 
       &\text{balance of linear momentum},
       \\
        \dv \overrightarrow  U = 0\qquad 
       &\text{conservation of mass},\\
        \partial_\tau\overrightarrow B +{\rm curl}\,  \overrightarrow  E = 0
        \qquad 
       &\text{Faraday's law,}\\
        \overrightarrow J = \sigma \big(\overrightarrow E + \overrightarrow U \times \overrightarrow B \big)
         \qquad 
       &\text{Ohm's law,}\\
        \frac{1}{c^2} 
        \partial_\tau \overrightarrow  E + \mu_0 \overrightarrow  J = {\rm curl}\overrightarrow  B 
                \qquad 
       &\text{Ampere's law,}\\
        {\rm div}\, \overrightarrow  B = 0
         \qquad 
       &\text{Gauss's law for magnetism,}\\
        \dv \overrightarrow  E = 0
         \qquad 
       &\text{Gauss's law for electric field.}
    \end{cases}
\end{equation}
The system and the corresponding state variables depends upon $(\tau,X,Y)\in (0,T)\times \mathbb{R}\times \mathbb{R}_+$ (instead of $(t,x,y)$, which are the variables of the boundary layers), for a positive time $T>0$. The following boundary conditions are also prescribed:
\begin{equation*}
    \overrightarrow U(\tau,X,0) = 0\in \mathbb{R}^2, 
    \quad 
    \overrightarrow B(\tau,X,0) = 
    \overrightarrow{\mathbf{B}}\in \mathbb{R}^2
    \quad 
    \text{and}
    \quad 
    E(\tau,X,0) = 
    \mathbf{E}\in \mathbb{R},
\end{equation*}
i.e.~the velocity field satisfies a no-slip boundary condition, whereas the surrounding medium in $(\tau,X,Y)\in (0,T)\times \mathbb{R}\times \{y<0\}$ is an insulator with a prescribed fixed magnetic field $ \overrightarrow{\mathbf{B}}\in \mathbb{R}^2$. 

\noindent
The constants $\rho>0$ and $\nu>0$ are the density of the fluid and the kinematic viscosity, respectively, while $c>0$ stands for the speed of light. The first term $(\rho \nu \mathcal{J}/c^2)  \partial_{\tau}^2 \overrightarrow U$ in the balance of linear momentum is due to the Cattaneo's law \cite{MR3942552,MR2045417,Coulaud,MR4429384,MR2404054,MR3085226} and depends on a general inertial constant $\mathcal J >0$. 
This law develops around a first-order Taylor expansion of a delayed relation on the Cauchy stress tensor
\begin{equation*}
    \mathbb{S}(\tau+ \tau_{\rm rel}, \cdot ) = \nu \frac{\nabla u + \nabla u^T}{2}(\tau,\cdot) 
\end{equation*}
where for us the relaxation time is given by $\tau_{\rm rel} =\rho \nu \mathcal{J}/c^2$. This particular form (in terms of $\mathcal J>0$ and not directly in $\tau_{\rm rel}>0$), will be  important indeed when rescaling our system for the boundary layers.

\noindent
We have denoted by $\sigma>0$ the electrical conductivity, by $\mu_0>0$ the magnetic permeability. 
We have further denoted $\overrightarrow U(\tau,X,Y) = (U_1(\tau,X,Y),U_2(\tau,X,Y))^T \in \mathbb{R}^2$ and 
$\overrightarrow B(\tau,X,Y) = (B_1(\tau,X,Y),B_2(\tau,X,Y))^T \in \mathbb{R}^2$ the velocity field and magnetic field of the media, respectively. The scalar pressure $P(\tau,X,Y)  \in \mathbb{R}$ is the Lagrangian multiplier that ensures the incompressibility of the velocity field. The current density $\overrightarrow  J = (0,0,J(\tau,X,Y))^T$ and the electric field $ \overrightarrow  E = (0,0,E(\tau,X,Y))^T$ are considered as three dimensional vector fields, being perpendicular to the plane in which the fluid motion occurs. Since we are dealing with the two dimensional version of the equations, we shall clarify the employed notation: 
\begin{equation*}
    \overrightarrow  J\times  \overrightarrow  B 
    = J  \overrightarrow   B^{\perp} \!\!= 
    J 
    \begin{pmatrix}
        -B_2\\ \hspace{0.3cm}B_1
    \end{pmatrix}\!\!,
     \;
    \curl  \overrightarrow  E = 
    \begin{pmatrix}
        \hspace{0.25cm}\partial_Y E\\ -\partial_X E
    \end{pmatrix}\!\!,
     \;
     \overrightarrow  U\times  \overrightarrow  B 
     =
     \begin{pmatrix}
        0\\0\\ U_1B_2-U_2B_1
    \end{pmatrix}\!\!,
     \;
    \curl  \overrightarrow  B 
     =
     \begin{pmatrix}
        0\\0\\ \partial_XB_2-\partial_Y B_1
    \end{pmatrix}\!\!.
\end{equation*}
The positive parameters $\nu$, $\mu_0$ and $\varepsilon_0$ correspond to the kinematic viscosity, the magnetic permeability and permittivity of free space, respectively. Furthermore, the parameter $\sigma$ represents the electrical conductivity of the medium. 

\smallskip
\noindent 
Some of the terms in \eqref{eq:Maxwell+Navier-Stokes} are redundant, indeed we can recast the overall system as five equations depending on $ \overrightarrow  U$, $ \overrightarrow  B$ and $ \overrightarrow  E$. First, we formulate the Faraday's law in \eqref{eq:Maxwell+Navier-Stokes} only in terms of the magnetic field $\overrightarrow B$, making use of the Ohm's law \eqref{eq:Maxwell+Navier-Stokes}: 
\begin{equation}\label{eq:first-step-for-erasing-E}
        \partial_\tau \overrightarrow B = - \curl \overrightarrow  E  
        =
        \curl \big(\overrightarrow U\times \overrightarrow B \big) - \frac{1}{\sigma}  \curl \overrightarrow J.
\end{equation}
Furthermore, to get rid of the current density in $\curl \overrightarrow J$, we apply the $\curl$ operator to the Ampere's law:
\begin{equation*}
    \frac{1}{c^2}\partial_\tau (\curl  \overrightarrow  E  ) + \mu_0 \curl  \overrightarrow  J = \curl\curl B,
\end{equation*}
which leads to 
\begin{equation*}
    \curl  \overrightarrow  J = \frac{1}{\mu_0c^2}\partial_{\tau}^2  \overrightarrow  B + 
    \frac{1}{\mu_0}(\nabla \dv  \overrightarrow  B - \Delta B ) = 
    \frac{1}{\mu_0c^2}\partial_{\tau}^2  \overrightarrow  B + \frac{1}{\mu_0}(\nabla \dv  \overrightarrow  B - \Delta B )
\end{equation*}
Thus, we can plug this last relation in equation \eqref{eq:first-step-for-erasing-E}, to finally obtain an hyperbolic form of the Ampere's law in terms of the magnetic field $ \overrightarrow B$:
\begin{equation}\label{eq:Faraday_without_E_J}
    \frac{1}{\sigma\mu_0 c^2}   
    \partial_{\tau}^2  \overrightarrow B + 
    \partial_\tau \overrightarrow B -
    \frac{1}{\sigma \mu_0}\Delta  \overrightarrow B = \curl \big(\overrightarrow U\times \overrightarrow B \big) = 
    \overrightarrow B\cdot \nabla \overrightarrow U -
    \overrightarrow U\cdot \nabla \overrightarrow B.
\end{equation}
Similarly, we can get rid of $\overrightarrow J$ also in the balance of linear momentum in \eqref{eq:Maxwell+Navier-Stokes} through
\begin{equation*}
    \frac{\rho \nu \mathcal{J}}{c^2} \partial_{\tau}^2 \overrightarrow U
    \! +\! 
    \rho (\partial_t \overrightarrow  U + \overrightarrow U \cdot \nabla \overrightarrow  U ) - \nu \Delta \overrightarrow  U + \nabla P = \sigma (\overrightarrow E + \overrightarrow U \times \overrightarrow B )\times  \overrightarrow  B 
    = \sigma \overrightarrow E  \times  \overrightarrow  B+ \sigma (\overrightarrow B(\overrightarrow U\cdot \overrightarrow B) - \overrightarrow U | \overrightarrow B|^2 ).
\end{equation*}
We are now in the condition to reduce the number of equations in  \eqref{eq:Maxwell+Navier-Stokes}. By considering the electric field $\overrightarrow E(\tau,X,Y)  = (0, 0, E(\tau,X,Y))^T $ (whose divergence is always null) and recalling the definition of the vector field $\overrightarrow B^T = (-B_2, B_1)^T$, we finally gather
\begin{equation}\label{eq:Maxwell+Navier-Stokes2}
    \left\{\!
    \begin{alignedat}{4}
        &
        \frac{\rho \nu \mathcal{J}}{c^2}
         \partial_{\tau}^2 \overrightarrow U
        \! +\! 
        \rho (\partial_t \overrightarrow U\! +\! \overrightarrow U \!\cdot \!\nabla \overrightarrow U )\! -\! \rho \nu \Delta \overrightarrow U \!+\! \nabla P \!= \!\sigma (\overrightarrow B(\overrightarrow U\!\cdot\! \overrightarrow B) - \overrightarrow U | \overrightarrow B|^2 ) + \sigma  E \overrightarrow B^T \;  (\tau,X,Y)\in 
		&&  (0,T)\!\times\! \mathbb{R}\!\times\! \mathbb{R}_+,\\
        &\dv \overrightarrow U = 0&&  (0,T)\!\times\! \mathbb{R}\!\times\! \mathbb{R}_+,\\
        &\frac{1}{\sigma \mu_0 c^2}\partial^2_{\tau} \overrightarrow B + \partial_\tau  \overrightarrow B - \frac{1}{\sigma \mu_0}\Delta \overrightarrow B = 
        \overrightarrow B\cdot \nabla \overrightarrow U - \overrightarrow U\cdot \nabla \overrightarrow B	
        && (0,T)\!\times\! \mathbb{R}\!\times\! \mathbb{R}_+,\\
        &\partial_\tau \overrightarrow B +{\rm curl}\, \overrightarrow E =0&&  (0,T)\!\times\! \mathbb{R}\!\times\! \mathbb{R}_+,\\
        &{\rm div}\, \overrightarrow B = 0&&  (0,T)\!\times\! \mathbb{R}\!\times\! \mathbb{R}_+,
    \end{alignedat}
    \right.
\end{equation}
%
with boundary conditions
\begin{equation*}
      \overrightarrow U(\tau,X,0) = 0\in \mathbb{R}^2, 
    \quad 
    \overrightarrow B(\tau,X,0) = 
    \overrightarrow{\mathbf{B}}\in \mathbb{R}^2
    \quad 
    \text{and}
    \quad 
    E(\tau,X,0) =
    \mathbf{E}
    \in \mathbb{R}.
\end{equation*}
Before performing an asymptotic analysis of equations \eqref{eq:Maxwell+Navier-Stokes2} to derive the boundary-layer model \eqref{eq:main-system}, it is reasonable to first recast equations \eqref{eq:Maxwell+Navier-Stokes2} in their dimensionless form.

\subsection{The equations in dimensionless form}\label{sec:navier-stokes-maxwell-dimensionless}$\,$

\noindent 
We shall first briefly recall some dimensionless parameters which are well-known in the magnetohydrodynamic theory. 
We refer to \cite{davidson_2001,priest_2014} for additional details and an exhaustive overview of the underlying physics.

\noindent Throughout this manuscript we denote by $ U_0\in \mathbb{R}_+$ and by $  B_0\in \mathbb{R}_+$ the sizes of the characteristic speed and magnetic field of the fluid, respectively. We further denote by $l_0\in \mathbb{R}_+$ and $t_0\in \mathbb{R}_+$ the underlying length and time scales (thus $U_0 = l_0/t_0$). The ratio between the sizes of the inertial and viscous terms is given by the Reynolds number ${\rm Re} = l_0 U_0/\nu$, while ${\rm Re}_m = l_0 U_0 \sigma \mu_0$ stands for the magnetic Reynolds number and measures the coupling between flow and magnetic field. The Hartmann number ${\rm Ha}=B_0 l_0 \sqrt{\sigma/\rho \nu }$ represents the ratio between the magnetic and viscous forces. 

\smallskip
\noindent 
Hence, we can introduce the change of variables $t' = \tau/t_0\in (0,T')$ (with $T' = T/t_0$), $x' = X/l_0\in \mathbb R$, $y' = Y/l_0 \in \mathbb R_+$ as well as the new state variables $ \overrightarrow U' =  \overrightarrow U/U_0 $, $ \overrightarrow B' =  \overrightarrow B/B_0 $ and $E' = E B_0U_0 $. In these new framework the equations in \eqref{eq:Maxwell+Navier-Stokes2} become
\begin{equation}\label{eq:Maxwell+Navier-Stokes2-dimensionless}
    \left\{
    \begin{alignedat}{2}
        &
        \left(\frac{ U_0}{c}\right)^2
        \frac{\mathcal{J}}{\rm Re}
        \partial_{t'}^2 \overrightarrow U'\!+\! 
        \partial_{t'} \overrightarrow U' \!+\! \overrightarrow U' \!\cdot\! \nabla \overrightarrow U'  \!-\! \frac{1}{{\rm Re}} 
        \Delta \overrightarrow U' \!+\! \nabla P\! 
       = \\
       &
       \hspace{3cm}=
       \!
       \frac{{\rm Ha^2}}{{\rm Re}} \big(\overrightarrow B'
       (\overrightarrow U'\!\cdot \!\overrightarrow B') \!- \!
       \overrightarrow U' | \overrightarrow B'|^2 \! +\!  E' \overrightarrow B'^T \big),
        \;  (t'\!,x'\!,y')\!\in 
		&&  (0,T')\!\times \!\mathbb{R}\!\times \! \mathbb{R}_+,\\
        &\dv \overrightarrow U' = 0,
        &&  (0,T')\!\times \!\mathbb{R}\!\times \! \mathbb{R}_+,\\
        &\left(\frac{U_0}{c}\right)^2\frac{1}{\,{\rm Re_m}}\partial^2_{t't'} \overrightarrow B' + \partial_{t'}  \overrightarrow B' -\frac{1}{\,{\rm Re_m}} \Delta \overrightarrow B' = 
        \overrightarrow B'\cdot \nabla \overrightarrow U' - \overrightarrow U'\cdot \nabla \overrightarrow B',
        &&  (0,T')\!\times \!\mathbb{R}\!\times \! \mathbb{R}_+,\\
        &\partial_{t'} \overrightarrow B' +{\rm curl}\, \overrightarrow E' =0,
        &&  (0,T')\!\times \!\mathbb{R}\!\times \! \mathbb{R}_+,\\
        &{\rm div}\, \overrightarrow B' = 0,&&  (0,T')\!\times \!\mathbb{R}\!\times \! \mathbb{R}_+,
    \end{alignedat}
    \right.
\end{equation}
with boundary conditions
\begin{equation}\label{initial-data-dimensionless-maxwell-navier-stokes}
      \overrightarrow U'(t',x',0) = 0\in \mathbb{R}^2, 
    \quad 
    \overrightarrow B'(t',x',0) = 
    \overrightarrow{\mathbf{B}}' 
    :=  \frac{\overrightarrow{\mathbf{B}}}{B_0}\in \mathbb{R}^2
    \quad 
    \text{and}
    \quad 
    E'(t,x,0) = 
    \mathbf{E}':=
    \mathbf{E}B_0U_0
    \in \mathbb{R},
\end{equation}
where all spatial derivatives $\nabla$, $\Delta$ and $\dv$ are now in terms of $(x',y')$. The behaviour of solutions of the system is therefore quantified by the Reynolds number ${\rm Re}$ and the Hartmann number in the momentum equation, as well as by the characteristic speed $U_0$ and the magnetic Reynolds number ${\rm Re_m}$ in the third equation. 
\begin{rem}\label{rem:differences-between-MHD-Prandtl-and-ours}
    Let us comment on a major difference between the parameters of the Navier-Stokes-Maxwell equations given by \eqref{eq:Maxwell+Navier-Stokes2-dimensionless} and classic MHD system (i.e.~when $c\sim+\infty$).
    The forcing term on the right-hand side of the balance of linear momentum in \eqref{eq:Maxwell+Navier-Stokes2-dimensionless} is driven by the dimensionless constant ${\rm Ha}^2/{\rm Re}$. At a first glance, this constant differs from ${\rm Ha}^2/({\rm Re}\,{\rm Re_m})$ in the linear momentum of MHD, i.e.
    \begin{equation*}
        \partial_{t'} \overrightarrow U' + \overrightarrow U' \cdot \nabla \overrightarrow U'  - \frac{1}{{\rm Re}} 
        \Delta \overrightarrow U' + \nabla P 
        =\frac{{\rm Ha^2}}{{\rm Re}\,{\rm Re_m}} 
        \Big(
            B' \cdot \nabla B' - \nabla \frac{|B'|^2}{2}
        \Big) 
        = \frac{{\rm Ha^2}}{{\rm Re}\,{\rm Re_m}} \curl B'\times B' ,
    \end{equation*}
    This observation is however imprecise. Indeed, by neglecting $1/c^2\partial_t E$ in \eqref{eq:Maxwell+Navier-Stokes} (thus also $U_0^2/(c^2{\rm Re_m})\partial^2_{t't'} \overrightarrow B'$ in the third equation of \eqref{eq:Maxwell+Navier-Stokes2-dimensionless}), the Ohm's law together with the Ampere's law imply
    \begin{equation*}
        \curl B'\times B' = \mu_0  l_0 U_0 \sigma \big( U'\times B' + E' \big) \wedge B' = 
        {\rm Re_m}
        \big(\overrightarrow B'(\overrightarrow U'\cdot \overrightarrow B') - \overrightarrow U' | \overrightarrow B'|^2  +  E' \overrightarrow B'^T \big),
    \end{equation*}
    which provides the additional constant $1/{\rm Re_m}$ to gather ${\rm Ha}^2/({\rm Re}\,{\rm Re_m})$. System \eqref{eq:Maxwell+Navier-Stokes2-dimensionless} is therefore an extension of the MHD-equations in dimensionless form.
\end{rem}

\noindent 
In the forthcoming sections, we aim to reveal system \eqref{eq:main-system}, by sending ${\rm Re}$, ${\rm Re_m}$ and $\rm Ha$ towards $\infty$ in \eqref{eq:Maxwell+Navier-Stokes2-dimensionless}. To this end, we consider suitable conditions on $U_0$ and ${\rm Ha}$, as well as some valid rescalings of the variables $(t',x',y')$ near the boundary.  In particular, in section \ref{sec:modelling-Prandtl} we derive system \eqref{eq:main-system} as boundary layer with thickness of Prandtl type, while in section   \ref{sec:modelling-Prandtl} the equations are revealed as a layer with thickness of Hartmann type.

\subsection{Insurgence of Boundary layers with thickness of Prandtl type}\label{sec:modelling-Prandtl}$\,$

\noindent 
The goal of this section is to prove Theorem \ref{thm:main-modelling}, analysing the asymptotic limit of equations \eqref{eq:Maxwell+Navier-Stokes2-dimensionless} on a thin layer near the boundary, under suitable assumptions on the dimensionless parameters. Boundary layers commonly appear in fluid dynamics at high values of the Reynolds number ${\rm Re} \gg 1$, however we shall also here clarify the asymptotic of the magnetic Reynolds number ${\rm Re}_m$ and the Hartmann number {\rm Ha}. 
We here assume that the ratio between the magnetic Reynolds number and the Reynolds number (known as magnetic Prandtl number) is fixed
\begin{equation*}
    \frac{{\rm Re_m}}{{\rm Re}} = {\rm Pr_m}>0.
\end{equation*}
thus also the magnetic Reynolds number assumes high values ${\rm Re_m} \gg 1$. Similarly, we address the case in which the Hartmann number {\rm Ha} diverges to $\infty$ proportionally to ${\rm Re}$.

\noindent
We can then summarise the statement of Theorem \ref{thm:main-modelling} as follows. 

\begin{theorem}\label{thm:main-modelling-sec2}
    Assume that the following relations between the dimensionless parameters in system \eqref{eq:Maxwell+Navier-Stokes2-dimensionless} are satisfied:
    \begin{equation}\label{conditions-Ha-U0-Prandtl}
        \lim_{{\rm Re}\to +\infty} \frac{{\rm Ha}}{{\rm Re}}= \mathbb{H}\in \mathbb{R},
        \quad  
         \frac{{\rm Re_m}}{{\rm Re}} = {\rm Pr_m}\text{ is fixed,}
        \quad \lim_{{\rm Re}\to +\infty} 
        \left( \frac{U_0}{c}\right)^2 \frac{1}{{\rm Re}} = \kappa \in \mathbb{R},
        \quad 
         \mathbb J := \kappa \mathcal{J}\text{ is fixed}.
    \end{equation}
    Furthermore, assume that the initial data $ \mathbf{B}' = (\mathbf{B}_1',\mathbf{B}_2')^T\in \mathbb R^2$ and $\mathbf{E}'\in \mathbb R$ in \eqref{initial-data-dimensionless-maxwell-navier-stokes} are such that 
    \begin{equation*}
       \mathbf{B}_1'= \mathbf{b}_1\in \mathbb{R} \text{ is fixed, while}\quad 
       \mathbf{B}_2' = \frac{1}{\sqrt{{\rm Re}}} \mathbf{b}_2,\quad 
       \mathbf{E}' = \frac{1}{\sqrt{{\rm Re}}} \mathbf{e},
    \end{equation*}
    for some $ \mathbf{b}_2\in \mathbb{R}$ and $ \mathbf{e}\in \mathbb{R}$.
    Then system \eqref{eq:main-system} appears as boundary layer in the region
    \begin{equation}\label{region-close-to-the-boundary}
        (t',x',y') \in \mathbb{R}_+ \times \mathbb{R}\times  \left[0,\frac{1}{\sqrt{\rm Re}}\right],
    \end{equation}
    when ${\rm Re}\to +\infty$ (and thus also when ${\rm Re}$, ${\rm Ha}\to +\infty$).
\end{theorem}

\begin{rem}
    Before addressing the proof of this Theorem, we shall first clarify certain aspects and terminologies of its statement. 
    \begin{itemize}
            \item The third condition in \eqref{conditions-Ha-U0-Prandtl} is rather unphysical, since it implies that the characteristic speed $U_0$ converges towards $\infty$ for high values of ${\rm Re}\gg 1$. This relation seems however necessary in order to avoid loosing contribution from the equations of $(b_1,b_2)$. A more physical scenario is addressed in Section \ref{sec:Hartmann}.
    \item 
        Although the boundary layer appears in the region given by \eqref{region-close-to-the-boundary}, 
        System \eqref{eq:main-system} is written in terms of rescaled variables $(t, x,y)$ in \eqref{Prandtl-change-of-variables} and rescaled functions $(u,b_1,b_2,p,e)$ in \eqref{new-variables-asymptotic}. Therefore, the domain of system \eqref{eq:main-system} does not depend on the Reynolds number: $(t, x,y)\in \mathbb{R}_+ \times \mathbb{R} \times [0,1]$.
    \item By saying that ``System \eqref{eq:main-system} appears as boundary layer'', we mean that as long as the triple $(\overrightarrow U', \overrightarrow B', E')$ (which depends on $\rm Re$) converges towards a profile $((u,0)^T,(b_1,b_2)^T,e)$ when ${\rm Re} \gg 1$ (under a suitable rescaling), then $(u,b_1,b_2,e)$ must satisfy system \eqref{eq:main-system}. The convergence is well known in the purely hydrodynamic regime $({\rm Ha} = 0)$, a fact that have been highly studied through the stability theory of the Prandtl equations. We infer that this convergence holds true also when ${\rm Ha} \neq 0$, but that is beyond the scope of this paper.
    \end{itemize}
\end{rem}

\begin{proof}
We denote by $\varepsilon^2= 1/{\rm Re}$ the inverse of the Reynolds number, which converges towards $0$ when ${\rm Re}$ converges to $\infty$. The parameter $\varepsilon\ll 1$ represents the size of the region in which the boundary layer occurs.
We derive System \eqref{eq:main-system} as a rescaled version of the asymptotic limit of \eqref{eq:Maxwell+Navier-Stokes2} within the domain  $(t,x,y) \in [0,T] \times \mathbb{R}\times  [0,\varepsilon] $, by (informally) sending $\varepsilon$ towards $0$. 
To this end, we shall first introduce the change of variables
\begin{equation}\label{Prandtl-change-of-variables}
        t = t'\in [0,T],\quad  
        x = x'\in \mathbb{R},\quad \text{and}\quad  
        y = y/\varepsilon \in [0,1],
\end{equation}
as well as the following new state variables
\begin{equation}\label{new-variables-asymptotic}
\begin{alignedat}{64}
    &u(t, x,y) 
    &&:= 
    U_1(t , x,\varepsilon y),
    \quad
    &&&&
    v(t, x,y) 
    &&&&&&&&
    := 
    U_2(t , x,\varepsilon y)/ \varepsilon,
    \quad
    &&&&&&&&&&&&&&&&
    b_1(t, x,y)    
    &&&&&&&&&&&&&&&&
    := 
    B_1(t , x,\varepsilon y),\\
    &b_2(t, x,y) 
    &&:= 
    B_2(t , x,\varepsilon y)/ \varepsilon,
    \quad 
    &&&&
    p(t, x, y) 
    &&&&&&&&
    :=  P(t , x,\varepsilon y),\qquad 
    &&&&&&&&&&&&&&&&
    e(t, x, y) 
    &&&&&&&&&&&&&&&&
    :=
     E(t , x,\varepsilon y)/\varepsilon.
\end{alignedat}
\end{equation}
Hence, we can develop System \eqref{eq:Maxwell+Navier-Stokes2-dimensionless} in terms of $u$ $v$ $e$ and $p$, as well as the variables $(t,x,y)$:
\begin{equation*}
    \begin{cases}
         \left(\frac{U_0}{c}\right)^2
         \frac{\varepsilon^2}{\kappa}
         \mathbb J
         \,
         \partial_t^2 u + 
         \partial_t   u + u \partial_x u + v \partial_y u  - \varepsilon^2  \partial_{x}^2 u -   \partial_{y}^2 u + \partial_x p = 
         {\rm Ha^2} \varepsilon^4  
         \big(b_1  b_2 v  - u  b_2^2 - b_2e\big),\\
         \varepsilon (\partial_t  v + u \partial_x v + v \partial_y v )- 
         \varepsilon^3  \partial_{x}^2v+
         \frac{1}{\varepsilon}\partial_y p = 
        {\rm Ha^2} \varepsilon^3  
        \big( b_2b_1 u - b_1^2 v + b_1 e \big),\\
        \partial_x u + \partial_y v = 0,\\
        \left(\frac{U_0}{c}\right)^2
        \frac{\varepsilon^2}{ {\rm Pr_m}}
        \partial_{t}^2 b_1 + \partial_t b_1 - \frac{\varepsilon^2}{{\rm Pr}_m} \partial_{x}^2 b_1 - \frac{1}{{\rm Pr}_m}\partial_{y}^2 b_1 = b_1 \partial_x u + b_2 \partial_y u  - u \partial_x b_1 - v \partial_y b_1,\\
        \left(\frac{U_0}{c}\right)^2
        \frac{\varepsilon^3}{ {\rm Pr_m}}
        \partial_{t}^2 b_2 + \varepsilon \partial_t b_2 - \frac{\varepsilon^3}{{\rm Pr}_m} \partial_{x}^2 b_2 - \frac{\varepsilon}{{\rm Pr}_m}\partial_{y}^2 b_2 = \varepsilon b_1 \partial_x v + \varepsilon b_2 \partial_y v  - \varepsilon u \partial_x b_2 - \varepsilon v \partial_y b_2,\\
        \partial_t b_1 + \partial_y e = 0,\\
        \varepsilon(\partial_t b_2 - \partial_x e) = 0,\\
        \partial_x b_1 + \partial_y b_2 = 0.\\
    \end{cases}
    \!
\end{equation*}
with boundary conditions
\begin{equation*}
    u(t ,x,0) = 0, 
    \quad 
    v(t ,x,0) = 0,
    \quad 
    b_1(t ,x,0) = 
    \mathbf{b}_1',
    \quad 
    b_2(t ,x,0) = 
    \mathbf{b}_2 
    \quad 
    \text{and}
    \quad 
    e(t ,x,0) =
    \mathbf{e}
    \in \mathbb{R}.
\end{equation*}
We now remark that the conditions on \eqref{conditions-Ha-U0-Prandtl} implies that ${\rm Ha}^2\varepsilon^4\to \mathbb H^2$ and $(U_0/c)^2\varepsilon^2 = (U_0/c)^2/{\rm Re} \to \kappa$ as $\varepsilon \to 0$. System \eqref{eq:main-system} appears therefore by multiplying the second equation in \eqref{eq:Maxwell+Navier-Stokes2-dimensionless} by $\varepsilon$, dividing the fifth equation by $\varepsilon$ and finally sending $\varepsilon$ towards $0$.
\end{proof}

\subsection{Boundary layers with thickness of Hartmann type}\label{sec:Hartmann}$\,$

\noindent 
In the modelling of Section \ref{sec:modelling-Prandtl}, we have imposed that the characteristic velocity $U_0$ blows up at high value of the Reynolds number (cf.~the third relation in \eqref{conditions-Ha-U0-Prandtl}). Being this a major drawback on the physics of the system, we can adjust this nonphysical scenario, by considering a different type of rescaling than the one introduced in \eqref{Prandtl-change-of-variables} and \eqref{new-variables-asymptotic}. The corresponding thickness of the boundary layer will now be of Hartmann type, since it is inversely proportional to {\rm Ha} (cf.~\eqref{new-variables-asymptotic-Hartmann} and \eqref{new-functions-Hartmann}). Since {\rm Ha} and {\rm Re} are also in this section proportional, 
the thickness behaves like $1/{\rm Re}$. This is much smaller than $1/\sqrt{\rm Re}$ as in \eqref{region-close-to-the-boundary}, so the nature of this thickness reminds the one of Hartmann layers (of course we have slightly abused the notation about Hartmann, since Hartmann layers usually occur when the magnetic field is oriented at some specific angles, whereas our modelling treat more general scenario).

\noindent 
Although we introduce a different scaling, the derived equations remain the same as in System \eqref{eq:main-system}. We exploit this aspect in the following statement.

\begin{theorem}\label{thm:main-modelling2Hartmann}
    Assume that the following relations between the dimensionless parameters in system \eqref{eq:Maxwell+Navier-Stokes2-dimensionless} are satisfied:
    \begin{equation}\label{conditions-Ha-U0-Hartmann}
        \lim_{{\rm Re}\to +\infty} \frac{{\rm Ha}}{{\rm Re}}= \mathbb{H}\in \mathbb{R} 
        \quad  
        \text{and}
        \quad
        {\rm Pr_m}:=\frac{{\rm Re_m}}{{\rm Re}},
        \;  \kappa:= \left( \frac{U_0}{c}\right)^2,
        \;
         \mathbb J := \kappa \mathcal{J}\text{ are all fixed}.
    \end{equation}
    Furthermore, assume that the initial data $ \mathbf{B}' = (\mathbf{B}_1',\mathbf{B}_2')^T$ in \eqref{initial-data-dimensionless-maxwell-navier-stokes} are such that 
    \begin{equation*}
       \mathbf{B}_2'= \mathbf{b}_2\in\mathbb R  \text{ is fixed, while} 
       \quad  
       \mathbf{B}_1' = \sqrt{\frac{\rm Ha}{\mathbb H}} \mathbf{b}_1,
       \quad 
       \mathbf{E}' =  \sqrt{\frac{\rm Ha}{\mathbb H}} \mathbf{e},
    \end{equation*}
    for some $ \mathbf{b}_1\in \mathbb{R}$ and $ \mathbf{e}\in \mathbb{R}$. Then system \eqref{eq:main-system} appears as boundary layer in the region
    \begin{equation}\label{region-close-to-the-boundary-Hartmann}
        (t',x',y') \in \left(0, \frac{\mathbb H}{\rm Ha}T'\right) \times \mathbb{R}\times  \left(0,\frac{\mathbb H}{\rm Ha}
        \right),
    \end{equation}
    when ${\rm Re}\to +\infty$ (and thus also when ${\rm Re}$, ${\rm Ha}\to +\infty$).
\end{theorem}
\begin{rem}
Some remarks are here in order:
\begin{itemize}
    \item We do not impose any condition on $U_0$ in terms of the Reynolds number, therefore this characteristic speed can range within physical values below the speed of light. 
    \item The interval $ [0, \mathbb H/{\rm Ha}]$ in \eqref{region-close-to-the-boundary-Hartmann} can be replaced by $ [0, 1/{\rm Re}]$, since our assumptions ensure that $ \mathbb H/{\rm Ha}\approx 1/{\rm Re}$ when $  {\rm Re}\gg 1 $. \item The domain in \eqref{region-close-to-the-boundary-Hartmann} represents a region close to the boundary of the domain $(t',x',y')\in \mathbb R\times\mathbb{R}^2_+$. From the scaling of the the new variables (cf.~\eqref{new-variables-asymptotic-Hartmann}) the domain of model \eqref{eq:main-system} shall be better understood as an asymptotic expansion of a different domain, namely $(t',x',y')\in [0, T/{\rm Ha}]\times [-\sqrt{X_0/ {\rm Ha}},  \sqrt{X_0/ {\rm Ha}}] \times  [0, \mathbb H/{\rm Ha}] $. In particular, the singular behaviour of the solutions given by \eqref{eq:main-system} appears close to the boundary of the domain $\mathbb R\times\mathbb{R}^2_+$ and close to the origin both in time $t'=0$ and in space $x'= 0$. This particular region of the domain is motivated by the fact that also the term $ ( U_0/c)^21/{\rm Re_m }\partial_{t't'}^2 \overrightarrow B'$ of the displacement current in \eqref{eq:Maxwell+Navier-Stokes2-dimensionless} is now vanishing. As a result, when ${\rm Re\gg 1}$, the limit system of \eqref{eq:Maxwell+Navier-Stokes2-dimensionless} requires different boundary conditions both in space (i.e.~in $y'= 0$) as well as in time  (i.e.~in $t'= 0$).
\end{itemize}
\end{rem}
\begin{proof}[Proof of Theorem \ref{thm:main-modelling2Hartmann}]
We denote by $\delta = \mathbb H/{\rm Ha}\ll 1$ the size of the region in which boundary layer occurs. We thus introduce the change of variables
\begin{equation}\label{new-variables-asymptotic-Hartmann}
    t= \frac{t'}{\delta }\in  \mathbb R_+,\qquad 
    x = \frac{x'}{\sqrt{\delta }} \in  \mathbb R,\qquad 
    \quad
    y = \frac{y'}{\delta} \in [0,1]
\end{equation}
and the functions
\begin{equation}\label{new-functions-Hartmann}
\begin{alignedat}{64}
    &u(t, x,y) 
    &&:= 
    \sqrt{\delta}\,
    U_1'(\delta t ,\sqrt{\delta} x,\delta y),
     \;
    &&&&
    v(t, x,y) 
    &&&&&&&&
    := 
    U_2'(\delta T ,\sqrt{\delta} x,\delta y),
    \;
    &&&&&&&&&&&&&&&&
    b_1(t, x,y)    
    &&&&&&&&&&&&&&&&
    := 
    \sqrt{\delta}
    B_1'(\delta t , \sqrt{\delta} x,\delta y),\\
    &b_2(\tau, X,Y) 
    &&:= 
    B_2'(\delta t ,\sqrt{\delta} x,\delta y),
    \quad 
    &&&&
    p(t, x, y) 
    &&&&&&&&
    := \delta  P'(\delta  t, \sqrt{\delta}x, \delta y),
    \hspace{0.2cm} 
    &&&&&&&&&&&&&&&&
    e(t, x, y) 
    &&&&&&&&&&&&&&&&
    :=
    \sqrt{\delta} E'(\delta t, \sqrt{\delta} x, \delta y).
\end{alignedat}
\end{equation}
Hence, by recasting system \eqref{eq:Maxwell+Navier-Stokes2} in terms of the new functions and variables, we deduce that $(u,v)$, $(b_1,b_2)$ and $e$ are solutions of
\begin{equation}\label{eq:final-hartmann}
    \begin{cases}
         \delta^{-\frac{3}{2}}
         \big(
         \mathbb J \frac{{\rm Ha}}{\mathbb H {\rm Re}}\partial_{t}^2 u+
         \partial_t   u + u \partial_x u + v \partial_y u\big)  - \frac{{\rm Ha}}{{\rm Re}\mathbb H}\delta^{-\frac 12}  \partial_{x}^2 u - 
         \frac{{\rm Ha}}{{\rm Re}\mathbb H}\delta^{-\frac 32 } \partial_{y}^2 u + \delta^{-\frac 32 }\partial_x p =\\
         \hspace{10cm}= 
         \frac{\rm Ha^2}{\rm Re}  \delta^{-\frac 12}  
         \big(b_1  b_2 v  - u  b_2^2 - b_2e\big),\\
         \delta^{-1}(\partial_t  v + u \partial_x v + v \partial_y v )- 
          \frac{{\rm Ha}}{{\rm Re}\mathbb H}
          \partial_{x}^2 v - 
          \frac{{\rm Ha}}{{\rm Re}\mathbb H}\delta^{-1}\partial_{y}^2 v  +  
          \delta^{-2}
         \partial_y p = 
        \frac{\rm Ha^2}{\rm Re}\delta^{-1}
        \big( b_2b_1 u - b_1^2 v + b_1 e \big),\\
        \delta^{-1}(\partial_x u + \partial_y v) = 0,\\
        \left(\frac{U_0}{c}\right)^2
        \frac{{\rm Ha}}{{\rm Re}\mathbb H}
        \frac{ \delta^{-\frac 32}}{ {\rm Pr_m}}
        \partial_{t}^2 b_1 
        \!+\!  
        \delta^{-\frac 32}\partial_t b_1 
        \!-\! 
        \frac{ \delta^{-\frac 12}}{{\rm Pr}_m}
        \frac{{\rm Ha}}{{\rm Re}\mathbb H}
        \partial_{x}^2 b_1 
        \!-\!
         \frac{ \delta^{-\frac 32}}{{\rm Pr}_m}
        \frac{{\rm Ha}}{{\rm Re}\mathbb H}
        \partial_{y}^2 b_1 \!=\!
        \delta^{-\frac 32}\big(
            b_1 \partial_x u \!+\! 
            b_2 \partial_y u \!-\! 
            u \partial_x b_1 \! - \!
            v \partial_y b_1
        \big),\\
        \left(\frac{U_0}{c}\right)^2
        \frac{{\rm Ha}}{{\rm Re}\mathbb H}
        \frac{ \delta^{-1}}{{\rm Pr}_m}
        \partial_{\tau\tau}^2 b_2 
        \!+\!
        \delta^{-1} \partial_t b_2 
        \!-\! 
        \frac{1}{{\rm Pr}_m} 
        \frac{{\rm Ha}}{{\rm Re}\mathbb H}
        \partial_{x}^2 b_2 
        \!-\! 
        \frac{ \delta^{-  1}}{{\rm Pr}_m}
        \frac{{\rm Ha}}{{\rm Re}\mathbb H}
        \partial_{y}^2 b_2 
        \!=\! 
        \delta^{-1}
        \big(
            b_1 \partial_x v \!+\! 
            b_2 \partial_y v \!-\!  
            u \partial_x b_2 \!-\! 
            v \partial_y b_2
        \big),\\
        \delta^{-\frac{3}{2}}\big(\partial_t b_1 + \partial_y e\big) = 0,\\
        \delta^{-1}\big(\partial_t b_2 - \partial_x e\big) = 0,\\
        \delta^{-1}\big(\partial_x b_1 + \partial_y b_2\big) = 0,\\
    \end{cases}
    \!
\end{equation}
with boundary conditions
\begin{equation*}
    u(t ,x,0) = 0, 
    \quad 
    v(t ,x,0) = 0,
    \quad 
    b_1(t ,x,0) = 
    \mathbb{b}_1,
    \quad 
    b_2(t ,x,0) = 
    \mathbf{b}_2' 
    \quad 
    \text{and}
    \quad 
    e(t ,x,0) =
    \mathbb{e}.
\end{equation*}
We thus multiply the first, fourth and sixth equations in \eqref{eq:final-hartmann} by $\delta^{3/2}$, as well as 
the second, third, fifth, seventh and eighth equations by $\delta$. Finally, remarking that
\begin{equation*}
    \lim_{{\rm Re} \to \infty} \frac{{\rm Ha}}{{\rm Re}\mathbb H} = 1,\qquad 
     \lim_{{\rm Re}\to \infty} \frac{{\rm Ha^2}}{{\rm Re}} = 
      \lim_{{\rm Re} \to \infty} \frac{{\rm Ha^2}}{{\rm Re}}\delta =
       \lim_{{\rm Re} \to \infty} \frac{{\rm Ha}}{{\rm Re}}\mathbb H = \mathbb H^2,
\end{equation*}
we finally derive the main system \eqref{eq:main-system}, by also denoting $\kappa = (U_0/c)^2$.
\end{proof}

\section{Analytic Solutions of System \eqref{eq:main-system}}\label{sec:analysis}

\noindent
The main goal of the present paragraph is to prove the existence of certain smooth solutions for the derived system \eqref{eq:main-system} (cf.~Theorem \ref{thm:existence-analytic-solutions}). The analysis of these equations owns similar challenges as the ones of classical Prandtl, in particular the fact that the system lacks of regularising effects on the horizontal variable $x\in \mathbb{R}$ (the dissipative mechanisms of the system reside indeed only on the variable $y\in (0,1)$). In order to cope with this difficulty, it is rather common to impose high regularities on the initial data, along such horizontal variable. Our work addresses in particular the case of analytic functions. Before stating our main result, we shall first clarify the definition of analytic solutions, that we will use throughout the next sections. Furthermore, we provide a suitable shortening of the overall system, that will simplify our forthcoming analysis.

\subsection{Analytic functions on the horizontal direction}\label{sec:analytic-fcts}$\,$

\noindent 
A function  $f = f(x)$ is analytic in $x\in \mathbb R$, if its Fourier transform $\mathcal{F}_x(f)(\xi) = \hat{f}(\xi)$ decays exponentially to zero as $e^{-\tau|\xi|}$, for some $\tau>0$, when the frequency $\xi \in \mathbb{R}$ diverges to $\pm \infty$. For a fixed $\tau>0$ (which stands for the radius of analyticity of $f$), this function space is indeed Banach. Between the several equivalent norms, we will make use of the one given by $\| e^{\tau |D_x|}f \|_{H^s(\mathbb{R})}$, where $H^s(\mathbb{R})$ is a Sobolev space with regularity $s>2$ and  $e^{\tau |D_x|}$ stands for the Fourier multiplier $ \mathcal{F}_x(e^{\tau |D_x|}f)(\xi) = e^{\tau|\xi|} \hat{f}(\xi) $.

\noindent 
Thus, we are interested in solutions $(u,v,b_1,b_2)$ which depend upon $(t,x,y)\in (0,T)\times \mathbb{R}\times (0,1)$ and are analytic in the variable $x\in \mathbb R$, as described by the following function space:
\begin{equation}\label{sec:exitence-function-space}
\begin{aligned}
    e^{\tau(t) |D_x|}u,\,
    e^{\tau(t) |D_x|}b_1,\,
    e^{\tau(t) |D_x|}b_2,\,
     &\in 
    C([0,T],H^{s+1,0}(\mathbb{R}\times (0,1)),
    \\ 
    e^{\tau(t) |D_x|}\partial_t u,\,
    e^{\tau(t) |D_x|}\partial_t b_1,\,
    e^{\tau(t) |D_x|}\partial_t b_2 ,\,
    e^{\tau(t) |D_x|}v,\,
    e^{\tau(t) |D_x|}e
    &\in 
    C([0,T],H^{s,0}(\mathbb{R}\times (0,1)).
\end{aligned}
\end{equation}
The space $C([0,T],H^{s,0}(\mathbb{R}\times (0,1))$ (which we will abbreviate from now on by $C([0,T],H^{s,0})$) is anisotropic in space, namely it has $H^s$ regularity in $x\in \mathbb{R}$ and only $L^2$-regularity in $y\in(0,1)$. The corresponding norm on a general function $g = g(t,x,y)$ is given by
\begin{equation*}
        \sup_{t\in (0,T)}
        \int_0^1 
        \int_\mathbb{R}
        (1+|\xi|)^{2s}
        \big|
            \hat{g}(t,\xi,y)
        \big|^2  
        d\xi
        dy
        < \infty,
        \quad 
        \text{where}
        \quad 
        \hat{g}(t,\xi,y)
        =
        \mathcal{F}_x(g)(t,\xi,y)
        := 
        \int_{\mathbb{R}}
        e^{-i x \xi}f(t,x,y)dx
\end{equation*}
stands for the Fourier transform in the horizontal variable $x\in \mathbb R$. 

\noindent
The radius of analyticity $\tau(t)>0$ of the solutions in \eqref{sec:exitence-function-space} is explicitly defined in Theorem \ref{thm:existence-analytic-solutions} (cf.~\eqref{radius-of-analyticity-thm:existence-of-solutions}). It depends on the radius of analyticity of the initial data and coincides with it at initial time $t= 0$. One shall furthermore remark that $\tau(t)$ decreases in $t\in (0,T)$, a fact that expresses (roughly speaking) the degrading mechanisms of the regularity in the horizontal variable.

\subsection{Reduced system}\label{sec:reduced-system}$\,$

\noindent 
We remark that system \eqref{eq:main-system} can be shortened, since both $v$ and $b_2$ in \eqref{eq:main-system} are determined by the divergence-free relations $\partial_x u + \partial_y v = 0$, $\partial_x b_1 + \partial_y b_2 = 0$ and the boundary conditions $v_{y=0} = 0$, $b_{2|y=0} = 0$:
\begin{equation}\label{formula-for-v-and-b2}
    v(t,x,y) = -\int_0^y\partial_x u(t,x,z) dz,
    \quad\text{and}\quad  
    b_2(t,x,y) = -\int_0^y\partial_x b_1(t,x,z) dz.
\end{equation}
The above identities are well-defined, since we coupe with solutions that are smooth in $x\in \mathbb R$ and the corresponding derivatives $\partial_x u(t,x,\cdot)$, $\partial_x b_1(t,x,\cdot)$ are $L^2$-integrable in $y\in (0,1)$ (for any $(t,x)\in (0,T)\times \mathbb{R}$). Moreover, for the sake of a compact presentation, we will shorten from now on the identities in \eqref{formula-for-v-and-b2} as $ v  = -\int_0^y\partial_x u $ and $b_1 =-\int_0^y\partial_x b_1$.

\noindent
Similarly, the magnitude $e$ of the electric field can be recasted just in terms of $b_1$, making use of the relation:
\begin{equation*}
    \begin{cases}
        \partial_t b_1 + \partial_y e = 0\quad 
        &\text{in }(0,T) \times \mathbb R \times (0,1),\\
        e = 0 &\text{on }(0,T) \times \mathbb R \times \{ 0 \},
    \end{cases}
    \quad \Longleftrightarrow \quad 
        e(t,x,y) = 
        -\int_0^y \partial_t b_1(t,x,z)dz
\end{equation*}
for any $(t,x,y)\in (0,T)\times \mathbb R \times (0,1)$. 

\noindent 
Because of these aspects, the equations of $b_2$ and $e$ are redundant in system \eqref{eq:main-system} and we can reduce the considered model uniquely in terms of $(u,\,b_1)$:
\begin{equation}\label{eq:MHD-Prandtl-without-e}
    \begin{cases}
        \mathbb{J}
        \partial_{tt} u
        +
        \partial_t u + 
        u \partial_x u +
        v \partial_y u - 
        \partial_{yy}^2 u + \partial_x p = 
        \mathbb{H}^2
        \Big\{
        b_1 
        b_2
        v  - 
        u 
        b_2^2
        +b_2
        \big( 
            \int_0^y
            \partial_t b_1 
        \big)
        \Big\},       
        \\
        \partial_y p = 
        \mathbb{H}^2 
        \Big\{  
            b_1 
            b_2 
            u 
            -
            b_1^2 v
        - 
        b_1
        \big( 
            \int_0^y
            \partial_t b_1 
        \big)
        \Big\}
        ,\\
        \frac{\kappa}{\rm Pr_m} \partial^2_{tt} b_1 
        +
        \partial_{t} b_1 
        + 
        u \partial_x b_1 
        +
        v
        \partial_y b_1
        -  
        \frac{1}{\rm Pr_m} \partial_{yy}^2 b_1 = b_1\partial_x u + 
        b_2\partial_y u,
    \end{cases}
\end{equation}
in $(t,x,y)\in (0,T) \times  \mathbb{R} \times  (0,1)$, coupled with \eqref{formula-for-v-and-b2} and the following initial and boundary conditions:
\begin{alignat*}{8}
    &u_{|t = 0} = \bar{u},\; 
    \partial_t u_{|t = 0} = \tilde{u},\; 
    b_{1|t= 0} = \bar b_{1},\; 
    \partial_t b_{1|t = 0} =\tilde b_1 
    &&\quad\text{ in }\mathbb{R}\times (0,1),
    \\
    &
    u_{|y= 0} = u_{|y= 1} 
    = b_{1|y= 0} = b_{1|y= 1} = 0
    &&\quad\text{ on }(0,T)\times \mathbb{R}.
\end{alignat*}

\subsection{Statement of the result}$\,$

\noindent
The function space being introduced, we can state our result, which asserts the existence of global-in-time analytic solutions with small initial data.
\begin{theorem} \label{thm:existence-analytic-solutions}
For any $s>2$, there exists a sufficiently small positive constant $\varepsilon_s\in [0,1)$ (which depends uniquely upon $s$),  such that the following result holds true. 
Let $\bar u$, $\bar b_{1}$, $\tilde u$ and $\tilde b_1$ be initial data that are  analytic in the variable $x\in \mathbb{R}$ with radius of analyticity $\tau_0>0$:
\begin{equation*}
\begin{alignedat}{4}
    e^{\tau_0 (1+|D_x|)} \bar u     
    \quad\text{and}\quad
    e^{\tau_0 (1+|D_x|)} \bar b_1    
    \quad 
    &\text{belong to}
    \quad 
    H^{s+1,1}(\mathbb R \times (0,1)),\\
    e^{\tau_0 (1+|D_x|)} \tilde u   
    \quad\text{and}\quad
    e^{\tau_0 (1+|D_x|)} \tilde b_1  
    \quad 
    &\text{belong to}
    \quad 
    H^{s,0}(\mathbb R \times (0,1)).\\
\end{alignedat}
\end{equation*} 
If the following smallness condition on the initial data holds true
\begin{equation}\label{thm:small-condition}
\begin{aligned}
        \|  
            &
            e^{\tau_0(1+|D_x|)} \bar u       
        \|_{H^{s+1,0}}
        \!+\!
        \|  
            e^{\tau_0(1+|D_x|)} 
            \partial_y \bar u       
        \|_{H^{s,0}}
        \!+\!
        \|  
            e^{\tau_0(1+|D_x|)} \tilde u       
        \|_{H^{s,0}}
        \!+\!
        \|  
            e^{\tau_0(1+|D_x|)}  
            \bar b_1     
        \|_{H^{s+1,0}}
        \!+\!
        \|  
            e^{\tau_0(1+|D_x|)} \partial_y \bar b_1     
        \|_{H^{s,0}}
        + \\
        &+
        \|  e^{\tau_0(1+|D_x|)}  \tilde b_1   
        \|_{H^{s,0}}
        \leq 
        \delta :=
        \bigg(
        \frac{\min \{  1, \mathbb J, \kappa/{\rm Pr}_m\}^\frac{3}{2}}
        {\max \{  1, \mathbb J, \kappa/{\rm Pr}_m\}^\frac{5}{2}}
        \frac{
        \min\{\tau_0,\tau_0^{-1}\}^\frac{3}{2}}{\max\{1, \mathbb{H}^2\}\max\{{\rm Pr_m}^{-1}, {\rm Pr}_m\}^\frac{1}{2}}
        \bigg)
        \varepsilon_s,
\end{aligned}
\end{equation}
then there exists a global-in-time analytic solution $(u, b_1)$ of \eqref{eq:MHD-Prandtl-without-e}, which has a decaying radius of analyticity $\tau:\mathbb R_+ \to (0,\tau_0]$ given by
\begin{equation}\label{radius-of-analyticity-thm:existence-of-solutions}
\begin{aligned}
    \tau(t)  := 
        \tau_0
            \exp\Big\{ 
                -\frac{t}{16 \max\{1, \mathbb J, \kappa/{\rm Pr}_m\} }
        \Big\}>0.
\end{aligned}
\end{equation}
Furthermore,  the analytic norms of the solution decay exponentially in time $t\in \mathbb R_+$ as follows:
\begin{equation}\label{thm-analysis:main-inequality}
\begin{aligned}
    \| &e^{\tau(t)(1+|D_x|)} u(t)             \|_{H^{s+1,0}}^2 + 
    \| e^{\tau(t)(1+|D_x|)} \partial_t u(t)   \|_{H^{s,0}}^2+ 
    \| e^{\tau(t)(1+|D_x|)} \partial_y u(t)   \|_{H^{s,0}}^2 +
    \\
    &+
    \| e^{\tau(t)(1+|D_x|)} b_1(t)              \|_{H^{s+1,0}}^2 + 
    \| e^{\tau(t)(1+|D_x|)} \partial_t b_1(t)   \|_{H^{s,0}}^2+ 
    \| e^{\tau(t)(1+|D_x|)} \partial_y b_1(t)   \|_{H^{s,0}}^2 
    \\
    &
    \leq 
    C({\mathbb J, \kappa, {\rm Pr_m},\tau_0})
   \bigg\{
    \| e^{\tau_0(1+|D_x|)} \bar u               \|_{H^{s+1,0}}^2 + 
    \| e^{\tau_0(1+|D_x|)}  \tilde  u           \|_{H^{s,0}}^2+ 
    \| e^{\tau_0(1+|D_x|)} \partial_y\bar  u    \|_{H^{s,0}}^2 +
    \\
    &+
    \| e^{\tau_0(1+|D_x|)} \bar b               \|_{H^{s+1,0}}^2 + 
    \| e^{\tau_0(1+|D_x|)}  \tilde  b           \|_{H^{s,0}}^2+ 
    \| e^{\tau_0(1+|D_x|)} \partial_y\bar  b    \|_{H^{s,0}}^2 
    \bigg\}
    \exp\bigg\{-\frac{t}{8\max\{1, \mathbb J, \kappa/{\rm Pr_m}\}}
    \bigg\}
    .
\end{aligned}    
\end{equation}
where the constant $ C({\mathbb J, \kappa, {\rm Pr_m},\tau_0})$ is defined by
\begin{equation*}
    C({\mathbb J, \kappa, {\rm Pr_m},\tau_0})
    =
    4^3
    \frac{
        \max\{1,\mathbb J, \kappa/{\rm Pr_m}\}^3
    }
    {
        \min\{1,\mathbb J, \kappa/{\rm Pr_m}\}^3
    }
    \max\big\{{\rm Pr_m}, {\rm Pr_m^{-1}}\big\}
    \max \{\tau_0, \tau_0^{-1} \}^2
\end{equation*}
\end{theorem}
\noindent 
In what follows, we shall first describe the main idea for the proof and postpone the detailed estimates (that are rather involved) to the subsequent sections.

\subsection{Proof of Theorem \ref{thm:existence-analytic-solutions}}\label{sec:proof-of-main-thm}$\,$

\noindent
Our approach is entirely performed in terms of a-priori estimates. Indeed, without loss of generality, we can assume that the regular initial data generates a local-in-time analytic solution, whose largest lifespan is denoted by $T>0$. If $T<+\infty$, the norms on the left-hand side of \eqref{thm-analysis:main-inequality} would blow up, thus our aim is to prolong the smallness condition of the initial data to any time $t \in (0,T)$. This implies in particular that $T = +\infty$ and thus the local solution is in reality global in time.

\noindent
The core of the proof relies on a suitable transformation of the state variables $(u, b_1)$ in system \eqref{eq:MHD-Prandtl-without-e}, which aims to highlight the time behaviour of the underlying radius of analyticity. More precisely, for a general non decreasing function $\eta: \mathbb R_+:=[0,+\infty) \to [0, \tau_0]$ in $ C^2(\mathbb{R}_+)$, with $\eta(0) = 0$, we introduce the transformation $f\to f_\eta$ to a general function $f = f(t,x,y)$, by applying the Fourier multiplier $e^{(\tau_0-\eta(t))(1+|D_x|)}$ :
\begin{equation} \label{eq:ubphi}
	\begin{aligned}
		&f_{\eta}(t,x,y) := e^{(\tau_0-\eta(t))(1+|D_x|)} f(t,x,y), \quad
		\text{i.e.} \quad 
		\mathcal{F}_x(f_{\eta})(t,\xi,y) = 
		e^{(\tau_0-\eta(t)) (1+|\xi|)}
		\mathcal{F}_x(f)(t,\xi,y).
	\end{aligned}
\end{equation}
Here $\mathcal{F}_x$ stands for the Fourier transform uniquely on the variable $x\in \mathbb R$, while $\tau_0$ is the radius of analyticity of the initial data. 

\noindent
The Fourier multiplier and the related transformation \eqref{eq:ubphi} are well defined, as long as $f$ is analytic in $x\in \mathbb{R}$ for fixed $(t,y)\in (0,T) \times (0,1)$, with radius of analyticity given by $\tau_0 -\eta(t)$ (or larger). In particular this positive radius degrades (or stays constant) as time $t \in \mathbb R_+$ increases, since we restrict the function $\eta$ within the interval  $[0,\tau_0)$.

\noindent 
Our approach is to apply the mentioned transformation to both $u$ and $b_1$ and eventually to determine an optimal function $\eta$ in \eqref{eq:ubphi},  such that the new functions $u_\eta$ and $b_{1, \eta}$ fulfill a specific dissipative energy estimate (cf.~Proposition \ref{prop:the-overall-final-estimate-in-eta} and \eqref{def-lambda} for the final form of $\eta$). This energy  controls higher-order Sobolev norms of the transformed state variables  $u_\eta$ and $b_{1, \eta}$ and thus of the analytic norms of the solutions  $u$ and $b_{1}$ themselves. 

\noindent 
Roughly speaking, for a general function $\eta$, the transformation produces some additional damping mechanisms and dissipations to the system, but at the same time introduces further nonlinearities that could complicate the overall analysis. Our goal is therefore to select a suitable function $\eta$, so that the damping mechanism are indeed predominant. In this regime, a suitable ``high-order" energy occurs that allows to control the $H^{s,0}$-norms of $u_\eta$ and $b_{1, \eta}$ (or equivalently of $e^{(\tau_0-\eta(t))|D_x|} u$ and  $e^{(\tau_0-\eta(t))|D_x|} b_1$, for the analyticity).

\noindent
We proceed now to formalise the described strategy and we begin with, by stating the following proposition, that provides the mentioned energy inequality (with higher-order Sobolev norms) for a general function $\eta\in C^2(\mathbb R_+)$.
\begin{prop}\label{prop:the-overall-final-estimate-in-eta}
Denote by $\mathfrak{m}=\mathfrak{m}(\mathbb J, \kappa, {\rm Pr}_m )$, 
$\mathfrak{M}=\mathfrak{M}(\mathbb J, \kappa, {\rm Pr}_m )$ and by $\mathcal{R} =\mathcal{R} (\mathbb J, \kappa, {\rm Pr}_m )$ the following three parameters:
\begin{equation*}
    \mathfrak{m} 
    := 
    \min 
    \Big\{ 
        1,\mathbb J, \frac{\kappa}{{\rm Pr}_m} 
    \Big\},\qquad 
    \mathfrak{M} 
    := 
    \max 
    \Big\{ 
        1,\mathbb J, \frac{\kappa}{{\rm Pr}_m} 
    \Big\},
    \qquad 
    \R 
    :=
    \frac{1}{4\mathfrak{M}} = 
    \frac{1}{4\max \{ 1, \mathbb J, \kappa/{\rm Pr}_m\}},
    \qquad 
    .
\end{equation*}
There exists a constant $D_s\geq 1$, which depends uniquely on the regularity $s>2$, such that the following inequality holds true for any function $\eta \in C^2(\mathbb R)$, with $\eta(0) = 0$:
\begin{equation}\label{ineq-proposition-Es}
\begin{aligned} 
    &
    \underbrace{
    e^{\R t}
    \Big(
    \mathcal{E}_{s}
    +
    \mathfrak{m}\,
    \eta'
    \mathcal{E}_{s+\frac{1}{2}}
    +
    \mathfrak{m}^2
    (\eta')^2
    \mathcal{E}_{s+1}
    \Big) 
    (t)
    }_{\text{Energy}}
    \!+\!
    \underbrace{
    \int_0^t
    \!\!\!
    e^{\R \tilde t}
    \Big(
    \mathcal{D}_{ s}
    \!+\!
    \eta'
    \mathcal{D}_{s+\frac{1}{2}}
    \!+\!
    \mathfrak{m}
    (\eta')^2
    \mathcal{D}_{s+1}
    \!+\!
    \mathfrak{m}^2
    (\eta')^3
    \mathcal{D}_{s+\frac{3}{2}}
    \Big)
    (\tilde{t})
    d\tilde{t}
    }_{\text{Dissipation}}
    + \\
    &\!-\!
    \underbrace{
    \int_0^t
    \!\!\!
    e^{\R \tilde t}
    \bigg\{
    \!
    \Big(
        \mathbb J
        \!+\! 
        \frac{\kappa}{{\rm Pr}_m}
    \Big)
    \eta'' 
    \mathcal{E}_{s+\frac 12}
    \!\!+\!2
    \Big(
        \mathbb J^2
        \!+\!
        \frac{\kappa^2}{{\rm Pr}_m^2}
    \Big)\eta'\eta''
    \mathcal{E}_{s+1}
    \!
    \bigg\}
    (\tilde t)
    d\tilde t
    }_{\text{Rest due to the hyperbolicity of the system}}
    \!\leq \!
    \underbrace{
     \mathcal{E}_s(0)
     \!+\!
     \mathfrak{M} 
     \eta'(0)
     \mathcal{E}_{s+\frac{1}{2}}(0)
     \!+\!
     \mathfrak{M}^2
     \eta'(0)^2
     \mathcal{E}_{s+1}(0)
     }_{\text{Initial energy}}
     \!+\\ 
     &+
    D_s
    \max\{1,\mathbb H^2\}
    \Bigg\{
     \underbrace{
     \max
     \Big\{
     \frac{1}{\sqrt{\mathbb J}},
     \frac{{\rm Pr}_m}{\sqrt{\kappa}}
     \Big\}
     \int_0^t
     e^{\R \tilde t}
    \sqrt{\mathcal{E}_{s}(\tilde t)}
     \mathcal{D}_{ s+\frac{1}{2}}(\tilde{t})
     d\tilde t
    +
    \int_0^t
    e^{\R \tilde t}
    \mathcal{E}_{s}(\tilde t)
     \mathcal{D}_{ s+\frac{1}{2}}(\tilde{t})
     d\tilde t
    }_{\text{Estimate of the bilinear terms} }
    +\\
    &\hspace{7cm}+
    \underbrace{
    \max
     \Big\{
     \frac{1}{\sqrt{\mathbb J}},
     \frac{{\rm Pr}_m}{\sqrt{\kappa}}
     \Big\}
    \int_0^t
    e^{\R \tilde t}
    \mathcal{E}_{ s}(\tilde{t})
    \sqrt{ \mathcal{D}_{ s+\frac{1}{2}}(\tilde{t}) \mathcal{D}_{ s+\frac{3}{2}}(\tilde{t})}
     d\tilde t
    }_{\text{Estimate of the trilinear terms}}
    \Bigg\}.
\end{aligned}
\end{equation}
The functionals $\mathcal{E}_{s}$, $\mathcal{E}_{s+1/2}$ and 
$\mathcal{E}_{s+1}$ are defined in terms of the following Sobolev norms for the transformed solution $(u_\eta,b_{1,\eta})$:
\begin{equation}\label{def:Eetas}
    \begin{alignedat}{128}
        &\mathcal{E}_{s}
        \!:=\! 
        \frac{\mathbb J^2}{2}
        \| 
        (\partial_t u)_{\eta}   
        \|_{H^{s,0}}^2 
        \!+\!
        \frac{1}{2}
        \|  
        \mathbb J
        (\partial_t u)_{\eta}   
        \!+\!
        u_\eta
        \|_{H^{s,0}}^2 
        \!+\!
        \mathbb J
        \| 
        \partial_y u_{\eta}   
        \|_{H^{s,0}}^2 \\
        &\hspace{3.27cm}+
        \frac{\kappa^2}{2{\rm Pr_m}^2}
        \| 
        (\partial_t b_1)_{\eta}   
        \|_{H^{s,0}}^2 
        \!+\!
        \frac{1}{2}
        \Big\|  
        \frac{\kappa}{{\rm Pr_m}}
        (\partial_t b_1)_{\eta}   
        \!+\!
        b_{1,\eta}
        \Big\|_{H^{s,0}}^2 
        \!+\!
        \frac{\kappa}{{\rm Pr_m^2}}
        \| 
         \partial_y b_{1,\eta}   
        \|_{H^{s,0}}^2,
        \\
        &\mathcal{E}_{s+\frac 12} 
        := 
        \frac{1}{2}
        \| 
             u_{\eta}  
        \|_{H^{s+\frac{1}{2},0}}^2 
        +
        \frac{1}{2}
        \|  
             b_{1,\eta}   
        \|_{H^{s+\frac{1}{2},0}}^2,\qquad
        \mathcal{E}_{s+1} 
        := 
        \| 
            u_{\eta}   
        \|_{H^{s+1,0}}^2 
        +
        \| 
             b_{1,\eta}   
        \|_{H^{s+1,0}}^2.
\end{alignedat}
\end{equation} 
Furthermore the dissipative functionals $\mathcal{D}_{s}$, $\mathcal{D}_{s+1/2}$, 
$\mathcal{D}_{s+1}$ and $\mathcal{D}_{s+\frac{3}{2}}$ are defined by
\begin{equation}\label{def:Detas}
\begin{aligned}
        \mathcal{D}_{s}
        &:= 
        \frac{1}{2}        
        \big\| \partial_y u_{\eta}   \|_{H^{s,0}}^2
        +
        \frac{1}{2}
        \big\| (\partial_t u)_{\eta} \|_{H^{s,0}}^2
        +
        \frac{1}{2}
        \big\| \partial_y b_{1,\eta} \|_{H^{s,0}}^2
         +
        \frac{1}{2}
        \big\| (\partial_t b)_{1,\eta} \|_{H^{s,0}}^2,
        \\
        \begin{matrix}
        \mathcal{D}_{s+\frac{1}{2}}\\
        $\vspace{0.2cm}$
        \end{matrix}
        &
        \begin{matrix}
        \hspace{0.125cm}:=\\
        $\vspace{0.2cm}$
        \end{matrix}
        \hspace{-0.19cm}
        \begin{alignedat}{64}        
            &\frac{1}{2}
            \| 
            \,(\partial_t u)_{\eta}   
            &&
            \|_{H^{s+\frac{1}{2},0}}^2 
            \!+\!
            \frac{1}{2}
            \|  
            \,\,(\partial_t u)_{\eta}   
            \!+\!
            u_\eta
            &&&&
            \|_{H^{s+\frac{1}{2},0}}^2 
            \!+\!
            2
            \| 
            \,\partial_y u_{\eta}   
            &&&&&&&& 
            \|_{H^{s+\frac{1}{2},0}}^2 
            \!+\!
            \| 
            \,
            \partial_t(u_{\eta})   
            &&&&&&&&&&&&&&&&
            \|_{H^{s+\frac{1}{2},0}}^2 
            \!+\!
            \frac{3}{8}
            \| 
            \,u_{\eta}   
            &&&&&&&&&&&&&&&&&&&&&&&&&&&&&&&&    
            \|_{H^{s+\frac{1}{2},0}}^2 
            \!+\!
            \\ 
            \!+
            &
            \frac{1}{2}
            \| 
            (\partial_t b_1)_{\eta}   
            &&
            \|_{H^{s+\frac{1}{2},0}}^2 
            \!+\!
            \frac{1}{2}
            \|  
            (\partial_t b_1)_{\eta}   
            \!+\!
            b_{1,\eta}
            &&&&
            \|_{H^{s+\frac{1}{2},0}}^2 
            \!+\!
            2
            \| 
             \partial_y b_{1,\eta}   
            &&&&&&&&
            \|_{H^{s+\frac{1}{2},0}}^2 +
            \| 
             \partial_t(b_{1,\eta})  
            &&&&&&&&&&&&&&&&
            \|_{H^{s+\frac{1}{2},0}}^2 
            \!+\!
            \frac{3}{8}
            \| 
                b_{1,\eta}   
            &&&&&&&&&&&&&&&&&&&&&&&&&&&&&&&&      
            \|_{H^{s+\frac{1}{2},0}}^2,
        \end{alignedat}\\
        \mathcal{D}_{s+1}
        &:= 
        \frac{3}{4}
        \| 
             u_{\eta}  
        \|_{H^{s+\frac{1}{2},0}}^2 
        +
        \frac{3}{4}
        \|  
         \partial_t b_{1,\eta}   
        \|_{H^{s+\frac{1}{2},0}}^2,\qquad
        \mathcal{D}_{s+\frac{3}{2}}(t)
        := 
        \| 
            u_{\eta}   
        \|_{H^{s+\frac{3}{2},0}}^2 
        +
        \| 
             b_{1,\eta}   
        \|_{H^{s+\frac{3}{2},0}}^2.
\end{aligned}
\end{equation} 

\end{prop}
\noindent 
Since the proof of this Proposition is rather technical, we postpone it to the forthcoming sections and we focus this paragraph to the remaining steps to prove Theorem \ref{thm:existence-analytic-solutions}. We shall however first provide some remarks on the main inequality \eqref{ineq-proposition-Es}, and highlight in particular the dissipative mechanisms due to $\eta$, as well as the more challenging terms, that we are indeed left to estimate.   

\noindent 
An explicit relation on the  constant $D_s\geq 1$ is formally determined later on (cf.~\eqref{Ds-explicit-determined}), which we shall nevertheless here outline :
\begin{equation}\label{Ds-explicit-not-proven}
    D_s = \frac{2^{2s+6}}{s-2}
    \left(
        1+\frac{s-2}{\sqrt{s-1}}
    \right)\geq 1. 
\end{equation}
By assuming that Proposition \ref{prop:the-overall-final-estimate-in-eta} holds true, the proof of Theorem \ref{thm:existence-analytic-solutions} follows with some straightforward steps. Indeed, our main goal is to determine a suitable function $\eta\in \mathcal{C}^2(\mathbb R_+)$ in \eqref{ineq-proposition-Es} and a small parameter $\varepsilon_s>0$ for the initial condition \eqref{thm:small-condition}, that ensure the following relations:
\begin{enumerate}[(a)]
    \item the terms on the left-hand side of \eqref{ineq-proposition-Es} are all non-negative and thus ``support'' the $H^{s,0}$-energy inequality,
    \item the right-hand side of \eqref{ineq-proposition-Es} can eventually be absorbed by some of the positive terms of the left-hand side, under a suitable smallness condition on the initial data.
\end{enumerate}
For what concerns part $(a)$, the only term that (for a general $\eta$) could reach negative values is the first integral at the second line of  \eqref{ineq-proposition-Es}, namely
\begin{equation}\label{integral-related-to-eta''}
    -
    \int_0^t
    e^{\R \tilde t}
    \bigg\{
    \Big(
        \mathbb J
        \!+\! 
        \frac{\kappa}{{\rm Pr}_m}
    \Big)
    \eta'' 
    \mathcal{E}_{s+\frac 12}
    +2
    \Big(
        \mathbb J^2
        \!+\!
        \frac{\kappa^2}{{\rm Pr}_m^2}
    \Big)\eta'\eta''
    \mathcal{E}_{s+1}
    \bigg\}
    (\tilde t)
    d\tilde t.
\end{equation}
The sign of this integral is entangled with the sign of the weights $-\eta''(t)$ and $-2\eta'(t)\eta''(t)$, $t \in \mathbb{R}$. It is natural therefore to calibrate the function $\eta\in \mathcal{C}^2(\mathbb R)$, in such a way that this integral provides a positive dissipation or at least vanishes. In other words, we shall seek for a function $\eta \in \mathcal{C}^2(\mathbb R_+)$ such that
\begin{equation}\label{cond-positivity-eta''}
     \tau_0-\eta(t)>0,\quad 
     \eta'(t)\geq 0,\quad 
     -\eta''( t)\geq 0,\qquad 
     \text{for any}\quad t\in \mathbb{R}_+.
\end{equation}
Among the several functions satisfying \eqref{cond-positivity-eta''}, we consider a specific family of the form $\eta(t):= \tau_0(1-e^{- \lambda t})$, where $\lambda$ is (momentarily) an arbitrary positive constant (the exact value of $\lambda$ for our analysis will be shortly be determined in \eqref{def-lambda}). Indeed, we remark that
\begin{equation*}
     \tau_0-\eta(t)
     = \tau_0e^{-\lambda t}>0 ,
     \quad 
     \eta'(t) =
     \lambda 
     \tau_0e^{- \lambda t} >0,
     \quad 
     -\eta''(t) = 
     \lambda^2 
     \tau_0e^{- \lambda t} >0,
     \qquad 
     \text{for any}\quad t\in \mathbb{R}_+.
\end{equation*}
In doing so, we can recast the main inequality \eqref{ineq-proposition-Es} into
\begin{equation}\label{main-estimate-with-lambda}
\begin{aligned} 
    &e^{\R t}
    \mathcal{E}_{s}(t)
    +
    \mathfrak{m}
    \lambda 
    \tau_0
    e^{(\R -\lambda)t}
    \mathcal{E}_{s+\frac{1}{2}}
    +
    \mathfrak{m}^2
    \lambda^2
    \tau_0^2
    e^{(\R -2\lambda) t}
    \mathcal{E}_{s+1}
    (t)
    \!+\!
    \int_0^t
    \!\!\!
    e^{\R \tilde t}
    \mathcal{D}_{ s}
    (\tilde t)
    d\tilde t
    \!+\!
    \tau_0
    \lambda
    \int_0^t
    e^{ (\R -\lambda )\tilde  t}
    \mathcal{D}_{s+\frac{1}{2}}
    (\tilde t)
    d\tilde t
    +\\
    &+
    \mathfrak{m}
    \tau_0^2
    \lambda^2
    \int_0^t
    \!\!\!
    e^{(\R - 2\lambda) \tilde  t} 
    \mathcal{D}_{s+1}
    (\tilde t)
    d\tilde t
    +
    \mathfrak{m}^2
    \tau_0^3
    \lambda^3
    \int_0^t
    \!\!\!
    e^{(\R - 3\lambda) \tilde  t}     
    \mathcal{D}_{s+\frac{3}{2}}
    (\tilde{t})
    d\tilde{t} 
    \leq 
    \mathcal{E}_s(0)
    +
    \mathfrak{M}
    \tau_0^2
    \lambda
    \mathcal{E}_{s+\frac{1}{2}}(0)
    +\\
    &+
    \mathfrak{M}^2
    \tau_0
    \lambda^2
    \mathcal{E}_{s+1}(0)
    +
     D_s
     \max\big\{1,\mathbb{H}^2\big\}
     \Bigg\{
     \max
     \bigg\{
     \frac{1}{\sqrt{\mathbb J}},
     \frac{{\rm Pr}_m}{\sqrt{\kappa}}
     \bigg\}
     \int_0^t
     e^{\R \tilde t}
     \sqrt{\mathcal{E}_{s}(\tilde t)}
     \mathcal{D}_{ s+\frac{1}{2}}(\tilde{t})
     d\tilde t
    +\\&
    \hspace{2cm}+
    \int_0^t
    e^{\R \tilde t}
    \mathcal{E}_{s}(\tilde t)
    \mathcal{D}_{ s+\frac{1}{2}}(\tilde{t})
     d\tilde t
    +
    \max
     \bigg\{
     \frac{1}{\sqrt{\mathbb J}},
     \frac{{\rm Pr}_m}{\sqrt{\kappa}}
     \bigg\}
    \int_0^t
    e^{\R \tilde t}
    \mathcal{E}_{s}(\tilde t)
    \sqrt{ \mathcal{D}_{ s+\frac{1}{2}}(\tilde{t}) \mathcal{D}_{ s+\frac{3}{2}}(\tilde{t})}
     d\tilde t
     \Bigg\},
\end{aligned}
\end{equation}
where we have dropped the integral \eqref{integral-related-to-eta''} on the left-hand side of the inequality, since it is positive and furthermore it does not support the next steps of our analysis. Our goal is therefore to determine a suitable $\lambda>0$ in \eqref{main-estimate-with-lambda} and a suitable parameter $\varepsilon_s>0$ at the smallness condition \eqref{thm:small-condition} of the initial data, such that all integrals in the third line of \eqref{main-estimate-with-lambda} can be absorbed by the dissipative terms of the inequality. Consequently, we shall first reformulate these integrals in accordance with the dissipative terms. First
\begin{equation}\label{estimate-of-the-first-two-integrals}
\begin{alignedat}{4}
    &
     \int_0^t
     e^{\R \tilde t}
     \sqrt{\mathcal{E}_{s}(\tilde t)}
     \mathcal{D}_{ s+\frac{1}{2}}(\tilde{t})
     d\tilde t
     &&\leq 
    \sup_{\tilde t\in (0,t)}
     \Big(
        e^{2\lambda\tilde  t}\mathcal{E}_{s}(\tilde t)
    \Big)^\frac{1}{2}
     \int_0^t
     e^{(\R  -\lambda )\tilde t}
     \mathcal{D}_{ s+\frac{1}{2}}(\tilde{t})
     d\tilde t,\\
     &\int_0^t
      e^{\R \tilde t}
     \mathcal{E}_{s}(\tilde t)
     \mathcal{D}_{ s+\frac{1}{2}}(\tilde{t})
     d\tilde t
     &&\leq 
     \sup_{\tilde t\in (0,t)}
     \Big(
        e^{\lambda t}\mathcal{E}_{s}(\tilde t)
    \Big)
     \int_0^t
     e^{(\R  -\lambda )\tilde t}
     \mathcal{D}_{ s+\frac{1}{2}}(\tilde{t})
     d\tilde t.
\end{alignedat}
\end{equation}
We hence remark that we can bound $ e^{2\lambda t}\mathcal{E}_{s}(\tilde t)$ and $e^{\lambda t}\mathcal{E}_{s}(\tilde t)$ with the energy $e^{\R t}\mathcal{E}_s(t)$ in \eqref{main-estimate-with-lambda}, as long as $\lambda>0$ is considered within the range  $\lambda\leq \R/2 = 1/(8\max\{ 1, \mathbb J, \kappa/{\rm Pr}_m\})$.
The major difficulties arise however from the last integrals at the fourth line of \eqref{main-estimate-with-lambda}, since this term involves the dissipation $\mathcal{D}_{s+3/2}$, which has indeed the highest Sobolev regularity. We deal with this integral, by observing that
\begin{equation}\label{estimate-of-the-third-integral} 
\begin{aligned}
     \max
     \bigg\{
    & \frac{1}{\sqrt{\mathbb J}},
     \frac{{\rm Pr}_m}{\sqrt{\kappa}}
     \bigg\}
    \int_0^t
    e^{\R \tilde t}\mathcal{E}_{s}(\tilde t)
    \sqrt{ \mathcal{D}_{ s+\frac{1}{2}}(\tilde{t}) 
    \mathcal{D}_{ s+\frac{3}{2}}(\tilde{t})}
    d\tilde t
    = \\
    &=
    \max
     \bigg\{
     \frac{1}{\sqrt{\mathbb J}},
     \frac{{\rm Pr}_m}{\sqrt{\kappa}}
     \bigg\}
    \int_0^t
    e^{2\lambda \tilde t}\mathcal{E}_{s}(\tilde t)
    \sqrt{e^{(\R- \lambda) \tilde t} \mathcal{D}_{ s+\frac{1}{2}}(\tilde{t}) }
    \sqrt{
        e^{(\R- 3\lambda) \tilde t} 
        \mathcal{D}_{ s+\frac{3}{2}}(\tilde{t})
    }
    d\tilde t,\\
    &\leq
    \sup_{\tilde t\in (0,t)}
     \Big(
        e^{2\lambda t}\mathcal{E}_{s}(\tilde t)
    \Big)
    \Bigg\{
    \int_0^t
    e^{(\R- \lambda) \tilde t} 
    \mathcal{D}_{ s+\frac{1}{2}}(\tilde{t}) 
    d\tilde t
    +
    \frac{1}{4}
    \max
     \bigg\{
     \frac{1}{ \mathbb J},
     \frac{{\rm Pr}_m^2}{\kappa}
     \bigg\}
    \int_0^t
        e^{(\R- 3\lambda) \tilde t} 
        \mathcal{D}_{ s+\frac{3}{2}}(\tilde{t})
    d\tilde t
    \Bigg\},
\end{aligned}
\end{equation}
where we can still bound $ e^{2\lambda t}\mathcal{E}_{s}(\tilde t)$ with the energy $e^{\R t}\mathcal{E}_s(t)$, as long as $\lambda\leq \R/2$.

\noindent 
We are now in the condition to set a specific value of $\lambda$, namely
\begin{equation}\label{def-lambda}
    \lambda := \frac{\mathcal R}{4} 
     = \frac{1}{16\max \{ 1, \mathbb J, \kappa/{\rm Pr}_m\}},
     \quad \text{namely}
     \quad 
     \eta(t) = \tau_0\Big(1 - e^{-\frac{t}{16\max \{ 1, \mathbb J, \kappa/{\rm Pr}_m\}}}\Big)
\end{equation}
(we do not consider the threshold $\lambda = \R/2$, since the exponential function in front of $\mathcal{E}_{s+1}$ in \eqref{main-estimate-with-lambda} would in that case vanish, not allowing us to derive an exponential decay of the related norms). In particular, coupling \eqref{estimate-of-the-first-two-integrals} and \eqref{estimate-of-the-third-integral} together with \eqref{main-estimate-with-lambda}, we obtain
\begin{equation}\label{final-estimate-lambda=1/12}
\begin{aligned} 
    &e^{\R t}
    \mathcal{E}_{s}(t)
    +
    \frac{\mathfrak{m}\R \tau_0 }{4}
    e^{\frac{3}{4}\R t}
    \mathcal{E}_{s+\frac{1}{2}}(t)
    +
    \frac{ \mathfrak{m}^2\R^2\tau_0}{4^2}
    e^{\frac{\R}{2}t}
    \mathcal{E}_{s+1}(t)
    +\\ 
    &+
    \bigg(
    \frac{\tau_0\R}{4}
    -
    D_s
    \max\big\{1,\mathbb{H}^2\big\}
    \sup_{\tilde t \in (0,t)}
    \bigg\{
    \max
     \bigg\{
     \frac{1}{\sqrt{\mathbb J}},
     \frac{{\rm Pr}_m}{\sqrt{\kappa}}
     \bigg\}
    \Big(
     e^{\R \tilde t}\mathcal{E}_{s}(\tilde t)
     \Big)^\frac{1}{2}
     \!\!\!+
     2
     e^{\R \tilde t}\mathcal{E}_{s}(\tilde t)
     \bigg\}
    \bigg)
    \int_0^t
    e^{\frac{3\R}{4} \tilde t }
    \mathcal{D}_{s+\frac{1}{2}}
    (\tilde t)
    d\tilde t
    +\\ 
    &+\!
    \bigg(
    \!
    \frac{\mathbf m^2\tau_0^3\R^3}{4^3}
    \!-\!
    \frac{D_s}{4}  \max\big\{1,\!\mathbb{H}^2\big\}
    \max
     \bigg\{
     \frac{1}{\mathbb J},
     \frac{{\rm Pr}_m^2}{ \kappa}
     \bigg\}
    \sup_{\tilde t \in (0,t)}
    \!\! e^{\R t}\mathcal{E}_{s}(\tilde t)
    \!
    \bigg)
    \int_0^t
    \!\!\!
    e^{\frac{\R}{4}\tilde t}     
    \mathcal{D}_{s+\frac{3}{2}}
    (\tilde{t})
    d\tilde{t} 
    \\ 
    &\leq 
     \mathcal{E}_{s}(0)
    \!+\!
    \frac{\mathfrak{M}\R \tau_0 }{4}
   \mathcal{E}_{s+\frac{1}{2}}(0)
    \!+\!
    \frac{ \mathfrak{M}^2\R^2\tau_0^2}{4^2}
    \mathcal{E}_{s+1}(0),
\end{aligned}
\end{equation}
where we have omitted to write the positive dissipative integrals in $\mathcal{D}_s$ and $\mathcal{D}_{s+1}$ on the left-hand side.
Now, we remark that $\mathcal{E}_{s}(0)
    \!+\!
    \frac{\mathfrak{M}\R \tau_0 }{4}
   \mathcal{E}_{s+\frac{1}{2}}(0)
    \!+\!
    \frac{ \mathfrak{M}^2\R^2\tau_0^2}{4^2}
    \mathcal{E}_{s+1}(0)$  can be estimated in terms of the initial data and the smallness condition \eqref{thm:small-condition}. Indeed, recalling the definition of $\mathcal{E}_s$, $\mathcal{E}_{s+1/2}$ and  $\mathcal{E}_{s+1}$ in \eqref{def:Eetas}, we observe that
\begin{align*}
    \mathcal{E}_s(0) 
    &\leq 
    \max 
    \Big\{
        1, \mathbb J^2, \Big(\frac{\kappa}{{\rm Pr_m}}\Big)^2
    \Big\}
    \max 
    \Big\{
        1,\frac{1}{{\rm Pr_m}}
    \Big\}
    \bigg(
        \frac{3}{2}
        \| \tilde u \|_{H^{s,0}}^2+ 
        \frac{1}{2}
        \|\bar u \|_{H^{s,0}}^2+ \\
        &\hspace{3cm}
        +
        \|\partial_y \bar u \|_{H^{s,0}}^2+
        \frac{3}{2}
        \| \tilde b_1 \|_{H^{s,0}}^2+ 
        \frac{1}{2}
        \|\bar b_1 \|_{H^{s,0}}^2+ 
        \|\partial_y \bar b_1 \|_{H^{s,0}}^2
    \bigg)\\
    &\leq 
    2
    \max 
    \Big\{
        1, \mathbb J, \frac{\kappa}{{\rm Pr_m}}
    \Big\}^2
    \max 
    \Big\{
        1,\frac{1}{{\rm Pr_m}}
    \Big\}
    \Big(
        \| \tilde u \|_{H^{s,0}}+ 
        \|\bar u \|_{H^{s+1,0}}+ \\
        &\hspace{3cm}
        +
        \|\partial_y \bar u \|_{H^{s,0}}+
        \| \tilde b_1 \|_{H^{s,0}}+ 
        \|\bar b_1 \|_{H^{s+1,0}}+ 
        \|\partial_y \bar b_1 \|_{H^{s,0}}
    \Big)^2\\
    &\leq 
    2
    \max 
    \Big\{
        1, \mathbb J ,  \frac{\kappa}{{\rm Pr_m}} 
    \Big\}^2
    \max 
    \Big\{
        1,\frac{1}{{\rm Pr_m}}
    \Big\}
    \delta^2,
\end{align*}
where we recall that $\delta>0 $ is the small parameter bounding the norms of the initial data and it is defined in \eqref{thm:small-condition} by
\begin{equation*}
    \delta :=
        \bigg(
        \frac{\min \{  1, \mathbb J, \kappa/{\rm Pr}_m\}^\frac{3}{2}}
        {\max \{  1, \mathbb J, \kappa/{\rm Pr}_m\}^\frac{5}{2}}
        \frac{
        \min\{\tau_0,\tau_0^{-1}\}^\frac{3}{2}}{\max\{1, \mathbb{H}^2\}\max\{{\rm Pr_m^{-1}}, {\rm Pr}_m\}^\frac{1}{2}}
        \bigg)
        \varepsilon_s,
\end{equation*}
for a small parameter $\varepsilon_s$ that depends only on $s>2$ (and that we have not determined, yet).
Similarly, we have that $
    \mathcal{E}_{s+1/2}(0)
    \leq (1/2)\delta^2$ and $ 
    \mathcal{E}_{s+1}(0)
    \leq 
    \delta^2$,
Hence, we obtain the following estimate of the right-hand side in \eqref{final-estimate-lambda=1/12}:
\begin{equation}\label{estimate-E(0)}
\begin{aligned}
    \mathcal{E}_{s}(0)
    &+
    \frac{\mathfrak{M}\R \tau_0 }{4}
    \mathcal{E}_{s+\frac{1}{2}}(0)
    +
    \frac{ \mathfrak{M}^2\R^2\tau_0^2}{4^2}
    \mathcal{E}_{s+1}(0)
    \leq 
    \mathcal{E}_s(0)
    +
    \frac{\tau_0}{4^2}
    \mathcal{E}_{s+\frac{1}{2}}(0)
    +
    \frac{\tau_0^2}{4^4}
    \mathcal{E}_{s+1}(0)\\
    &\leq 
    \max\{1, \tau_0\}^2
    \bigg(
        2
    \max 
    \Big\{
        1, \mathbb J, \frac{\kappa}{{\rm Pr_m}}
    \Big\}^2
    \max 
    \Big\{
        1,\frac{1}{{\rm Pr_m}}
    \Big\}
    \delta ^2 + \frac{1}{4^2}\delta ^2 + \frac{1}{4^4}\delta ^2 
    \bigg)
    \\
    &\leq 
    4\max\{1, \tau_0\}^2
    \max 
    \Big\{
        1, \mathbb J, \frac{\kappa}{{\rm Pr_m}}
    \Big\}^2
    \max 
    \Big\{
        1,\frac{1}{{\rm Pr_m}}
    \Big\}
    \delta^2\\
    &\leq 
    \frac{
        4 
        \max\{1, \tau_0\}^2 
        \min\{\tau_0,\tau_0^{-1}\}^3\min\{1, \mathbb J, \kappa/{\rm Pr_m}\}^3}
        {
        \max\{1, \mathbb H^2\}^2
        \max\{1, \mathbb J, \kappa/{\rm Pr_m}\}^3
        \max\{{\rm Pr_m^{-1}}, {\rm Pr_m}\}
        \min\{1, {\rm Pr_m}\}}
     \varepsilon_s^2\\
    &\leq 
    \frac{
        4 
        \min\{1, \tau_0\}^3\min\{1, \mathbb J, \kappa/{\rm Pr_m}\}^3}
        {
        \max\{1, \mathbb H^2\}^2
        \max\{1, \mathbb J, \kappa/{\rm Pr_m}\}^3
        \max\{1, {\rm Pr_m}\}}
     \varepsilon_s^2.
    \end{aligned}
\end{equation}
where, in the last line, we have used the identity $ \max\{{\rm Pr_m^{-1}}, {\rm Pr_m}\}
        \min\{1, {\rm Pr_m}\} = \max\{ 1,  {\rm Pr_m}\} $, as well as the inequality
$\max\{1, \tau_0\} ^2
        \min\{\tau_0,\tau_0^{-1}\}^3\leq \min\{1,\tau_0\}^2$.
Inequality \eqref{final-estimate-lambda=1/12} together with \eqref{estimate-E(0)} allow us to conclude the proof by means of a bootstrap method. Indeed, by introducing the small parameter $\varepsilon_s>0$ in \eqref{thm:small-condition} and a maximal time $T^*\in (0,T)$, such that
\begin{equation}\label{def-esp_s-T*}
\begin{aligned}
    \varepsilon_s^2 
    &:= \frac{1}{ 2^{16}  D_s^2}
    \qquad
    \bigg(
    \text{thus }\varepsilon_s := 
      \frac{s-2}{2^{2s+14}}
      \frac{1}{1+\frac{s-2}{\sqrt{s-1}}},
      \text{ because of \eqref{Ds-explicit-not-proven}}
    \bigg)  
    ,
    \\
    T^* 
    &:= 
    \sup 
    \bigg\{t\in (0,T)\,:\, e^{\R t}\mathcal{E}_s(t) 
    \leq 
    \frac{1}{2^{12}}
     \frac{\min\{1, \tau_0\}^3\min\{1, \mathbb J, \kappa/{\rm Pr_m}\}^3}{D_s^2 \max\{1, \mathbb H^2\}^2
     \max\{1, \mathbb J, \kappa/{\rm Pr_m}\}^3\max\{1, {\rm Pr_m}\}}<1
    \bigg\},
\end{aligned}    
\end{equation}
we have that, for any time $t \in (0,T^*)$, the constants in front of the dissipative terms of \eqref{final-estimate-lambda=1/12} are indeed positive:
\begin{align*}
    &
    \frac{\tau_0\R}{4}
    -
    D_s
    \max\big\{1,\mathbb{H}^2\big\}
    \sup_{\tilde t \in (0,t)}
    \bigg\{
    \max
     \bigg\{
     \frac{1}{\sqrt{\mathbb J}},
     \frac{{\rm Pr}_m}{\sqrt{\kappa}}
     \bigg\}
    \Big(
     e^{\R \tilde t}\mathcal{E}_{s}(\tilde t)
     \Big)^\frac{1}{2}
     \!\!\!+
    2
     e^{\R \tilde t}\mathcal{E}_{s}(\tilde t)
     \bigg\}=\\
     &=
    \frac{\tau_0}{4^2\max\{1, \mathbb J, \kappa/{\rm Pr_m}\}}
    -
    D_s
    \max\big\{1,\mathbb{H}^2\big\}
    \sup_{\tilde t \in (0,t)}
    \bigg\{
     \Big(
     e^{\R \tilde t}\mathcal{E}_{s}(\tilde t)
     \Big)^\frac{1}{2}
    \bigg(
    \max
     \bigg\{
     \frac{1}{\sqrt{\mathbb J}},
     \frac{{\rm Pr}_m}{\sqrt{\kappa}}
     \bigg\}
     +
    2
    \Big(
     e^{\R \tilde t}\mathcal{E}_{s}(\tilde t)
     \Big)^\frac{1}{2}
     \bigg)
     \bigg\}\\
     &\geq 
    \frac{\tau_0}{4^2\max\{1, \mathbb J, \kappa/{\rm Pr_m}\}}
    -
    D_s
    \max\big\{1,\mathbb{H}^2\big\}
    \sup_{\tilde t \in (0,t)}
    \bigg\{
     \Big(
     e^{\R \tilde t}\mathcal{E}_{s}(\tilde t)
     \Big)^\frac{1}{2}
    \bigg\}
    \bigg(
    \max
     \bigg\{
     \frac{1}{\sqrt{\mathbb J}},
     \frac{{\rm Pr}_m}{\sqrt{\kappa}}
     \bigg\}
     +
     2
    \bigg)\\
     &\geq 
    \frac{\tau_0}{4^2\max\{1, \mathbb J, \kappa/{\rm Pr_m}\}}
    -
    D_s
    \max\big\{1,\mathbb{H}^2\big\}
    \sup_{\tilde t \in (0,t)}
    \bigg\{
     \Big(
     e^{\R \tilde t}\mathcal{E}_{s}(\tilde t)
     \Big)^\frac{1}{2}
    \bigg\}
   \frac{4\max\{1, {\rm Pr_m}\}^\frac{1}{2}}{\min\{1,  {\mathbb J},  {\kappa/{\rm Pr_m}}\}^\frac{1}{2}}\\
     &\geq 
    \frac{\tau_0}{4^2\max\{1, \mathbb J, \kappa/{\rm Pr_m}\}}
    -
     \frac{
     D_s
    \max\big\{1,\mathbb{H}^2\big\}
     \min\{1, \tau_0\}^\frac{3}{2}\min\{1, \mathbb J, \kappa/{\rm Pr_m}\}^\frac{3}{2}
     }
     {
     2^{6}
     D_s \max\{1, \mathbb H^2\}
     \max\{1, \mathbb J, \kappa/{\rm Pr_m}\}^\frac{3}{2}
     \max\{1, {\rm Pr_m}\}^\frac{1}{2}
     }
   \frac{4\max\{1, {\rm Pr_m}\}^\frac{1}{2}}{\min\{1,  {\mathbb J},  {\kappa/{\rm Pr_m}}\}^\frac{1}{2}}
   \\
    &\geq 
    \frac{\min\{1, \tau_0\}}{\max\{1, \mathbb J, \kappa/{\rm Pr_m}\}}
    \bigg(
    \frac{1}{4^2}
    -
    \frac{1}{2^{4}}
    \frac{\min\{1, \tau_0\}^\frac{1}{2}\min\{1, \mathbb J, \kappa/{\rm Pr_m}\} }{
     \max\{1, \mathbb J, \kappa/{\rm Pr_m}\}^\frac{1}{2}}
     \bigg)
     \geq 0
\end{align*}
and
\begin{align*}
    &\frac{\mathbf m^2\tau_0^3\R^3}{4^3} 
    \!-\!
    \frac{D_s}{4}  \max\big\{1,\!\mathbb{H}^2\big\}
    \max
    \bigg\{\frac{1}{\mathbb J}, \frac{\rm Pr_m^2}{\kappa}
    \bigg\}
    \sup_{\tilde t \in (0,t)}
    \!\! e^{\R t}\mathcal{E}_{s}(\tilde t)
    \\
    &
    \geq 
    \frac{\tau_0^3\min\{1, \mathbb J, \kappa/{\rm Pr_m}\}^2}{4^6\max\{1, \mathbb J, \kappa/{\rm Pr_m}\}^3}
    \!-\!
    \frac{D_s\max\big\{1,\!\mathbb{H}^2\big\}\max\big\{1,\! {\rm Pr_m}\big\}}{4\min\{1, \mathbb J, \kappa/{\rm Pr_m}\}}
     \frac{\min\{1, \tau_0\}^3\min\{1, \mathbb J, \kappa/{\rm Pr_m}\}^3}{2^{12}D_s^2 \max\{1, \mathbb H^2\}^2
     \max\{1, \mathbb J, \kappa/{\rm Pr_m}\}^3\max\{1, {\rm Pr_m}\}}
    \\
    &\geq 
    \frac{\tau_0^3\min\{1, \mathbb J, \kappa/{\rm Pr_m}\}^2}{4^6\max\{1, \mathbb J, \kappa/{\rm Pr_m}\}^3}
    -  
    \frac{
        \min\{1, \tau_0\}^3\min\{1, \mathbb J, \kappa/{\rm Pr_m}\}^2
    }
    {
        2^{14}\max\{1, \mathbb J, \kappa/{\rm Pr_m}\}^3
    }
    \geq 0.
\end{align*}
Hence the energy inequality \eqref{final-estimate-lambda=1/12} and the estimate \eqref{estimate-E(0)} at initial time $t=0$ imply that the functional $e^{t/4}\mathcal{E}_{s}(t)$ stays small for any $t\in (0,T^*)$:
\begin{align*} 
     e^{\R t}
    \mathcal{E}_{s}(t)
    \leq 
    \frac{1}{2^{14}}
     \frac{\min\{1, \tau_0\}^3\min\{1, \mathbb J, \kappa/{\rm Pr_m}\}^3}{
     D_s^2\max\{1, \mathbb H^2\}^2
     \max\{1, \mathbb J, \kappa/{\rm Pr_m}\}^3\max\{1, {\rm Pr_m}\}}
    .
\end{align*}
From the definition of $T^*$ in \eqref{def-esp_s-T*}, we finally deduce that $T^* = T$, which must be $\infty$ since the analytic norm does not blow up at this time. Accordingly, the solution $(u,b_1)$ 
is indeed global in time.

\noindent 
Finally, inequality \eqref{thm-analysis:main-inequality} about the exponential decay of the norms of the solution follows directly from the estimate \eqref{final-estimate-lambda=1/12}, which implies 
\begin{equation*}
\begin{aligned}
    e^{\frac{\R}{2} t}
    \min
    &\Big\{
        1, \frac{\mathfrak{m}\R}{4}
    \Big\}^2
    \min\{1, \tau_0\}
    \Big(
     \mathcal{E}_s(t) + 
     \mathcal{E}_{s+1/2}(t)
     +
     \mathcal{E}_{s+1}(t)
    \Big)
    \leq \\
    &\leq
    \max\{1, \tau_0\}^2
    \Big(
     \mathcal{E}_s(0) + 
     \mathcal{E}_{s+1/2}(0)
     +
     \mathcal{E}_{s+1}(0)
    \Big).
\end{aligned}
\end{equation*}
We hence achieve \eqref{thm-analysis:main-inequality}, by manipulating $\mathfrak{m} = \min \{ \mathbb J, \kappa/{\rm Pr_m} \}$ and $\R = 1/(4\max\{1, \mathbb J, \kappa/{\rm Pr_m} \})$, as well as by recalling the definition of the functions $ \mathcal{E}_s$ $\mathcal{E}_{s+1/2}$ and  $\mathcal{E}_{s+1}$ in \eqref{def:Eetas}.
This concludes the proof of Theorem~\ref{thm:existence-analytic-solutions}. 

\section{Proof of Proposition \ref{prop:the-overall-final-estimate-in-eta}}
\label{sec:proof-proposition-est}$\,$

\noindent
The core of our approach being showed, it remains to prove the ``high-order'' energy estimate described by Proposition \ref{prop:the-overall-final-estimate-in-eta}. 

\subsection{Estimates related to the equation of $u$}$\,$

\noindent 
In this section we deal with the momentum equation of $u$ in the main system \eqref{eq:MHD-Prandtl-without-e}, which satisfies
\begin{equation}\label{eq:u-sec:final-estimates}
\begin{aligned}
    \mathbb{J}
        \partial_{tt} u
        +
        \partial_t u + 
        u \partial_x u +
        v \partial_y u - 
        \partial_{yy}^2 u + \partial_x p = 
        \mathbb{H}^2
        \Big\{
        b_1 
        b_2
        v  - 
        u 
        b_2^2
        +b_2
        \big( 
            \smallint_0^y
            \partial_t b_1 
        \big)
        \Big\}.
\end{aligned}
\end{equation}
Recalling the definition of $u_{\eta}(t, \cdot)  = e^{(\tau_0-\eta(t) )(1+|D_x|)}u(t,\cdot)$ and the value  $\mathcal{R} =
1/ (4\max\{1, \mathbb J, \kappa/{\rm Pr_m}\})$, we remark that the function $e^{\mathcal{R}t/2} u_{\eta}(t,x,y) $ is solution of
\begin{equation}\label{eq:ueta}
\begin{aligned}
    \mathbb{J}
    \partial_t
    \big( 
        e^{\frac{\mathcal{R}}{2}t}
        (\partial_t u )_\eta
    \big)
    &
    +
    \mathbb{J}
    e^{\frac{\mathcal{R}}{2}t}
    \eta'(t) (1+|D_x|)
    (\partial_t  u)_{\eta}
    +
    e^{\frac{\mathcal{R}}{2}t}
    (\partial_t  u)_{\eta} 
    +
    e^{\frac{\mathcal{R}}{2}t}
     (u\partial_xu)_{\eta}
     +\\
     &+
    e^{\frac{\mathcal{R}}{2}t}
    (v\partial_yu)_{\eta}+
    e^{\frac{\mathcal{R}}{2}t}
    \partial_xp_{\eta}-\pa_y^2\big(
    e^{\frac{\mathcal{R}}{2}t}
    u_{\eta}
    \big)
    =
    \frac{\mathbb J \R }{2}
    e^{\frac{\mathcal{R}}{2}t}
    (\partial_t u)_{\eta}
    +
    e^{\frac{\mathcal{R}}{2}t}
    F_{\eta},
\end{aligned}
\end{equation}
where the forcing term $F_\eta$ in \eqref{eq:MHD-Prandtl-without-e} is generated by applying the Fourier multiplier $ e^{(\tau_0-\eta(t) )(1+|D_x|)}$ to the right-hand side of \eqref{eq:u-sec:final-estimates}:
\begin{equation}\label{Feta}
    F_{\eta} = 
        \mathbb H^2
        (
        b_1 
        b_2 
        v
        )_\eta 
        - 
        \mathbb H^2
        (
        u 
        b_2^2
        )_\eta 
        +
        \mathbb H^2
        \big(
        b_2
        \big( 
            \smallint_0^y
            \partial_t b_1 
        \big)
         \big)_\eta.
\end{equation}
We can further derive an equivalent form of this equation, by developing the time derivative $e^{\mathcal R t/2} (\partial_t u )_\eta$ in the third term of \eqref{eq:ueta}  by means of $e^{\mathcal Rt/2} (\partial_t (u _\eta)+\eta'(t)(1+|D_x|)u_\eta)$. Thus equation \eqref{eq:ueta} can also be recasted as
\begin{equation}\label{eq:ueta2}
\begin{aligned}
    \partial_t
    \big( 
    e^{\frac{\R }{2}t}
    \big(
    \mathbb{J}
    (\partial_t u )_\eta + u_\eta
    \big)
    \big)
    &+
    e^{\frac{\R }{2}t}
    \eta'(t) 
    (1+|D_x|)
    \big(
    \mathbb J
        (\partial_t  u)_{\eta}
        + u_\eta
    \big)
    +
    e^{\frac{\R }{2}t}
    (u\partial_xu)_{\eta}
    +\\
    &+
    e^{\frac{\R }{2}t}
    (v\partial_yu)_{\eta}+
    e^{\frac{\R }{2}t}
    \partial_xp_{\eta}-\pa_y^2\big(
    e^{\frac{\R }{2}t}
    u_{\eta}
    \big)
    =
    \frac{\R }{2}
    e^{\frac{\R }{2}t}
    \big(
        \mathbb J
        (\partial_t u)_{\eta}
        +
        u_\eta
    \big)
    +
    e^{\frac{\R }{2}t}
    F_{\eta}.
\end{aligned}
\end{equation}
Next, we take the $H^{s, 0}$-inner product between \eqref{eq:ueta} and $ e^{\R t/2}  \mathbb J( \partial_t u)_\eta$ and adding the result with the $H^{s, 0}$-inner product between \eqref{eq:ueta2} and $e^{\R t/2}  (\mathbb J( \partial_t u)_\eta+u_\eta)$, we gather that
\begin{equation}\label{energy-id-lemma-ub}
\begin{aligned}
    \frac{d}{dt}
    \Big[
        e^{\R t}
        \Big(
        \frac{\mathbb J^2}{2}
        \| 
        (\partial_t u)_{\eta}   
        \|_{H^{s,0}}^2 
        +
        \frac{1}{2}
        \|  
        \mathbb J (\partial_t u)_{\eta}   
        \!+\!
        u_\eta
        \|_{H^{s,0}}^2 
        +
        \mathbb J
        \| 
         \partial_y u_{\eta}   
        \|_{H^{s,0}}^2 
    \Big)
    \Big] 
    \!+\!
    e^{\R t}
    \Big(
    \|  \pa_y   u_{\eta}      \|_{H^{s,0}}^2
    \!+\!
    \mathbb J
    \| 
        (\partial_t u)_{\eta}   
    \|_{H^{s,0}}^2
    \Big)+\\
    +
    \eta'(t)
    e^{\R t}
    \Big\{
        \mathbb J^2
        \| 
            (\partial_t u)_{\eta}       
        \|_{H^{s+\frac12,0}}^2
        + 
        \| 
            \mathbb J (\partial_t u)_{\eta} + u_\eta       
        \|_{H^{s+\frac12,0}}^2 
        +
    2\mathbb J
     \| \partial_y u_{\eta} \|_{H^{s+\frac12, 0}}^2
    \Big\}
    =
    \R
    e^{\R t}
    \Big[ 
    \frac{\mathbb J^2}{2}
    \| (\partial_t u)_{\eta} \|_{H^{s, 0}}^2
    +\\+
    \frac{1}{2}
    \|   \mathbb J(\partial_t u)_{\eta} +  u_{\eta} \|_{H^{s, 0}}^2
    + 
    \mathbb J
    \| \partial_y u_{\eta} \|_{H^{s, 0}}^2
    \Big]
    +
    e^{\R t}
    \Big\{
    -\psca{ (u\pa_xu)_\eta,
    2 \mathbb J(\partial_t u)_\eta \!+\! u_\eta}_{H^{s,0}}
    +\\-
      \langle 
     ( 
        v
        \pa_yu
     )_\eta, 
   2\mathbb J(\partial_t u)_\eta \!+\! u_\eta\rangle_{H^{s,0}}
    \!+\!
     \psca{F_{\eta},
    2 \mathbb J(\partial_t u)_\eta \!+\! u_\eta}_{H^{s,0}}
     \!-\!
    \psca{\pa_x p_{\eta},
    2\mathbb J (\partial_t u)_\eta \!+\!u_\eta}_{H^{s,0}}\!
    \Big\}.
\end{aligned}
\end{equation}
We begin with by observing that the term in the square brackets of the right-hand side in \eqref{energy-id-lemma-ub} can be absorbed by the dissipation $e^{\R t}\big(\|  \pa_y   u_{\eta}      \|_{H^{s,0}}^2 \!+\! \mathbb J  \|  (\partial_t u)_{\eta}    \|_{H^{s,0}}^2\big)$ of the left-hand side. Indeed, since the value of $\mathcal R $ is smaller than $\min \{1/4, 1/(4\mathbb J)\}$, we have that
\begin{align*}
    \R
    e^{\R t}
    \Big[ 
    \frac{\mathbb J^2}{2}
    \| (\partial_t u)_{\eta} \|_{H^{s, 0}}^2
    &
    \!+\!
    \frac{1}{2}
    \|   \mathbb J(\partial_t u)_{\eta}\! + \! u_{\eta} \|_{H^{s, 0}}^2
    \!+\! 
    \mathbb J
     \| \partial_y u_{\eta} \|_{H^{s, 0}}^2
    \Big]
    \!\!\leq\!
    \R
    e^{\R t}
    \Big[ 
    \frac{3\mathbb J^2}{2}
    \| (\partial_t u)_{\eta} \|_{H^{s, 0}}^2
    \!+\!
    \|  u_{\eta} \|_{H^{s, 0}}^2
    \! \!+\! 
    \mathbb J
     \| \partial_y u_{\eta} \|_{H^{s, 0}}^2
    \Big]\\
    &\leq
    \frac{3\mathbb J}{8}
    e^{\R t}
    \| (\partial_t u)_{\eta} \|_{H^{s, 0}}^2
    +
    \frac{1}{4}
    e^{\R t}
    \|  u_{\eta} \|_{H^{s, 0}}^2
    + 
    \frac{1}{4}
    e^{\R t}
    \| \partial_y u_{\eta} \|_{H^{s, 0}}^2.
\end{align*}
To absorb this last term, we shall then invoke the Poincar\'e inequality in $y\in (0,1)$: $\| u_{\eta} \|_{H^{s, 0}} \leq \| \partial_y u_{\eta} \|_{H^{s, 0}} $ (here the homogeneous boundary condition $u = 0$ in $y = 0$ comes into play). Thus 
\begin{equation*}
    \begin{aligned}
    \R
    e^{\R t}
    \Big[ 
    \frac{\mathbb J^2}{2}
    \| (\partial_t u)_{\eta} \|_{H^{s, 0}}^2
    \!+\!
    \frac{1}{2}
    \|   \mathbb J(\partial_t u)_{\eta}\! + \! u_{\eta} \|_{H^{s, 0}}^2
    \!+\! 
    \mathbb J
     \| \partial_y u_{\eta} \|_{H^{s, 0}}^2
    \Big]
    \leq 
    e^{\R t}
    \Big(
    \frac{\mathbb J}{2}
    \| 
        (\partial_t u)_{\eta}   
    \|_{H^{s,0}}^2
    \!+\!
    \frac 12 
    \|  \pa_y   u_{\eta}      \|_{H^{s,0}}^2
    \Big),
    \end{aligned}
\end{equation*}
which corresponds to half dissipation on the left-hand side of \eqref{energy-id-lemma-ub}.

\noindent 
We can summarise what obtained with the following estimate:
\begin{equation}\label{energy-id-lemma-u}
\begin{aligned}
    \frac{d}{dt}
    \Big[
        e^{\R t}
        \Big(
        \frac{\mathbb J^2}{2}
        \| 
        (\partial_t u)_{\eta}   
        \|_{H^{s,0}}^2 
        \!+\!
        \frac{1}{2}
        \|  
        \mathbb J (\partial_t u)_{\eta}   
        \!+\!
        u_\eta
        \|_{H^{s,0}}^2 
        +
        \mathbb J
        \| 
         \partial_y u_{\eta}   
        \|_{H^{s,0}}^2 
    \Big)
    \Big] 
    \!\!+\!
    \frac{1}{2}
    e^{\R t}
    \Big(
    \|  \pa_y   u_{\eta}      \|_{H^{s,0}}^2
    \!+\!
    \mathbb J\| 
        (\partial_t u)_{\eta}   
    \|_{H^{s,0}}^2
    \Big)+\\
    +
    \eta'(t)
   e^{\R t}
     \Big\{
        \mathbb J^2
        \| 
            (\partial_t u)_{\eta}       
        \|_{H^{s+\frac12,0}}^2
        + 
        \| 
            \mathbb J (\partial_t u)_{\eta} + u_\eta       
        \|_{H^{s+\frac12,0}}^2 
        +
     2\mathbb J\| \partial_y u_{\eta} \|_{H^{s+\frac12, 0}}^2
    \Big\}
    \leq \\
    \leq 
    e^{\R t}
    \Big\{
    -\psca{ (u\pa_xu)_\eta,
    2 \mathbb J(\partial_t u)_\eta \!+\! u_\eta}_{H^{s,0}}
    +
      \langle 
    (
        v
        \pa_yu
    )_\eta, 
   2\mathbb J(\partial_t u)_\eta \!+\! u_\eta\rangle_{H^{s,0}}
    \!+\\
    +
     \psca{F_{\eta},
    2 (\partial_t u)_\eta \!+\! u_\eta}_{H^{s,0}}
     \!-\!
    \psca{\pa_x p_{\eta},
    2\mathbb J (\partial_t u)_\eta \!+\!u_\eta}_{H^{s,0}}\!
    \Big\}.
\end{aligned}
\end{equation}
The left-hand side of \eqref{energy-id-lemma-u} already provides information on several norms of the solution. However these norms are of Sobolev regularities lower than $s+1/2$, while our final estimate \eqref{ineq-proposition-Es} in Proposition \ref{prop:the-overall-final-estimate-in-eta} incorporates higher regularities, such as $H^{s+3/2, 0}$ in $\mathcal{D}_{s+3/2}$. We shall therefore perform a further development of the above inequality. To this end, we first isolate
$\frac{\eta'(t)}{2}e^{\R t}(\mathbb J^2\big\|  (\partial_t u )_\eta  \big\|_{H^{s+1/2, 0}}^2+\big\| \mathbb J (\partial_t u )_\eta+ u_\eta \big\|_{H^{s+1/2, 0}}^2)$ 
at the second line of 
\eqref{energy-id-lemma-u} and, recalling the formula  
$(\partial_t u )_\eta = \partial_t (u_\eta) + \eta'(t) (1+|D_x|) u_\eta$, we remark that
\begin{align*}
    &
    \frac{\eta'(t)}{2}
    e^{\R t}
    \Big\{
    \mathbb J^2
    \big\| 
        (\partial_t u )_\eta  
    \big\|_{H^{s+\frac{1}{2}, 0}}^2 
    +
    \| 
        \mathbb J (\partial_t u)_{\eta} + u_{\eta}       
    \|_{H^{s+\frac12,0}}^2\Big\} = \\
    &=
   \frac{\eta'(t)}{2}
   e^{\R t}
    \Big\{
    \mathbb J^2
    \big\| 
        \partial_t (u_{\eta} ) + \eta'(t) (1+|D_x|)u_{\eta}   
    \big\|_{H^{s+\frac{1}{2}, 0}}^2
     +
    \| 
       \mathbb J \partial_t (u_{\eta}) 
        +
        \mathbb J
        \eta'(t)
        (1+|D_x|)u_{\eta}
        + u_{\eta}       
    \|_{H^{s+\frac12,0}}^2\Big\}
    \\
    &=
   \frac{\eta'(t)}{2}
   e^{\R t}
    \bigg\{
    \mathbb J^2
    \big\| 
        \partial_t (u_{\eta})
    \big\|_{H^{s+\frac{1}{2}, 0}}^2 
    +
    \mathbb J^2
    \eta'(t)^2
    \big\| 
        u_{\eta}
    \big\|_{H^{s+\frac{3}{2}, 0}}^2   
    + 
    \mathbb J^2
    2\eta'(t) 
    \langle 
        \partial_t (u_{\eta}), 
        u_{\eta} 
    \rangle_{H^{s+1, 0}}
    +
    \mathbb J^2
    \big\| 
        \partial_t (u_{\eta})
    \big\|_{H^{s+\frac{1}{2}, 0}}^2
    +
    \\
    &+
    \mathbb J^2
    \eta'(t)^2
    \big\| 
        u_{\eta}
    \big\|_{H^{s+\frac{3}{2}, 0}}^2
    +
    \| u_{\eta}\|_{H^{s+\frac{1}{2}, 0}}^2
    +
    \mathbb J^2
    2\eta'(t) 
    \langle 
        \partial_t (u_{\eta}), 
         u_{\eta} 
    \rangle_{H^{s+1, 0}}
    +
    \mathbb J
    2
    \langle 
        \partial_t (u_{\eta}), 
        u_{\eta} 
    \rangle_{H^{s+\frac{1}{2}, 0}}
    \!+\!
    \mathbb J
    2\eta'(t)
    \|u_{\eta} \|_{H^{s+1, 0}}^2
    \bigg\}.
\end{align*}
In this last identity, we have used the relation $\| (1+|D_x|)u_{\eta} \|_{H^{s+\frac{1}{2}, 0}} = \| u_{\eta} \|_{H^{s+\frac{3}{2}, 0}}$, as well as the inner products  $\langle \partial_t (u_{\eta}), (1\!+\!|D_x|)u_{\eta} \rangle_{H^{s+\frac 12, 0}} = \langle \partial_t (u_{\eta}), u_{\eta} \rangle_{H^{s+1, 0}}$ and
$\langle (1 \!+\!|D_x|)u_{\eta}, \,u_{\eta} \rangle_{H^{s+\frac 12, 0}} =  
\|u_{\eta} \|_{H^{s+1, 0}}^2$. Thus, the isolated term satisfies
\begin{align*}
   \frac{\eta'(t)}{2}
   e^{\R t}
    &\Big\{
    \mathbb J^2
    \big\| 
        (\partial_t u )_\eta  
    \big\|_{H^{s+\frac{1}{2}, 0}}^2 
    +
    \| 
        \mathbb J (\partial_t u)_{\eta} + u_{\eta}       
    \|_{H^{s+\frac12,0}}^2\Big\} = 
    \mathbb J^2
    e^{\R t}
    \eta'(t)^2
    \frac{d}{dt}
     \big\| 
         u_{\eta}
    \big\|_{H^{s+1, 0}}^2
    \!+\!
    \frac{\mathbb J}{2}
    e^{\R t}
    \eta'(t)
    \frac{d}{dt}
     \big\| 
         u_{\eta}
    \big\|_{H^{s+\frac{1}{2}, 0}}^2+
    \\
    &+ 
    e^{\R t}
    \bigg\{
    \mathbb J^2
    \eta'(t)
    \big\| 
        \partial_t (u_{\eta})
    \big\|_{H^{s+\frac{1}{2}, 0}}^2 
    +
    \mathbb J^2
    \eta'(t)^3
    \big\| 
         u_{\eta} 
    \big\|_{H^{s+\frac{3}{2}, 0}}^2
    +
    \frac{\eta'(t)}{2}
    \big\| 
         u_{\eta} 
    \big\|_{H^{s+\frac{1}{2}, 0}}^2
    +
    \mathbb J
    \eta'(t)^2
    \big\| 
         u_{\eta} 
    \big\|_{H^{s+1, 0}}^2
    \bigg\}.
\end{align*}
and by bringing the time derivative in front of $\mathbb J e^{\R t}\eta'(t)^2$ and $\mathbb J\eta'(t)/2$ also
\begin{equation}\label{Hs+3inb1}
\begin{aligned}
    \frac{\eta'(t)}{2}
    e^{\R t}
    \!
    &\Big\{
    \mathbb J^2
    \big\| 
        (\partial_t u )_\eta  
    \big\|_{H^{s+\frac{1}{2}, 0}}^2 
    \!\!+\!
    \| 
        \mathbb J (\partial_t u)_{\eta} \!+\! u_{\eta}       
    \|_{H^{s+\frac12,0}}^2\Big\}
    \!=\!\!
    \frac{d}{dt}
    \Big[
    \mathbb J^2
   \eta'(t)^2
    e^{\R t}
     \big\| 
         u_{\eta}
    \big\|_{H^{s+1, 0}}^2
    \!\!+\!
    \frac{\mathbb J}{2}
    \eta'(t)
    e^{\R t}
    \big\| 
        u_{\eta}
    \big\|_{H^{s+\frac{1}{2}, 0}}^2
    \Big]\\ 
    &
    -
    2
    \mathbb J^2
    e^{\R t}
    \eta'(t)\eta''(t)
     \big\| 
         u_{\eta}
    \big\|_{H^{s+1, 0}}^2
    -
    \frac{\mathbb J}{2}
    \eta''(t)
    e^{\R t}
    \big\| 
         u_{\eta}
    \big\|_{H^{s+\frac{1}{2}, 0}}^2
    +
     e^{\R t}
    \bigg\{
    \mathbb J^2
    \eta'(t)
    \big\| 
        \partial_t (u_{\eta})
    \big\|_{H^{s+\frac{1}{2}, 0}}^2 
    +\\ 
    &
    \hspace{1cm}+
    \mathbb J^2
    \eta'(t)^3
    \big\| 
         u_{\eta} 
    \big\|_{H^{s+\frac{3}{2}, 0}}^2
    +
    \Big(
        \frac{1}{2}
        -
        \frac{\mathbb J \R}{2}
    \Big)
    \eta'(t)
    \big\| 
         u_{\eta} 
    \big\|_{H^{s+\frac{1}{2}, 0}}^2
    +
    \big(
        \mathbb J
        -
        \R 
        \mathbb J^2
    \big)
    \eta'(t)^2
    \big\| 
         u_{\eta} 
    \big\|_{H^{s+1, 0}}^2
    \bigg\}.
\end{aligned}
\end{equation}
The Sobolev norm of $H^{s+3/2, 0}$ has now appeared (first term of the third line). Furthermore, recalling that $\mathcal{R}\leq (1 / 4)\min\{1, \mathbb J^{-1}\} $ both $(1/2-\mathbb J \R/2 ) \geq 3/8$ and 
$ \mathbb J -\R \mathbb J^2\geq 3\mathbb J/4$ are positive. Needless to say, this high regularity comes with a price, namely the appearance of certain terms which depend on the second time derivative $\eta''(t)$ (and are also in our main estimate \eqref{ineq-proposition-Es}). By coupling \eqref{energy-id-lemma-u} together with \eqref{Hs+3inb1}, we eventually gather the estimate \vspace{-0.2cm}
\begin{equation}\label{energy-id-lemma-uc}
\begin{aligned}
    \frac{d}{dt}
    \Big[
         e^{\R t}
        \Big(
        \underbrace
        {
         \frac{\mathbb J^2}{2}
        \| 
        (\partial_t u)_{\eta}   
        \|_{H^{s,0}}^2 
        +
         \frac{1}{2}
        \|  
        \mathbb J
        (\partial_t u)_{\eta}   
        \!+\!
        u_\eta
        \|_{H^{s,0}}^2 
        +
        \mathbb J
        \| 
         \partial_y u_{\eta}   
        \|_{H^{s,0}}^2 \!
        }_{\text{$u$-terms in }\mathcal{E}_{s}(t)}
        +
        \eta'(t)
        \hspace{-0.2cm}
        \underbrace{
        \frac{\mathbb J}{2}
        \| 
             u_{\eta}
        \|_{H^{s+\frac{1}{2}, 0}}^2
        }_{\text{$u$-term in }\mathcal{E}_{s+\frac{1}{2}}(t)}
        \hspace{-0.3cm}
        +
        \eta'(t)^2
        \hspace{-0.1cm}
        \underbrace{
        \mathbb J^2
        \| 
             u_{\eta}
        \|_{H^{s+1, 0}}^2
         }_{\text{$u$-term in }\mathcal{E}_{s+1}(t)}
    \Big)
    \Big]+\\ 
   +
    \eta'(t)
     e^{\R t}
     \bigg\{ 
     \underbrace{
        \frac{\mathbb J^2}{2}
        \| 
            (\partial_t u)_{\eta}       
        \|_{H^{s+\frac12,0}}^2
        \!\!+\!\! 
        \frac{1}{2}
        \| 
           \mathbb J (\partial_t u)_{\eta} \!+\! u_\eta       
        \|_{H^{s+\frac12,0}}^2 
        \!\!+\!\!
        2\mathbb J
     \| \partial_y u_{\eta}     \|_{H^{s+\frac12, 0}}^2
     \!\!+\!\!
     \mathbb J^2
     \| \partial_t (u_{\eta})   \|_{H^{s+\frac12,0}}^2 
     \!\!+\!\! 
     \frac{3}{8}
     \| u_\eta \|_{H^{s+\frac12,0}}^2
     }_{\text{$u$-terms in }\mathcal{D}_{s+\frac{1}{2}}(t)}
    \bigg\}
    +
    \\
    +
     \mathbb J
     \eta'(t)^2
     e^{\R t}
     \underbrace{
     \frac{3}{4}
     \| u_\eta \|_{H^{s+1,0}}^2
     }_{\text{$u$-term in }\mathcal{D}_{s+1}(t)}
     \hspace{-0.2cm}
     +
    \mathbb J^2
    \eta'(t)^3
    e^{\R t}
    \hspace{-0.2cm}
    \underbrace{
    \| u_\eta \|_{H^{s+\frac{3}{2},0}}^2
    }_{\text{$u$-terms in }\mathcal{D}_{s+\frac{3}{2}}(t)}
    \hspace{-0.4cm}
      -
    2
    \mathbb J^2
    e^{\R t}
    \eta'(t)\eta''(t)
    \hspace{-0.2cm}
     \underbrace{
     \big\| 
         u_{\eta}
    \big\|_{H^{s+1, 0}}^2
    }_{ \text{$u$-term in }\mathcal{E}_{s+1}(t)}
    \hspace{-0.3cm}
    -
    \mathbb J
    e^{\R t}
    \eta''(t)
    \hspace{-0.2cm}
    \underbrace{
    \frac{1}{2}
     \big\| 
         u_{\eta}
    \big\|_{H^{s+\frac{1}{2}, 0}}^2
    }_{ \text{$u$-term in }\mathcal{E}_{s+\frac{1}{2}}(t)}
    +\\
    +
    \frac{ e^{\R t}}{2}
    \underbrace{
    \Big(
    \|  \pa_y   u_{\eta}      \|_{H^{s,0}}^2
    \!+\!
    \mathbb J
    \| 
        (\partial_t u)_{\eta}   
    \|_{H^{s,0}}^2
    \Big)
    }_{ \text{$u$-term in }\mathcal{D}_{s}(t)}
    \leq  
    e^{\R t}
    \Big\{
    -\psca{ (u\pa_xu)_\eta,
    2 \mathbb J(\partial_t u)_\eta \!+\! u_\eta}_{H^{s,0}}
    +\\
    -
    \langle 
    ( 
        v
        \pa_yu
    )_\eta, 
   2\mathbb J(\partial_t u)_\eta \!+\! u_\eta\rangle_{H^{s,0}}
    \!+\!
     \psca{F_{\eta},
    2\mathbb J (\partial_t u)_\eta \!+\! u_\eta}_{H^{s,0}}
     \!-\!
    \psca{\pa_x p_{\eta},
    2\mathbb J (\partial_t u)_\eta \!+\!u_\eta}_{H^{s,0}}\!
    \Big\}.
\end{aligned}
\end{equation}
We next proceed to estimate each term on the right-hand side of \eqref{energy-id-lemma-uc}. For each estimated term, we will determine a suitable lower bound of the constant $D_s\geq 1$ in the main inequality \eqref{ineq-proposition-Es} of Proposition \ref{prop:the-overall-final-estimate-in-eta}. 
This lower bound will increase at any step. The last term will therefore provide the exact form of $D_s$. 

\noindent
Throughout the forthcoming analysis, we will repeatedly use the following estimates, which recast the $H^s$-Sobolev norm of $u_\eta$ and $b_{1,\eta}$ in terms of $\mathcal{E}_s$.
\begin{equation}\label{Hs-bounded-by-Es}
\begin{aligned}
    \| u_\eta \|_{H^{s,0}} 
    &\leq 
    \|  u_\eta \!+\! \mathbb J (\partial_t u)_\eta \|_{H^{s,0}}
    \!+\!
    \mathbb J\| (\partial_t u)_\eta \|_{H^{s,0}}
    \\
    &\leq 
    2
    \Big(
        \frac{\mathbb J^2}{2}
        \| (\partial_t u)_\eta \|_{H^{s,0}}^2
        \!+\!
        \frac{1}{2}
        \|\mathbb J (\partial_t u)_\eta 
        \!+\! u_\eta \|_{H^{s,0}}^2
    \Big)^\frac{1}{2}
    \!\!
    \leq 
    2
    \sqrt{\mathcal{E}_s},
    \\
    \| b_{1,\eta} \|_{H^{s,0}} 
    &\leq 
    \Big\| 
        b_{1,\eta} +
        \frac{\kappa}{{\rm Pr}_m}
        (\partial_t b_1)_\eta 
    \Big\|_{H^{s,0}}
    \!+\!
    \frac{\kappa}{{\rm Pr}_m}
    \| (\partial_t b_1)_\eta \|_{H^{s,0}}
    \\
    &\leq 
    2
    \Big(
        \frac{\kappa^2}{2{\rm Pr}_m^2}
        \| (\partial_t b_1)_\eta \|_{H^{s,0}}^2
        \!+\!
        \frac{1}{2}
        \Big\| 
        \frac{\kappa}{{\rm Pr}_m}
        (\partial_t b_1)_\eta + b_{1,\eta} 
    \Big\|_{H^{s,0}}^2
    \Big)^\frac{1}{2}\!\!
    \leq 
    2
    \sqrt{\mathcal{E}_s}.
\end{aligned}
\end{equation}
Similarly, we can connect the $H^{s+1/2}$-norms of $u_\eta$ and $b_{1,\eta}$ in terms of $\mathcal{D}_{s+\frac{1}{2}}$:
\begin{equation}\label{Hs+1/2-bounded-by-Ds}
\begin{alignedat}{8}
    \| u_\eta \|_{H^{s+\frac{1}{2},0}} 
    \leq 
    2\sqrt{\mathcal{D}_{s+\frac{1}{2}}},
    \qquad
    \qquad
    \| b_{1,\eta} \|_{H^{s+\frac{1}{2},0}} 
    \leq 
    2\sqrt{\mathcal{D}_{s+\frac{1}{2}}} .
\end{alignedat}
\end{equation}

\noindent
The first term on the right-hand side of \eqref{energy-id-lemma-uc}, that we deal with, is the convection
\begin{equation}\label{secu:first-estimate}
\begin{aligned}
    &\big|
       e^{\R t}
       \psca{ (u\pa_xu)_\eta,2\mathbb J(\partial_t u)_\eta + u_{\eta}}_{H^{s,0}}
    \big|
    \leq
    e^{\R t}
    \| (u\pa_xu)_\eta   \|_{H^{s-\frac{1}{2},0}}
    \Big(
    \mathbb J
     \|
     (\partial_t u)_\eta         \|_{H^{s+\frac{1}{2},0}}
     \!\!\!+\!
    \| \mathbb J(\partial_t u)_\eta \!+\! u_{\eta}        \|_{H^{s+\frac{1}{2},0}}
    \Big)\\
    &\leq
    e^{\R t}
    \| (u\pa_xu)_\eta   \|_{H^{s-\frac{1}{2},0}}
    2
    \Big(
    \frac{\mathbb J^2}{2}
     \|
     (\partial_t u)_\eta         \|_{H^{s+\frac{1}{2},0}}^2
     \!+\!
     \frac{1}{2}
    \| \mathbb J (\partial_t u)_\eta \!+\! u_{\eta}        \|_{H^{s+\frac{1}{2},0}}^2
    \Big)^\frac{1}{2}
    \!\leq
    2
    e^{\R t}
    \| (u\pa_xu)_\eta   \|_{H^{s-\frac{1}{2},0}}
    \sqrt{\mathcal{D}_{s+\frac{1}{2}}(t)}.
\end{aligned}
\end{equation}
In order to cope with $(u\pa_xu)_\eta $ in $H^{s-1/2,0}$, we shall transfer the $\eta$-transformation (i.e.~the Lagrangian multiplier $e^{(\tau_0-\eta(t))(1+|D_x|)}$) to each component $u$ and $\partial_x u$. Of course, $(u\pa_xu)_\eta\neq u_\eta \pa_xu_\eta$ in general. However, we are here controlling a Sobolev norm and not the functions themselves, pointwise. The following product law therefore allows us to transfer the mentioned Lagrangian multiplier in terms of pure Sobolev estimates:
\begin{lemma}\label{lemma:technical-lemma-product-law-Hs0}
    Let  $f,\,g:\mathbb R\times (0,1)\to \mathbb R$ be two functions such that $f_\eta$, $g_\eta$ and $\partial_y f_\eta$ belong to $H^{\sigma_1,0}(\mathbb R\times (0,1))$ with $\sigma_1>1/2$ (and thus also to $H^{\sigma_2,0}(\mathbb R\times (0,1))$, for any $\sigma_2\leq \sigma_1$). 
    Furthermore, assume that $f\equiv 0$ in $y = 0$ in the sense of trace. 
    Then
    \begin{equation*}
    \begin{aligned}
        \| (f g)_\eta \|_{H^{\sigma_1, 0}} 
        \leq
            \frac{2^{\sigma_1-\frac{1}{2}}}{\sqrt{\sigma_2 - \frac{1}{2}}}
            \bigg(
            \| \partial_y f_\eta \|_{H^{\sigma_1, 0}}
            \|  g_\eta           \|_{H^{\sigma_2, 0}}
            +
            \| \partial_y f_\eta \|_{H^{\sigma_2, 0}}
            \| g_\eta           \|_{H^{\sigma_1, 0}}
            \bigg)
    \end{aligned}
    \end{equation*}
     for any regularities $\sigma_2 \in (1/2, \sigma_1]$.
\end{lemma}
\noindent 
We postpone the technical proof of this lemma \ref{lemma:technical-lemma-product-law-Hs0} to the appendix  (cf.~Lemma \ref{appx:lemma-product-eta}). Addressing our original estimate \eqref{secu:first-estimate}, we are in the position to apply Lemma \ref{lemma:technical-lemma-product-law-Hs0} with the regularities $\sigma_1=\sigma_2 = s-1/2>1/2$ and the functions  $f(\cdot) = u(t,\cdot),\,g(\cdot) = \partial_x u (t,\cdot)$. We deduce therefore that
\begin{align*}
    \| (u\pa_xu)_\eta  \|_{H^{s-\frac{1}{2},0}}
    &\leq  
    \frac{2^{s-1}}{\sqrt{s - 1}}
    2\| \partial_y u_\eta \|_{H^{s-\frac{1}{2},0}}
    \| \partial_x u_\eta \|_{H^{s-\frac{1}{2},0}}
    \leq 
    \frac{2^{s}}{\sqrt{s - 1}}
    \| \partial_y u_\eta \|_{H^{s,0}}
    \| u_\eta \|_{H^{s+\frac{1}{2},0}}
    \\
    &\leq 
    \frac{2^{s+1}}{\sqrt{s - 1}}
    \| \partial_y u_\eta \|_{H^{s,0}}
    \sqrt{\mathcal{D}_{s+\frac{1}{2}}(t)},
\end{align*}
where in the last inequality we have indeed applied \eqref{Hs+1/2-bounded-by-Ds}.
Plugging this inequality to the original estimate \eqref{secu:first-estimate}, we eventually gather that the convective term is bounded by
\begin{equation*}
    \begin{aligned}
        \big|
            e^{\R t}
            \psca{ (u\pa_xu)_\eta,u_{\eta}}_{H^{s,0}}
        \big|
        \leq
        \frac{2^{s+2}}{\sqrt{s - 1}}
        e^{\R t}
        \frac{1}{\sqrt{\mathbb J}}
        \sqrt{\mathcal{E}_{s}(t)}
       \mathcal{D}_{s+\frac{1}{2}}(t).
    \end{aligned}
\end{equation*}
We shall now remark that the above right-hand side is indeed in the first integrand of the third line of our main inequality \eqref{ineq-proposition-Es} of Proposition \ref{prop:the-overall-final-estimate-in-eta}. 
A necessary condition for the validity of this Proposition is that the related constant $D_s$ must satisfy $D_s\geq 2^{s+2}/\sqrt{s-1}$.

\noindent
Next, we deal with the second term on the right-hand side of \eqref{energy-id-lemma-u}, more precisely
\begin{align*}
      e^{\R t}
    \big|
    \big\langle 
    ( 
        v
        \pa_yu
    )_\eta,\
    2\mathbb J(\partial_t u)_\eta + u_{\eta}
    \big\rangle_{H^{s,0}}
    \big|
    &\leq 
    e^{\R t}
    \big\|
        ( 
            v\pa_yu
        )_\eta
    \big\|_{H^{s-\frac{1}{2},0}}
    \Big(
        \mathbb J\|(\partial_t u)_\eta \|_{H^{s+\frac{1}{2},0}}
        +
        \| \mathbb J(\partial_t u)_\eta \!+\! u_{\eta} \|_{H^{s+\frac{1}{2},0}} 
    \Big)\\
    &\leq 
    e^{\R t}
    \big\|
        ( 
            v\pa_yu
        )_\eta
    \big\|_{H^{s-\frac{1}{2},0}}
    2\sqrt{D_{s+\frac{1}{2}}(t)}.
\end{align*}
We apply  Lemma \ref{lemma:technical-lemma-product-law-Hs0} once more with regularities $\sigma_1=\sigma_2 = s-1/2>1/2$, but with functions  $f =v$ and $g  = \partial_y u$. Hence
\begin{align*}
    e^{\R t}
    \big|
    \langle 
    & 
    (
         v\pa_yu
    )_\eta,\,
    2\mathbb J(\partial_t u)_\eta + u_{\eta}
    \rangle_{H^{s,0}}
    \big|
    \leq 
    \frac{2^{s-1}}{\sqrt{s - 1}}
    e^{\R t}
    2
    \| 
        \partial_y v_\eta 
    \|_{H^{s-\frac{1}{2},0}} 
    \|
        \partial_y u_\eta
    \|_{H^{s-\frac{1}{2},0}}
    2\sqrt{D_{s+\frac{1}{2}}(t)}
     \\
    &\leq 
    \frac{2^{s+1}}{\sqrt{s - 1}}
    e^{\R t}
    \big\| 
        \partial_x  u_\eta 
    \big\|_{H^{s-\frac{1}{2},0}} 
    \|
        \partial_y u_\eta
    \|_{H^{s-\frac{1}{2},0}}
   \sqrt{D_{s+\frac{1}{2}}(t)} \\
    &\leq 
    \frac{2^{s+1}}{\sqrt{s - 1}}
    e^{\R t}
    \|
        \partial_y u_\eta
    \|_{H^{s,0}}
    \| u_\eta \|_{H^{s+\frac{1}{2},0}}
    \sqrt{D_{s+\frac{1}{2}}(t)}
    \\
    &\leq 
    \frac{2^{s+1}}{\sqrt{s - 1}}
    e^{\R t}
    \|
        \partial_y u_\eta
    \|_{H^{s,0}}
    \mathcal{D}_{s+\frac{1}{2}}(t)
    \leq 
    \frac{2^{s+2}}{\sqrt{s - 1}}
    e^{\R t}
    \frac{1}{\sqrt{\mathbb J}}
    \sqrt{\mathcal{E}_{s}(t)}
    \mathcal{D}_{s+\frac{1}{2}}(t),
\end{align*}
which is in the first integrand of the third line of our energy inequality \eqref{ineq-proposition-Es}. Hence $D_s$ must satisfy at least $D_s\geq 2^{s+3}/\sqrt{s-1}$.

\noindent 
Next, we aim to estimate each component of the function $F_\eta$ in \eqref{energy-id-lemma-u} (see also \eqref{Feta}). We begin with
\begin{align*}
    e^{\R t}
    \big|
    \big\langle 
    (
       b_1 
       b_2
       v
    )_\eta 
    ,\,2\mathbb J(\partial_t u)_\eta + u_{\eta}
    \big\rangle_{H^{s,0}}
    \big|
    &\leq 
    e^{\R t}
    \big\| 
       (
        b_1 
        b_2 
        v
    )_\eta 
    \big\|_{H^{s-\frac{1}{2},0}}
    \Big(
    \|
       \mathbb J (\partial_t u)_\eta + u_{\eta}
    \|_{H^{s+\frac{1}{2},0}}
    +
    \mathbb J
    \| (\partial_t u)_\eta \|_{H^{s+\frac{1}{2},0}}
    \Big)\\
    &\leq 
    e^{\R t}
    \big\| 
       (
        b_1 
        b_2 
        v
    )_\eta 
    \big\|_{H^{s-\frac{1}{2},0}}
    2
    \sqrt{\mathcal{D}_{s+\frac{1}{2}}(t)}
\end{align*}    
Thanks to Lemma \ref{lemma:technical-lemma-product-law-Hs0}, with $\sigma_1 =s-1/2$, $\sigma_2 = s-1$, $f = v$ and $g = b_1b_2 $, we deduce that
\begin{align*}
     &
    e^{\R t}
    \big|
    \big\langle 
    (
       b_1 
       b_2
       v
    )_\eta 
    ,\,
    2\mathbb J(\partial_t u)_\eta + u_{\eta}
    \big\rangle_{H^{s,0}}
    \big|\\
    &\leq 
    e^{\R t}
    \frac{2^{s-1}}{\sqrt{s - 3/2}}
    \Big(
    \| 
        \partial_y v_\eta 
    \|_{H^{s-\frac{1}{2},0}}
    \|
        \big(
         b_{1} 
         b_2
         \big)_\eta
    \|_{H^{s-1,0}}
    +
    \| 
        \partial_y v_\eta 
    \|_{H^{s-1,0}}
    \|
        (
         b_{1} 
         b_2
         )_\eta
    \|_{H^{s-\frac{1}{2},0}}
    \Big)
    2
    \sqrt{\mathcal{D}_{s+\frac{1}{2}}(t)}\\
    &\leq 
    e^{\R t}
    \frac{2^{s}}{\sqrt{s - 3/2}}
    \Big(
    \| 
          u_\eta 
    \|_{H^{s+\frac{1}{2},0}}
    \|
        (
         b_{1} 
         b_2
        )_\eta
    \|_{H^{s-1,0}}
    +
    \| 
          u_\eta 
    \|_{H^{s,0}}
    \|
        (
         b_{1} 
         b_2
        )_\eta
    \|_{H^{s-\frac{1}{2},0}}
    \Big)
    \sqrt{\mathcal{D}_{s+\frac{1}{2}}(t)}
\end{align*}
Next, we apply Lemma \ref{lemma:technical-lemma-product-law-Hs0} twice, first with regularities $\sigma_1 = \sigma_2 = s-1$ and functions $f = b_2$, $g= b_1$ (to deal with the term $\| (b_1b_2)_\eta \|_{H^{s-1}}$), secondly with  regularities $\sigma_1 = \sigma_2 = s-1/2$ and same functions $f = b_2$, $g= b_1$ (to deal with the term $\| (b_1b_2)_\eta \|_{H^{s-1/2}}$).
By recalling the divergence-free condition $\partial_y b_{2,\eta} = -\partial_x b_{1,\eta}$, we gather
\begin{align*}
    &
     e^{\R t}
    \big|
    \big\langle 
    (
       b_1 
       b_2
       v
    )_\eta 
    ,\,
    2\mathbb J(\partial_t u)_\eta + u_{\eta}
    \big\rangle_{H^{s,0}}
    \big|
    \leq 
    e^{\R t}
    \frac{2^{s}}{\sqrt{s - 3/2}}
    \bigg\{
    \| u_\eta \|_{H^{s+\frac{1}{2},0}}
    \frac{2^{s-\frac{3}{2}}}{\sqrt{s-3/2}}
    2
    \| b_{1,\eta} \|_{H^{s-1,0}} 
    \|
        \partial_x 
          b_{1,\eta} 
    \|_{H^{s-1,0}} 
    +\\ 
    &\hspace{2cm}+
    \frac{2^{s-1}}{\sqrt{s-1}}
    2
    \| u_\eta   \|_{H^{s,0}}
    \| b_{1,\eta} \|_{H^{s-\frac{1}{2},0}} 
    \|
        \partial_x 
          b_{1,\eta} 
    \|_{H^{s-\frac{1}{2},0}} 
    \bigg\}
    \sqrt{\mathcal{D}_{s+\frac{1}{2}}(t)}
    \\
    &\leq 
    \frac{2^{2s}}{s - 3/2}
     e^{\R t}
    \bigg\{
    \| b_{1,\eta} \|_{H^{s,0}}^2
    \| u_\eta \|_{H^{s+\frac{1}{2},0}}
    +
   \| u_\eta   \|_{H^{s,0}}
    \| b_{1,\eta} \|_{H^{s,0}}
    \|
        b_{1,\eta} 
    \|_{H^{s+\frac{1}{2},0}} 
    \bigg\}
    \sqrt{\mathcal{D}_{s+\frac{1}{2}}(t)}.
\end{align*}
Thus, by recalling \eqref{Hs-bounded-by-Es} and \eqref{Hs+1/2-bounded-by-Ds}, we deduce that
\begin{align*} 
    & e^{\R t}
    \big|
    \big\langle 
    (
       b_1 
       b_2
       v
    )_\eta 
    ,\,
    2\mathbb J(\partial_t u)_\eta + u_{\eta}
    \big\rangle_{H^{s,0}}
    \big|
    \leq \\
    &\leq 
     \frac{2^{2s}}{s - 3/2}
     e^{\R t}
    \bigg\{
    2^2\mathcal{E}_s(t)
    2\sqrt{D_{s+\frac{1}{2}}(t)}
    +
    2^2
    \mathcal{E}_s(t)
    2
    \sqrt{D_{s+\frac{1}{2}}(t)}
    \bigg\}
    \sqrt{D_{s+\frac{1}{2}}(t)}
    \\
    &
    \leq 
    \frac{2^{2s+4}}{s - 3/2}
     e^{\R t}
    \mathcal{E}_{s}(t)
    \mathcal{D}_{s+\frac{1}{2}}(t),
\end{align*}    
which is one of integral in the third line of \eqref{lemma:technical-lemma-product-law-Hs0}. 
Hence $D_s$ must satisfy $D_s\geq 2^{s+3}/\sqrt{s-1} + 2^{2s+4}/(s - 3/2)$.

\noindent 
The remaining components of $F_\eta$ are dealt with an analogous procedure. Thanks to Lemma \ref{lemma:technical-lemma-product-law-Hs0}, we have indeed
\begin{align*}
      e^{\R t}
    \big|
    \big\langle 
    (
       b_2^2
       u
    )_\eta 
    ,\,2\mathbb J(\partial_t u)_\eta + u_{\eta}
    \big\rangle_{H^{s,0}}
    \big|
    &\leq 
     e^{\R t}
    \big\| 
      (
       b_2^2
       u
      )_\eta  
    \big\|_{H^{s-\frac{1}{2},0}}
    \big(
    \|
        \mathbb J(\partial_t u)_\eta + u_{\eta}
    \|_{H^{s+\frac{1}{2},0}}
    +
    \mathbb J\| (\partial_tu)_\eta \|_{H^{s+\frac{1}{2},0}}
    \big)\\
    &\leq 
     e^{\R t}
    \big\| 
       \big(
       b_2^2
       u
    \big)_\eta 
    \big\|_{H^{s-\frac{1}{2},0}}
    2
    \sqrt{\mathcal{D}_{s+\frac{1}{2}}(t)}.
\end{align*}    
We invoke once more Lemma \ref{lemma:technical-lemma-product-law-Hs0}, with $\sigma_1 =s-1/2$, $\sigma_2 = s-1$, $f = b_2$ and $g = u b_2 $. Thanks to the divergence-free condition $\partial_y b_{2,\eta} = -\partial_x b_{1,\eta}$, we deduce that
\begin{align*}
     &  e^{\R t}
    \big|
    \big\langle 
    (
       b_2^2
       u
    )_\eta 
    ,\,2\mathbb J(\partial_t u)_\eta + u_{\eta}
    \big\rangle_{H^{s,0}}
    \big|\\
    &\leq 
     e^{\R t}
    \frac{2^{s-1}}{\sqrt{s - 3/2}}
    \Big(
    \| 
        \partial_y b_{2, \eta}
    \|_{H^{s-\frac{1}{2},0}}
    \|
        \big(
         u 
         b_2
         \big)_\eta
    \|_{H^{s-1,0}}
    +
   \| 
        \partial_yb_{2, \eta}
   \|_{H^{s-1,0}}
    \|
        \big(
         u 
         b_2
         \big)_\eta
    \|_{H^{s-\frac{1}{2},0}}
    \Big)
    2
    \sqrt{\mathcal{D}_{s+\frac{1}{2}}(t)}\\
    &\leq 
    \frac{2^s}{\sqrt{s - 3/2}}
     e^{\R t}
    \Big(
    \| 
           b_{1,\eta} 
    \|_{H^{s+\frac{1}{2},0}}
    \|
        (
         u
         b_2
         )_\eta
    \|_{H^{s-1,0}}
    +
    \| 
          b_{1,\eta} 
    \|_{H^{s,0}}
    \|
        (
         u 
         b_2
         )_\eta
    \|_{H^{s-\frac{1}{2},0}}
    \Big)
    \sqrt{\mathcal{D}_{s+\frac{1}{2}}(t)}.
\end{align*}
We are now in the position to apply Lemma \ref{lemma:technical-lemma-product-law-Hs0} to cope with 
$ \|(u b_2)_\eta\|_{H^{s-1,0}}$ and $ \|(u b_2)_\eta\|_{H^{s-1/2,0}}$. We first consider regularities $\sigma_1 = \sigma_2 = s-1$ and functions $f = b_2$, $g= u$ and secondly regularities $\sigma_1 = \sigma_2 = s-1/2$, with functions $f = b_2$, $g= u$:
\begin{align*}
    & e^{\R t}
    \big|
    \big\langle 
    (
       b_2^2
       u
    )_\eta 
    ,\,2\mathbb J(\partial_t u)_\eta + u_{\eta}
    \big\rangle_{H^{s,0}}
    \big|
    \leq 
    \frac{2^s }{\sqrt{s - 3/2}}
     e^{\R t}
    \bigg\{
    \| b_{1,\eta} \|_{H^{s+\frac{1}{2},0}}
    \frac{2^{s-\frac{3}{2}}}{\sqrt{s-3/2}}
    2
    \|
        \partial_x 
          b_{1,\eta} 
    \|_{H^{s-1,0}} 
    \| u_\eta \|_{H^{s-1,0}} 
    +\\ 
    &\hspace{2cm}+
    \frac{2^{s-1}}{\sqrt{s-1}}
    2
    \|  b_{1,\eta}   \|_{H^{s,0}}
    \|
        \partial_x 
          b_{1,\eta} 
    \|_{H^{s-\frac{1}{2},0}} 
    \| u_\eta \|_{H^{s-\frac{1}{2},0}} 
    \bigg\}
    \sqrt{\mathcal{D}_{s+\frac{1}{2}}(t)}
    \\
    &\leq 
    \frac{2^{2s}}{s - 3/2}
     e^{\R t}
    \bigg\{
    \| b_{1,\eta} \|_{H^{s+\frac{1}{2},0}}
    \|  b_{1,\eta}   \|_{H^{s,0}}
    \| u_\eta \|_{H^{s,0}} 
    + 
    \| b_{1,\eta} \|_{H^{s,0}}
    \|  b_{1,\eta}   \|_{H^{s+\frac{1}{2},0}}
    \| u_\eta \|_{H^{s,0}}
    \bigg\}
    \sqrt{\mathcal{D}_{s+\frac{1}{2}}(t)}
    \\
    &\leq 
    \frac{2^{2s+1}}{s - 3/2}
     e^{\R t}
    \|  b_{1,\eta}   \|_{H^{s+\frac{1}{2},0}}
    \| b_{1,\eta} \|_{H^{s,0}}
    \| u_\eta \|_{H^{s,0}}
    \sqrt{\mathcal{D}_{s+\frac{1}{2}}(t)}.
\end{align*}
Thus, recalling \eqref{Hs-bounded-by-Es} and \eqref{Hs+1/2-bounded-by-Ds}, we obtain
\begin{align*}
     e^{\R t}
    \big|
    \big\langle 
    (
       b_2^2
       u
    )_\eta 
    ,\,2\mathbb J(\partial_t u)_\eta + u_{\eta}
    \big\rangle_{H^{s,0}}
    \big|
    &\leq 
    \frac{2^{2s+1}}{s - 3/2}
     e^{\R t}
    \bigg\{
    2
    \sqrt{\mathcal{D}_{s+\frac{1}{2}}(t)}
    2^2
    \mathcal{E}_{s}(t)
    \bigg\}
    \sqrt{\mathcal{D}_{s+\frac{1}{2}}(t)}
    = 
    \frac{2^{2s+4}}{s - 3/2}
     e^{\R t}
    \mathcal{E}_{s}(t)
    \mathcal{D}_{s+\frac{1}{2}}(t).
\end{align*}    
We need therefore to impose $D_s \geq 2^{s+3}/\sqrt{s-1} + 2^{2s+5}/(s - 3/2)$. 

\noindent 
The last term of $F_\eta$ is finally estimated as follows:
\begin{align*}
    \big|
     e^{\R t}
    \big\langle 
        \big(
        b_2
        \big( 
            \smallint_0^y
            \partial_t b_1 
        \big)
         \big)_\eta ,
        2\mathbb J(\partial_t u)_\eta + u_{\eta}
    \big\rangle_{H^{s,0}}
    \big|
    &\leq 
     e^{\R t}
    \big\|
        \big(
        b_2
        \big( 
            \smallint_0^y
            \partial_t b_1 
        \big)
         \big)_\eta 
    \big\|_{H^{s-\frac{1}{2},0}}
    \|
       2\mathbb J(\partial_t u)_\eta + u_{\eta}
    \|_{H^{s+\frac{1}{2},0}}\\
    &\leq 
    \frac{2^{s}}{\sqrt{s - 1}}
     e^{\R t}
    \| \partial_x b_{1,\eta} \|_{H^{s-\frac{1}{2}, 0}}
    \big\|    
    \smallint_0^y
        (\partial_t b_{1})_\eta
    \big\|_{H^{s-\frac{1}{2},0}}
    2\sqrt{\mathcal{D}_{s+\frac{1}{2}}(t)}\\
    &\leq 
    \frac{2^{s+1}}{\sqrt{s - 1}}
     e^{\R t}
    2\sqrt{\mathcal{D}_{s+\frac{1}{2}}(t)}
    \big\|    
    \smallint_0^y
        (\partial_t b_{1})_\eta
    \big\|_{H^{s-\frac{1}{2},0}}
    \sqrt{\mathcal{D}_{s+\frac{1}{2}}(t)}.
\end{align*}
Since $y\in (0,1)$, one has that $ \| \smallint_0^y (\partial_t b_{1})_\eta
\big\|_{H^{s-\frac{1}{2},0}}\leq \| (\partial_t b_{1})_\eta
\big\|_{H^{s-\frac{1}{2},0}}\leq 
 \| (\partial_t b_{1})_\eta
\big\|_{H^{s,0}}\leq (2/\sqrt{\mathbb J})\sqrt{\mathcal{E}_s(t)}$, hence
\begin{align*}
    \big|
     e^{\R t}\big\langle 
        \big(
        b_2
        \big( 
            \smallint_0^y
            \partial_t b_1 
        \big)
         \big)_\eta ,
        2\mathbb J(\partial_t u)_\eta + u_{\eta}
    \big\rangle_{H^{s,0}}
    \big|
    &\leq 
    \frac{2^{s+2}}{\sqrt{s - 1}}
     e^{\R t}
    \frac{1}{\sqrt{\mathbb J}}
    \sqrt{\mathcal{E}_{s}(t)}
    \mathcal{D}_{s+\frac{1}{2}}(t),
\end{align*}
which is indeed in the second integral of \eqref{lemma:technical-lemma-product-law-Hs0}. We shall impose 
$D_s \geq 3\cdot 2^{s+2}/\sqrt{s-1} + 2^{2s + 5}/(s - 3/2)$.

\subsection{Estimates of the pressure}$\,$

\noindent 
To complete the estimates of the momentum equation, we shall finally address the term in \eqref{energy-id-lemma-uc} related to the pressure. 
First, we remark that
\begin{align*}
    -
    \psca{\pa_x p_{\eta},
    2\mathbb{J} (\partial_t u)_\eta \!+\!u_\eta
    }_{H^{s,0}}
    &=
    \psca{\pa_x p_{\eta},
    2\mathbb{J} (\partial_t\partial_x u)_\eta \!+\!\partial_x u_\eta
    }_{H^{s,0}}=
    -
    \psca{p_{\eta},2\mathbb{J} (\partial_t \partial_y v)_\eta \!+\!\partial_y v_\eta}_{H^{s,0}} 
    \\
    &=
     e^{\R t}
    \psca{\pa_y p_{\eta}, 
     2\mathbb{J} (\partial_t v)_\eta \!+\!v_\eta}_{H^{s,0}}.
\end{align*}  
Furthermore, making use of the second equation in \eqref{eq:MHD-Prandtl-without-e}, we can decompose this term as
\begin{equation}\label{terms-related-to-the-pressure}
\begin{aligned}
    e^{\R t}
    \psca{\pa_y p_{\eta}, 
     2\mathbb{J} (\partial_t v)_\eta \!+\!v_\eta}_{H^{s,0}} 
    \!=\! 
     e^{\R t}
    \Big\{
    \big\langle\big(
            b_1 
            b_2 
            u 
            \big)_\eta ,  
            2 \mathbb{J}(\partial_t v)_\eta \!+\!v_\eta
    \big\rangle_{H^{s, 0}}
     +
    \big\langle 
    (b_1^2 v)_\eta ,   
    2 \mathbb{J}(\partial_t v)_\eta \!+\!v_\eta
    \big\rangle_{H^{s, 0}}
    +\\
    -
    \big\langle 
         \big(
        b_1
        \big( 
           \smallint_0^y
            \partial_t b_1 
        \big)
        \big)_\eta
        ,   
        2 \mathbb{J}(\partial_t v)_\eta \!+\!v_\eta
    \big\rangle_{H^{s, 0}}
    \Big\}.
\end{aligned}
\end{equation}
We begin with by estimating the inner product $\langle ( b_1 b_2 u )_\eta ,  2 \mathbb{J}(\partial_t v)_\eta \!+\!v_\eta
\rangle_{H^{s, 0}}$, where we can first localise a dissipation of the form $\mathcal{D}_{s+1/2}(t)^{1/2}$:
\begin{equation}\label{first-estimate-first-trilinear-term}
\begin{aligned}
    \big| 
    e^{\R t}
    \big\langle
            (b_1 b_2 u )_\eta ,   
    &         2\mathbb{J} (\partial_t v)_\eta \!+\!v_\eta
    \big\rangle_{H^{s, 0}}
    \big|
    \leq 
     e^{\R t}
    \big\| 
            (b_1 b_2 u )_\eta
    \big\|_{H^{s+\frac{1}{2}, 0}}
    \|  2\mathbb{J} (\partial_t v)_\eta \!+\!v_\eta \|_{H^{s-\frac{1}{2}, 0}}\\
    &\leq 
    e^{\R t}
    \big\| 
            (b_1 b_2 u )_\eta
    \big\|_{H^{s+\frac{1}{2}, 0}}
    \Big(
    \mathbb{J}
    \big\| 
        \smallint_0^y \partial_x (\partial_t u)_\eta 
    \big\|_{H^{s-\frac{1}{2}, 0}}
    +
    \big\| 
        \mathbb{J} \smallint_0^y \partial_x (\partial_t u)_\eta + 
        \smallint_0^y
    \partial_x u_\eta \big\|_{H^{s-\frac{1}{2}, 0}}
    \Big)
    \\
    &\leq 
    e^{\R t}
    \big\| 
            (b_1 b_2 u )_\eta
    \big\|_{H^{s+\frac{1}{2}, 0}}
    \Big(
    \mathbb{J}
    \big\| 
         \partial_x (\partial_t u)_\eta 
    \big\|_{H^{s-\frac{1}{2}, 0}}
    +
    \big\| 
        \mathbb{J}\partial_x (\partial_t u)_\eta + 
        \partial_x u_\eta \big\|_{H^{s-\frac{1}{2}, 0}}
    \Big)
    \\
    &\leq 
    e^{\R t}
    \big\| 
            (b_1 b_2 u )_\eta
    \big\|_{H^{s+\frac{1}{2}, 0}}
    \Big(
    \mathbb{J}
    \big\| 
           (\partial_t u)_\eta 
    \big\|_{H^{s+\frac{1}{2}, 0}}
    +
    \big\| 
        \mathbb{J} (\partial_t u)_\eta + 
         u_\eta \big\|_{H^{s+\frac{1}{2}, 0}}
    \Big)
    \\
    &\leq 
    e^{\R t}
    \big\| 
            (b_1 b_2 u )_\eta
    \big\|_{H^{s+\frac{1}{2}, 0}}
    2
    \sqrt{\mathcal{D}_{s+\frac{1}{2}}(t)}
\end{aligned}
\end{equation}
We shall now address the trilinear term $\| (b_1 b_2 u )_\eta \|_{H^{s+1/2, 0}}$ and we apply Lemma \ref{lemma:technical-lemma-product-law-Hs0} with $\sigma_1 =s+1/2  $, $\sigma_2 = s-3/2$, $f = b_{2}$ and $g = b_1 u$:
\begin{equation}\label{est1-where-we-need-s>2}
\begin{aligned}
    \big\| 
            (b_1 b_2 u )_\eta
    \big\|_{H^{s+\frac{1}{2}, 0}}
    &\leq 
    \frac{2^{s}}{\sqrt{s-2}}
    \bigg\{
    \| 
        \partial_y  b_{2, \eta}
    \|_{H^{s+\frac{1}{2}, 0}}
    \| (u b_1 )_\eta   \|_{H^{s-\frac{3}{2}, 0}}
    +
    \| 
        \partial_y  b_{2, \eta}
    \|_{H^{s-\frac{3}{2}, 0}}
    \| (u b_1 )_\eta   \|_{H^{s+\frac{1}{2}, 0}}
    \Big\}\\
    &\leq 
    \frac{2^{s}}{\sqrt{s-2}}
    \bigg\{
    \| 
        \partial_x  b_{1, \eta}
    \|_{H^{s+\frac{1}{2}, 0}}
    \| (u b_1 )_\eta   \|_{H^{s-\frac{3}{2}, 0}}
    +
    \| 
        \partial_x  b_{1, \eta}
    \|_{H^{s-\frac{3}{2}, 0}}
    \| (u b_1 )_\eta   \|_{H^{s+\frac{1}{2}, 0}}
    \Big\}
    \\
    &\leq 
    \frac{2^{s}}{\sqrt{s-2}}
    \bigg\{
    \| 
        b_{1, \eta}
    \|_{H^{s+\frac{3}{2}, 0}}
    \| (u b_1 )_\eta   \|_{H^{s-\frac{3}{2}, 0}}
    +
    \| 
        b_{1, \eta}
    \|_{H^{s-\frac{1}{2}, 0}}
    \| (u b_1 )_\eta   \|_{H^{s+\frac{1}{2}, 0}}
    \Big\}\\
    &\leq 
    \frac{2^{s}}{\sqrt{s-2}}
    \bigg\{
    2
    \sqrt{\mathcal{D}_{s+\frac{3}{2}}(t)}
    \| (u b_1 )_\eta   \|_{H^{s-\frac{3}{2}, 0}}
    +
    2
    \sqrt{\mathcal{E}_s(t)}
    \| (u b_1 )_\eta   \|_{H^{s+\frac{1}{2}, 0}}
    \bigg\}.
\end{aligned}
\end{equation}
Next, we apply twice Lemma \ref{lemma:technical-lemma-product-law-Hs0}, in order to deal with $\| (ub_1 )_\eta \|_{H^{s-3/2, 0}}$ and $\| ( u b_1)_\eta \|_{H^{s+1/2, 0}}$. For both cases we consider functions $f = u$ and $g = b_1$, however the regularities are considered $\sigma_1 =s-3/2 = \sigma_2 $ and 
$\sigma_1 =s+1/2$, $\sigma_2=s$, respectively. We gather 
\begin{align*}
    \big\| 
            (b_1 b_2 u )_\eta
    \big\|_{H^{s+\frac{1}{2}, 0}}
    &\leq 
    \frac{2^s}{\sqrt{s-2}}
    \bigg\{
    2\sqrt{D_{s+\frac{3}{2}}(t)} 
    \frac{2^{s-2}}{\sqrt{s-2}}
    \| \partial_y u_\eta       \|_{H^{s-\frac{3}{2}, 0}}
    \| b_{1,\eta}              \|_{H^{s-\frac{3}{2}, 0}}
    +\\
    &
    +
    2\sqrt{\mathcal{E}_s(t)}
    \frac{2^s}{\sqrt{s-1/2}}
    \bigg(
    \| \partial_y u_\eta   \|_{H^{s+\frac{1}{2}, 0}}
    \| b_{1,\eta}          \|_{H^{s, 0}}
    +
    \| \partial_y u_\eta   \|_{H^{s, 0}}
    \| b_{1,\eta}          \|_{H^{s+\frac{1}{2}, 0}}
    \bigg)
    \bigg\}
    \\
    &\leq 
    \frac{2^{s}}{\sqrt{s-2}}
    \bigg\{
    \frac{2^{s+1}}{\sqrt{s-2}}
    \sqrt{\mathcal{D}_{s+\frac{3}{2}}(t)}
    \frac{1}{\sqrt{\mathbb J}}
    \sqrt{\mathcal{E}_s(t)}
    2
    \sqrt{\mathcal{E}_s(t)}
    + \\ 
    &+
     \frac{2^{s+1}}{\sqrt{s-1/2}}
    \sqrt{\mathcal{E}_{s}(t)}
    \bigg(
    \frac{1}{\sqrt{\mathbb J}}
    \sqrt{\mathcal{D}_{s+\frac{1}{2}}(t)}
    2
    \sqrt{\mathcal{E}_{s}(t)}
    + 
    \frac{1}{\sqrt{\mathbb J}}
    \sqrt{\mathcal{E}_{s}(t)}
    2
    \sqrt{\mathcal{D}_{s+\frac{1}{2}}(t)}
    \bigg)
    \bigg\}\\
    & \leq 
    \frac{2^{2s+3}}{s-2}
    \frac{1}{\sqrt{\mathbb J}}
    \bigg\{
        \sqrt{\mathcal{D}_{s+\frac{3}{2}}(t)}
        +
        \sqrt{\mathcal{D}_{s+\frac{1}{2}}(t)}
    \bigg\}
    \mathcal{E}_s(t).
\end{align*}
Plugging the last inequality into \eqref{first-estimate-first-trilinear-term}, we eventually obtain
\begin{equation}\label{last-estimate-first-trilinear-term}
     \Big| 
           e^{\R t}
    \big\langle
            (b_1 b_2 u )_\eta ,   
             2
             \mathbb J
             (\partial_t v)_\eta \!+\!v_\eta
    \big\rangle_{H^{s, 0}}
    \Big|
    \leq 
    \frac{2^{2s+4}}{s-2}
    \frac{1}{\sqrt{\mathbb J}}
    \sqrt{\mathcal{D}_{s+\frac{3}{2}}(t)\mathcal{D}_{s+\frac{1}{2}}(t)}
    \mathcal{E}_s(t)
    +
    \frac{2^{2s+4}}{s-2}
    \frac{1}{\sqrt{\mathbb J}}
    \mathcal{E}_s(t)
    \mathcal{D}_{s+\frac{1}{2}}(t).
\end{equation}
We shall thus impose 
$D_s \geq  3\cdot 2^{s+2}/\sqrt{s-1} + 2^{2s + 5}/(s - 3/2) + 2^{2s+4}/(s-2)$.

\noindent 
Coming back to \eqref{terms-related-to-the-pressure}, we infer that a similar approach as the one used to show \eqref{last-estimate-first-trilinear-term} implies that
\begin{equation*}
    \big| 
           e^{\R t}
           \big\langle
           (b_1^2 v )_\eta ,   
           2 (\partial_t v)_\eta \!+\!v_\eta
    \big\rangle_{H^{s, 0}}
    \big|
   \leq 
    \frac{2^{2s+4}}{s-2}
    \frac{1}{\sqrt{\mathbb J}}
   \sqrt{\mathcal{D}_{s+\frac{3}{2}}(t)\mathcal{D}_{s+\frac{1}{2}}(t)}
    \mathcal{E}_s(t)
    +
    \frac{2^{2s+4}}{s-2}
    \frac{1}{\sqrt{\mathbb J}}
    \mathcal{E}_s(t)
    \mathcal{D}_{s+\frac{1}{2}}(t),
\end{equation*}
thus we must impose $D_s \geq  3\cdot 2^{s+2}/\sqrt{s-1} + 2^{2s + 5}/(s - 3/2) + 2^{2s+5}/(s-2)$.

\noindent 
To conclude the estimates of the momentum equation, we consider
$\langle (b_1 ( \smallint_0^y \partial_t b_1 ))_\eta,2 \mathbb{J}(\partial_t v)_\eta \!+\!v_\eta \rangle_{H^{s, 0}}$, the last term of \eqref{terms-related-to-the-pressure}.
\begin{align*}
     e^{\R t}
    \big\langle 
        & \big(
        b_1
        \big( 
           \smallint_0^y
            \partial_t b_1 
        \big)
        \big)_\eta
        ,   
        2 \mathbb{J}(\partial_t v)_\eta \!+\!v_\eta
    \big\rangle_{H^{s, 0}}
    \leq 
    e^{\R t}
    \big\| 
        \big(
        b_1
        \big( 
           \smallint_0^y
            \partial_t b_1 
        \big)
        \big)_\eta
    \big\|_{H^{s+\frac{1}{2}, 0}}
    2\sqrt{D_{s+\frac{1}{2}(t)}}\\
    &\leq 
    e^{\R t}
    \frac{2^{s}}{\sqrt{s-1/2}}
    \Big(
    \|     b_{1,\eta} \|_{H^{s+\frac{1}{2},0}}
    \|   \partial_t  b_{1,\eta}\|_{H^{s,0}}
    +
    \|     b_{1,\eta} \|_{H^{s,0}}
    \|   \partial_t  b_{1,\eta}\|_{H^{s+\frac{1}{2},0}}
    \Big)
    2
    \sqrt{D_{s+\frac{1}{2}(t)}}\\
    &\leq 
    e^{\R t}
    \frac{2^{s+1}}{\sqrt{s-1/2}}
    \Big(
    2\sqrt{\mathcal{D}_{s+\frac{1}{2}}(t)}
    \frac{2{\rm Pr}_m}{\kappa}
     \sqrt{\mathcal{E}_{s}(t)}
    +
    2
    \sqrt{\mathcal{E}_{s}(t)}
    \frac{2{\rm Pr}_m}{\kappa}
    \sqrt{\mathcal{D}_{s+\frac{1}{2}}(t)}
    \Big)
    \sqrt{D_{s+\frac{1}{2}(t)}}\\
    &\leq 
    e^{\R t}
    \frac{2^{s+4}}{\sqrt{s-1/2}}
    \frac{{\rm Pr}_m}{\kappa}
    \sqrt{\mathcal{E}_{s}(t)}
    D_{s+\frac{1}{2}}(t).
\end{align*}
We shall impose $D_s \geq  3\cdot 2^{s+2}/\sqrt{s-1} + 2^{2s + 5}/(s - 3/2) + 2^{2s+5}/(s-2)+ 2^{2s+4}/\sqrt{s-1/2}$. This concludes the estimates related to the momentum equation.

\subsection{Estimates related to the equation of $b_1$}$\,$

\noindent 
In this section we cope with the equation of $b_1$ in the main system \eqref{eq:MHD-Prandtl-without-e}, more precisely we deal with
\begin{equation}\label{eq:b1-sec:final-estimates}
\begin{aligned}
    \frac{\kappa}{{\rm Pr}_m}
    \partial_{tt} b_1 
    +
    \partial_t b_1 
    + u \partial_x b_1 
    &+
    v
    \partial_y b_1 
    - 
    \frac{1}{{\rm Pr}_m}
    \partial_{yy}^2 b_1 
     = 
    b_1\partial_x u + 
    b_2\partial_y u.
\end{aligned}
\end{equation}
By coupling the definition of $b_{1,\eta}  = e^{(\tau_0-\eta )(1+|D_x|)}b_1$ and equation \eqref{eq:b1-sec:final-estimates}, we remark that the function $(t,x,y)\in (0,T)\times \mathbb{R}\to  e^{\R t/2} b_{1,\eta}(t,x,y)$ is solution of the following equation
\begin{equation}\label{eq:b1eta}
\begin{aligned}
     \frac{\kappa}{{\rm Pr}_m}
    \partial_t
    \big( 
        e^{\frac{\R}{2} t}
        &(\partial_t b_1 )_\eta
    \big)
    +
    e^{\frac{\R}{2} t}
    (\partial_t  b_1)_{\eta} 
    +
    \eta'(t) (1+|D_x|)\big(
    e^{\frac{\R}{2} t}
    (\partial_t  b_1)_{\eta}
    \big)
    +
     e^{\frac{\R}{2} t}
     (u\partial_xb_1)_{\eta}
     +\\
     &+
     e^{\frac{\R}{2} t}
     (v\partial_yb_1)_{\eta}
     -
    e^{\frac{\R}{2} t}
    \partial_y^2
    \big(
       e^{\frac{\R}{2} t}
        b_{1,\eta}
    \big)
    =
     \frac{\kappa}{{\rm Pr}_m}
    \frac{\R}{2}
    e^{\frac{\R}{2} t}
    (\partial_t b_1)_{\eta}
    +
    e^{\frac{\R}{2} t}
    (b_1\partial_xu)_{\eta}+
   e^{\frac{\R}{2} t}
    (b_2\partial_yv)_{\eta}.
\end{aligned}
\end{equation}
With a similar approach as the one used to prove inequality \eqref{energy-id-lemma-uc}, we infer that the following estimate holds true:
\begin{equation}\label{energy-id-lemma-b1c}
\begin{aligned}
    \frac{d}{dt}
    \bigg[
         e^{\R t}
        \bigg(
        \underbrace
        {
         \frac{\kappa^2}{2{\rm Pr}_m^2}
        \| 
        (\partial_t b_1)_{\eta}   
        \|_{H^{s,0}}^2 
        +
         \frac{1}{2}
        \Big\| 
        \frac{\kappa}{{\rm Pr}_m}
        (\partial_t b_1)_{\eta}   
        \!+\!
        b_{1,\eta}
        \Big\|_{H^{s,0}}^2 
        \!\!\!
        +\!
        \frac{\kappa}{{\rm Pr}_m^2}
        \| 
         \partial_y b_{1,\eta}   
        \|_{H^{s,0}}^2 \!
        }_{\text{$b_1$-terms in }\mathcal{E}_{s}(t)}
        +
        \frac{\kappa}{{\rm Pr}_m}
        \eta'(t)
        \hspace{-0.2cm}
        \underbrace{
        \frac{1}{2}
        \| 
             b_{1,\eta}
        \|_{H^{s+\frac{1}{2}, 0}}^2
        }_{\text{$b_1$-term in }\mathcal{E}_{s+\frac{1}{2}}(t)}
        \!\!\!
        +\\ 
        +
        \eta'(t)^2
        \frac{\kappa^2}{{\rm Pr}_m^2}
        \hspace{-0.3cm}
         \underbrace{
        \| 
             b_{1,\eta}
        \|_{H^{s+1, 0}}^2
         }_{\text{$b_1$-term in }\mathcal{E}_{s+1}(t)}
    \Big)
    \bigg] 
   +
    \eta'(t)
      e^{\R t}
     \bigg\{ 
     \underbrace{
        \frac{\kappa^2}{2{\rm Pr}_m^2}
        \| 
            (\partial_t b_1)_{\eta}       
        \|_{H^{s+\frac12,0}}^2
        \!+\! 
        \frac{1}{2}
        \| 
            \frac{\kappa}{{\rm Pr}_m}
            (\partial_t b_1)_{\eta} \!+\! u_\eta       
        \|_{H^{s+\frac12,0}}^2 
     }_{\text{$b_1$-terms in }\mathcal{D}_{s+\frac{1}{2}}(t)}
     +\\+
    \underbrace{    
        \frac{2\kappa}{2{\rm Pr}_m^2}
     \| \partial_y b_{1,\eta}     \|_{H^{s+\frac12, 0}}^2
     \!+\!
     \frac{\kappa^2}{{\rm Pr}_m^2}
     \| \partial_t (b_{1,\eta})   \|_{H^{s+\frac12,0}}^2 
     \!+\! 
     \frac{3}{8}
     \| b_{1,\eta} \|_{H^{s+\frac12,0}}^2
     }_{\text{$b_1$-terms in }\mathcal{D}_{s+\frac{1}{2}}(t)}
    \bigg\}
    +
     \frac{\kappa}{{\rm Pr}_m}
     \eta'(t)^2
      e^{\R t}
     \hspace{-0.3cm}
     \underbrace{
     \frac{3}{4}
     \| u_\eta \|_{H^{s+1,0}}^2
     }_{\text{$b_1$-term in }\mathcal{D}_{s+1}(t)}
     \hspace{-0.2cm}
     + \\
    +
    \frac{\kappa^2}{{\rm Pr}_m^2}
    \eta'(t)^3
     e^{\R t}
    \hspace{-0.3cm}
    \underbrace{
    \| b_{1,\eta} \|_{H^{s+\frac{3}{2},0}}^2
    }_{\text{$b_1$-terms in }\mathcal{D}_{s+\frac{3}{2}}(t)}
    \hspace{-0.4cm}
      -
    2
    \frac{\kappa^2}{{\rm Pr}_m^2}
     e^{\R t}
    \eta'(t)\eta''(t)
    \hspace{-0.3cm}
     \underbrace{
     \big\| 
         b_{1,\eta}
    \big\|_{H^{s+1, 0}}^2
    }_{ \text{$b_1$-term in }\mathcal{E}_{s+1}(t)}
    \hspace{-0.3cm}
    -
    \frac{\kappa}{{\rm Pr}_m}
     e^{\R t}
    \eta''(t)
    \hspace{-0.2cm}
    \underbrace{
    \frac{1}{2}
     \big\| 
         b_{1,\eta}
    \big\|_{H^{s+\frac{1}{2}, 0}}^2
    }_{ \text{$b_1$-term in }\mathcal{E}_{s+\frac{1}{2}}(t)}
    +\\
     +
    \frac{ e^{\R t}}{2}
    \underbrace{
    \Big(
    \frac{1}{{\rm Pr}_m}
    \|  \pa_y   b_{1,\eta}      \|_{H^{s,0}}^2
    \!+\!
    \frac{\kappa}{{\rm Pr}_m}
    \| 
        (\partial_t b_1)_{\eta}   
    \|_{H^{s,0}}^2
    \Big)
    }_{ \text{$b_1$-term in }\mathcal{D}_{s}(t)}
    \leq 
      e^{\R t}
    \Big\{
    \Big|\psca{ (u\pa_xb_1)_\eta,
    2 (\partial_t b_1)_\eta \!+\! b_{1,\eta}}_{H^{s,0}}
    \Big|
    +\\+
    \Big|
   \psca{
        (v\pa_yb_1)_\eta, 
        \frac{2\kappa}{{\rm Pr}_m}(\partial_t b_1)_\eta \!+\! b_{1,\eta}
   }_{H^{s,0}}
   \Big|
    +
    \Big|
    \psca{
        (b_1\partial_xu)_{\eta},
        \frac{2\kappa}{{\rm Pr}_m}(\partial_t b_1)_\eta \!+\! b_{1,\eta}
    }_{H^{s,0}}
    \Big|
    +
    \\
    +
     \Big|
    \psca{
        (b_2\partial_xv)_{\eta},
        \frac{2\kappa}{{\rm Pr}_m}(\partial_t b_1)_\eta \!+\! b_{1,\eta}
   }_{H^{s,0}}\!
    \Big|\Big\}.
\end{aligned}
\end{equation}
We next proceed to estimate each term on the right-hand side of \eqref{energy-id-lemma-b1c}. For each estimated term, we will determine a suitable (increasing) lower bound of the constant $D_s$. The last term will therefore provide the exact form of $D_s$. We recall that from the estimate of the momentum equation,  $D_s \geq  3\cdot 2^{s+2}/\sqrt{s-1} + 2^{2s + 5}/(s - 3/2) + 2^{2s+5}/(s-2)+ 2^{2s+4}/\sqrt{s-1/2}$, momentarily.

\noindent
The first term in \eqref{energy-id-lemma-b1c} we deal with is the convection
\begin{equation}\label{secu:first-estimateb1}
\begin{aligned}
    &
    e^{\R t} \Big|
    \big\langle
        (u\pa_xb_1)_\eta, 
        \frac{2\kappa}{{\rm Pr}_m}(\partial_t b_1)_\eta \!+\! b_{1,\eta}
    \big\rangle_{H^{s,0}}
    \Big|
    \leq \\
    &\leq
     e^{\R t}
    \| (u\pa_xb_{1,\eta})_\eta   \|_{H^{s-\frac{1}{2},0}}
    \Big(
    \frac{\kappa}{{\rm Pr}_m}
     \|
     (\partial_tb_1)_\eta         \|_{H^{s+\frac{1}{2},0}}
     +
    \Big\| 
    \frac{\kappa}{{\rm Pr}_m}
    (\partial_t b_1)_\eta +b_{1,\eta}       
    \Big\|_{H^{s+\frac{1}{2},0}}
    \Big)\\
    &\leq
     e^{\R t}
    \| (u\pa_xb_1)_\eta   \|_{H^{s-\frac{1}{2},0}}
    2
    \Big(
    \frac{\kappa^2}{2{\rm Pr}_m^2}
     \|
     (\partial_t b_1)_\eta         \|_{H^{s+\frac{1}{2},0}}^2
     +
     \frac{1}{2}
    \Big\| 
     \frac{\kappa}{{\rm Pr}_m}(\partial_t b_1)_\eta + b_{1,\eta}       
     \Big\|_{H^{s+\frac{1}{2},0}}^2
    \Big)\\
    &\leq
    2
     e^{\R t}
    \| (u\pa_xb_1)_\eta   \|_{H^{s-\frac{1}{2},0}}
    \sqrt{\mathcal{D}_{s+\frac{1}{2}}(t)}
\end{aligned}
\end{equation}
We are in the position to apply Lemma \ref{lemma:technical-lemma-product-law-Hs0} with the regularities $\sigma_1=\sigma_2 = s-1/2>1/2$ and the functions  $f(\cdot) = u(t,\cdot),\,g(\cdot) = \partial_x u (t,\cdot)$. We deduce therefore that
\begin{align*}
    \| (u\pa_xb_1)_\eta  & \|_{H^{s-\frac{1}{2},0}}
    \leq  
    \frac{2^{s-1}}{\sqrt{s - 1}}
    2
    \| \partial_y u_\eta \|_{H^{s-\frac{1}{2},0}}
    \| \partial_x b_{1,\eta} \|_{H^{s-\frac{1}{2},0}}
    \leq 
    \frac{2^{s}}{\sqrt{s - 1}}
    \| \partial_y u_\eta \|_{H^{s,0}}
    \| b_{1,\eta} \|_{H^{s+\frac{1}{2},0}}\\
    &\leq 
    \frac{2^{s}}{\sqrt{s - 1}}
    \frac{1}{\sqrt{\mathbb{J}}}
    \sqrt{\mathcal{E}_s(t)}
    2
    \sqrt{\mathcal{D}_{s+\frac{1}{2}}(t)}.
\end{align*}
Plugging the above estimate to the original convective term \eqref{secu:first-estimateb1}, we eventually gather that
\begin{equation*}
    \begin{aligned}
        e^{\R t}
        \Big|
        \big\langle
            (u\pa_xb_1)_\eta, 
            \frac{2\kappa}{{\rm Pr}_m}(\partial_t b_1)_\eta \!+\! b_{1,\eta}
        \big\rangle_{H^{s,0}}
        \Big|
        \leq
        \frac{2^{s+2} }{\sqrt{s - 1}}
        e^{\R t} \frac{1}{\sqrt{\mathbb{J}}}
        \sqrt{\mathcal{E}_{s}(t)}
       \mathcal{D}_{s+\frac{1}{2}}(t).
    \end{aligned}
\end{equation*}
Therefore we require 
$D_s \geq  2^{s+4}/\sqrt{s-1} + 2^{2s + 5}/(s - 3/2) + 2^{2s+5}/(s-2)+ 2^{2s+4}/\sqrt{s-1/2}$.

\noindent
Next, we treat the second term on the right-hand side of \eqref{energy-id-lemma-u}:
\begin{align*}
      e^{\R t}
    \Big|
    \big\langle
        (v\pa_yb_1)_\eta, 
        \frac{2\kappa}{{\rm Pr}_m}(\partial_t b_1)_\eta \!+\! b_{1,\eta}
    \big\rangle_{H^{s,0}}
    \Big|
     &\leq 
      e^{\R t}
    \big\|
        ( 
        v
        \pa_yb_1
        )_\eta
    \big\|_{H^{s-\frac{1}{2},0}}
    2
    \sqrt{\mathcal{D}_{s+\frac{1}{2}}(t)}.
\end{align*}
We apply  Lemma \ref{lemma:technical-lemma-product-law-Hs0} once more with regularities $\sigma_1=\sigma_2 = s-1/2>1/2$, with functions  $f  =  v$ and $g  = \partial_y b_1$. Hence
\begin{align*}
      e^{\R t}
    \Big|
    \big\langle
       & (v\pa_yb_1)_\eta, 
        \frac{2\kappa}{{\rm Pr}_m}(\partial_t b_1)_\eta \!+\! b_{1,\eta}
    \big\rangle_{H^{s,0}}
    \Big|
    \leq 
    \frac{2^{s-1}}{\sqrt{s - 1}}
      e^{\R t}
    2
    \| 
        \partial_y v_\eta 
    \|_{H^{s-\frac{1}{2},0}} 
    \|
        \partial_y b_{1,\eta}
    \|_{H^{s-\frac{1}{2},0}}
    2
    \sqrt{\mathcal{D}_{s+\frac{1}{2}}(t)}
     \\
    &\leq 
    \frac{2^{s+1}}{\sqrt{s - 1}}
      e^{\R t}
     \big\| 
        \partial_x  u_\eta 
    \big\|_{H^{s-\frac{1}{2},0}} 
    \|
        \partial_y b_{1,\eta}
    \|_{H^{s-\frac{1}{2},0}}
    \sqrt{\mathcal{D}_{s+\frac{1}{2}}(t)} \\
     &\leq 
    \frac{2^{s+1}}{\sqrt{s - 1}}
      e^{\R t}
    \| 
         u_\eta 
    \|_{H^{s+\frac{1}{2},0}} 
    \|
        \partial_y b_{1,\eta}
    \|_{H^{s-\frac{1}{2},0}}
    \sqrt{\mathcal{D}_{s+\frac{1}{2}}(t)} \\
    &\leq 
    \frac{2^{s+2}}{\sqrt{s - 1}}
    e^{\R t}
    \|
        \partial_y b_{1,\eta}
    \|_{H^{s,0}}
     \mathcal{D}_{s+\frac{1}{2}}(t)
    \leq 
    \frac{2^{s+2}}{\sqrt{s - 1}}
    e^{\R t}
    \frac{{\rm Pr}_m}{\sqrt{\kappa}}
    \sqrt{\mathcal{E}_{s}(t)}
    \mathcal{D}_{s+\frac{1}{2}}(t),
\end{align*}
which is still on the second integral of \eqref{lemma:technical-lemma-product-law-Hs0}. Hence $D_s$ must satisfy at least $D_s \geq  { 5\cdot 2^{s+2}/\sqrt{s-1}} + 2^{2s + 5}/(s - 3/2) + 2^{2s+5}/(s-2)+ 2^{2s+4}/\sqrt{s-1/2}$.

\noindent 
Next, we deal with 
\begin{equation*}
\begin{aligned}
    e^{\R t}
    \big|
    \langle 
        (b_1\partial_x u )_\eta , 
        \frac{2\kappa}{{\rm Pr}_m}(\partial_t b_1)_\eta + b_{1,\eta}
    \rangle_{H^{s,0}}
    \big|
    &\leq 
    e^{\R t}\| (b_1\partial_x u )_\eta \|_{H^{s-\frac{1}{2},0}}
    \Big(
        \Big\|
            \frac{\kappa}{{\rm Pr}_m}(\partial_t b_1)_\eta \!+\! b_{1,\eta} 
        \Big\|_{H^{s+\frac{1}{2},0}} 
        +
        \Big\|\frac{\kappa}{{\rm Pr}_m}(\partial_t b_1)_\eta \Big\|_{H^{s+\frac{1}{2},0}}
    \Big)\\
     &\leq 
      e^{\R t}
    \| (b_1\partial_x u )_\eta \|_{H^{s-\frac{1}{2},0}}
    2
    \sqrt{\mathcal{D}_{s+\frac{1}{2}}(t)}.
\end{aligned}
\end{equation*}
Applying Lemma \ref{lemma:technical-lemma-product-law-Hs0} with regularities $\sigma_1=\sigma_2 = s-1/2>1/2$, with functions  
$f = b_1$ and $g  = \partial_x u$, we gather 
\begin{equation*}
\begin{aligned}
    e^{\R t}
    \big|
    \langle 
        (b_1\partial_x u )_\eta , 
        \frac{2\kappa}{{\rm Pr}_m}
        (\partial_t b_1)_\eta + b_{1,\eta}
    \rangle_{H^{s,0}}
    \big|
    &\leq 
     \frac{2^{s}
      e^{\R t}}{\sqrt{s-1}}
    \| \partial_y b_1       \|_{H^{s-\frac{1}{2},0}}
    \| \partial_x u_\eta    \|_{H^{s-\frac{1}{2},0}}
    2
    \sqrt{\mathcal{D}_{s+\frac{1}{2}}(t)}\\
    &\leq 
     \frac{2^{s+2}
      e^{\R t}}{\sqrt{s-1}}
    \| \partial_y b_1       \|_{H^{s-\frac{1}{2},0}} 
    \mathcal{D}_{s+\frac{1}{2}}(t)
    \leq 
    \frac{2^{s+2}
      e^{\R t}}{\sqrt{s-1}} \frac{\rm Pr_m}{{\sqrt{\kappa}}}
     \sqrt{\mathcal{E}_{s}(t)}
     \mathcal{D}_{s+\frac{1}{2}}(t).
\end{aligned}
\end{equation*}
Hence $D_s$ must satisfy at least $D_s \geq 6\cdot 2^{s+2}/\sqrt{s-1} + 2^{2s+5}/(s - 3/2)+  2^{2s+5}/(s - 2) + 2^{2s + 4}/\sqrt{s-1/2}$.

\noindent 
Finally, we deal with 
\begin{equation*}
\begin{aligned}
    e^{\R t}
    \big|
    \langle 
        (b_2\partial_y u )_\eta , 
        \frac{2\kappa}{{\rm Pr}_m}(\partial_t b_1)_\eta + b_{1,\eta}
    \rangle_{H^{s,0}}
    \big|
    &\leq 
    e^{\R t}
    \| (b_2\partial_y u )_\eta \|_{H^{s-\frac{1}{2},0}}
    \Big(
        \|\frac{\kappa}{{\rm Pr}_m}(\partial_t b_1)_\eta \!+\! b_{1,\eta} \|_{H^{s+\frac{1}{2},0}} 
        +
        \|\frac{\kappa}{{\rm Pr}_m}(\partial_t b_1)_\eta \|_{H^{s+\frac{1}{2},0}}
    \Big)
    \\
     &\leq 
      e^{\R t}
    \| (b_2\partial_y u )_\eta \|_{H^{s-\frac{1}{2},0}}
    2
    \sqrt{\mathcal{D}_{s+\frac{1}{2}}(t)}.
\end{aligned}
\end{equation*}
Applying Lemma \ref{lemma:technical-lemma-product-law-Hs0} with regularities $\sigma_1=\sigma_2 = s-1/2>1/2$, with functions  
$f = b_1$ and $g  = \partial_x u$, we gather 
\begin{equation*}
\begin{aligned}
    e^{\R t}
    \big|
    \langle 
        (b_2\partial_y u )_\eta , 
         \frac{2\kappa}{{\rm Pr}_m}(\partial_t b_1)_\eta + b_{1,\eta}
    \rangle_{H^{s,0}}
    \big|
    &\leq 
     \frac{2^{s}
      e^{\R t}}{\sqrt{s-1}}
    \| \partial_y b_{2,\eta}       \|_{H^{s-\frac{1}{2},0}}
    \| \partial_y u_\eta    \|_{H^{s-\frac{1}{2},0}}
    2
    \sqrt{\mathcal{D}_{s+\frac{1}{2}}(t)}\\
    &\leq 
     \frac{2^{s+1}
      e^{\R t}}{\sqrt{s-1}}
    \| \partial_x b_{1,\eta}       \|_{H^{s-\frac{1}{2},0}}
    \| \partial_y u_\eta    \|_{H^{s-\frac{1}{2},0}}
    \sqrt{\mathcal{D}_{s+\frac{1}{2}}(t)}\\
    &\leq 
     \frac{2^{s+1}
      e^{\R t}}{\sqrt{s-1}}
    \|  b_{1,\eta}                 \|_{H^{s+\frac{1}{2},0}}
    \| \partial_y u_\eta    \|_{H^{s-\frac{1}{2},0}}
    \sqrt{\mathcal{D}_{s+\frac{1}{2}}(t)}\\
    &\leq 
     \frac{2^{s+2}
      e^{\R t}}{\sqrt{s-1}}
    \| \partial_y u_\eta       \|_{H^{s-\frac{1}{2},0}}
    \mathcal{D}_{s+\frac{1}{2}}(t)
    \leq 
    \frac{2^{s+2}
      e^{\R t}}{\sqrt{s-1}}\frac{1}{\sqrt{\mathbb J}}
     \sqrt{\mathcal{E}_{s}(t)}
     \mathcal{D}_{s+\frac{1}{2}}(t).
\end{aligned}
\end{equation*} 
Hence $D_s$ must satisfy at least $D_s \geq 7\cdot 2^{s+2}/\sqrt{s-1} + 2^{2s+5}/(s - 3/2)+  2^{2s+5}/(s - 2) + 2^{2s + 4}/\sqrt{s-1/2}$. 
Since, we have concluded our estimates, we are now in the position to set a value of $D_s$. For instance, as compact form. we can consider $D_s$ as 
\begin{equation}\label{Ds-explicit-determined}
    D_s = \frac{2^{2s+6}}{s-2}
    \left(
        1+\frac{s-2}{\sqrt{s-1}}
    \right)
\end{equation}
This concludes the proof of Proposition \ref{prop:the-overall-final-estimate-in-eta}. 

\appendix

\section{A suitable product law}
 
\noindent 
This appendix is devoted to the proof of Lemma \ref{lemma:technical-lemma-product-law-Hs0}. We recall here the statement.
\begin{lemma}\label{appx:lemma-product-eta}
 	   Let  $f,\,g:\mathbb R\times (0,1)\to \mathbb R$ be two functions such that $f_\eta$, $g_\eta$ and $\partial_y f_\eta$ belong to $H^{\sigma_1,0}(\mathbb R\times (0,1))$ with $\sigma_1>1/2$ (and thus also to $H^{\sigma_2,0}(\mathbb R\times (0,1))$, for any $\sigma_2\leq \sigma_1$). 
 	   Furthermore, assume that $f\equiv 0$ in $y = 0$ in the sense of trace. Then
    \begin{equation*}
    \begin{aligned}
        \| (f g)_\eta \|_{H^{\sigma_1, 0}} 
        \leq
            \frac{2^{\sigma_1-\frac{1}{2}}}{\sqrt{\sigma_2 - \frac{1}{2}}}
            \bigg(
            \| \partial_y f_\eta \|_{H^{\sigma_1, 0}}
            \|  g_\eta           \|_{H^{\sigma_2, 0}}
            +
            \| \partial_y f_\eta \|_{H^{\sigma_2, 0}}
            \| g_\eta           \|_{H^{\sigma_1, 0}}
            \bigg)
    \end{aligned}
    \end{equation*}
     for any regularities $\sigma_2 \in (1/2, \sigma_1]$.
 \end{lemma}

\begin{proof}
According to the definition of $(fg)_{\eta}$ and the anisotropic Sobolev norm $H^{\sigma_1,0}$, we have
\ben\label{def-fg-norm-1}
	\| (f g)_\eta \|_{H^{\sigma_1, 0}}^2 &=&\int_0^1\int_{\mathbb{R}}(1+|\xi|)^{2\sigma_1}e^{2(\tau_0-\eta(t))(1+|\xi|)}|\widehat{fg}(\xi,y)|^2~d\xi dy\notag\\
	&=&\int_0^1\int_{\mathbb{R}}(1+|\xi|)^{2\sigma_1}e^{2(\tau_0-\eta(t))(1+|\xi|)}|\widehat{f}*\widehat{g}(\xi,y)|^2~d\xi dy\notag\\
	&=&\int_0^1\int_{\mathbb{R}}(1+|\xi|)^{2\sigma_1}e^{2(\tau_0-\eta(t))(1+|\xi|)}\left|\int_{\mathbb{R}}\widehat{f}(\xi-\eta,y)\widehat{g}(\eta,y)d\eta\right|^2~d\xi dy.
\een
We claim that 
\ben\label{p-m-ineq}
(1+|\xi|)^{2\sigma_1}\leq 2^{2\sigma_1-1}\left[(1+|\xi-\eta|)^{2\sigma_1}+(1+|\eta|)^{2\sigma_1}\right],
\een
for $\sigma_1>\frac12$. Indeed, we have 
\beno
(1+|\xi|)^{2\sigma_1}&\leq& (1+|\xi-\eta|+|\eta|)^{2\sigma_1}\\
&\leq& \left[(1+ |\xi -\eta|)^{2\sigma_1} + |\eta|^{2\sigma_1}\right]^{2\sigma_1 \times \frac{1}{2\sigma_1}}\big( 1+ 1\big)^{(1-\frac{1}{2\sigma_1})\times 2\sigma_1},
\eeno
where in the last step we have used the H\"older inequality $(a+b)\leq (a^\theta +b^\theta)^{1/\theta}2^{1-1/\theta}$, 
with $\theta=2\sigma_1$. Continuing our estimate, we have therefore
\beno
(1+|\xi|)^{2\sigma_1}&\leq& 2^{2\sigma_1 -1} \left((1+ |\xi -\eta|)^{2\sigma_1} + |\eta|^{2\sigma_1} \right) \\
&\leq& 2^{2\sigma_1-1}\left((1+|\xi-\eta|)^{2\sigma_1}+(1+|\eta|)^{2\sigma_1}\right),
\eeno
which is indeed \eqref{p-m-ineq}.
Next, we bring inequality \eqref{p-m-ineq} to \eqref{def-fg-norm-1}, thus
\ben\label{def-fg-norm-2}
\| (f g)_\eta \|_{H^{\sigma_1, 0}}^2 
&=&2^{2\sigma_1-1} \int_0^1\int_{\mathbb{R}}(1+|\xi-\eta|)^{2\sigma_1}e^{2(\tau_0-\eta(t))(1+|\xi|)}\left|\int_{\mathbb{R}}\widehat{f}(\xi-\eta,y)\widehat{g}(\eta,y)d\eta\right|^2~d\xi dy\notag\\
&~&+ 2^{2\sigma_1-1} \int_0^1\int_{\mathbb{R}}(1+|\eta|)^{2\sigma_1}e^{2(\tau_0-\eta(t))(1+|\xi|)}\left|\int_{\mathbb{R}}\widehat{f}(\xi-\eta,y)\widehat{g}(\eta,y)d\eta\right|^2~d\xi dy\notag\\
&\leq&2^{2\sigma_1-1} \int_0^1\int_{\mathbb{R}}\left|\int_{\mathbb{R}}(1+|\xi-\eta|)^{\sigma_1}e^{(\tau_0-\eta(t))(1+|\xi-\eta|)}\widehat{f}(\xi-\eta,y)e^{(\tau_0-\eta(t))(1+|\eta|)}\widehat{g}(\eta,y)d\eta\right|^2~d\xi dy\notag\\
&~&+2^{2\sigma_1-1} \int_0^1\int_{\mathbb{R}}\left|\int_{\mathbb{R}}e^{(\tau_0-\eta(t))(1+|\xi-\eta|)}\widehat{f}(\xi-\eta,y)(1+|\eta|)^{\sigma_1}e^{(\tau_0-\eta(t))(1+|\eta|)}\widehat{g}(\eta,y)d\eta\right|^2~d\xi dy\notag.
\een
We are now in the condition to apply the Young's inequality for the convolution of functions, more precisely:
\beno
\| (f g)_\eta \|_{H^{\sigma_1, 0}}^2 &\leq& 2^{2\sigma_1-1} \| (1+|\xi|)^{\sigma_1} f_{\eta} \|_{L^2_\xi L^\infty_y}^2 \|  g_{\eta} \|_{L^1_\xi L^2_y}^2+2^{2\sigma_1-1} \| (1+|\xi|)^{\sigma_1} g_{\eta} \|_{L^2}^2 \|  f_{\eta} \|_{L^1_\xi L^\infty_y}^2.
\eeno
Next, since $\sigma_2\geq 1/2$ and $f(t,x,y) = \int_0^y\partial_y f(t,x,z)dz $, we have that  
\begin{equation*}
\begin{aligned}
 \|  f_{\eta} \|_{L^1_\xi L^\infty_y}^2 &= 
    \|  (1+|\xi|)^{-\sigma_2} (1+|\xi|)^{\sigma_2} f_{\eta}
    \|_{L^1_\xi L^\infty_y}\\
    &\leq 
    \bigg(
        \int_{\mathbb R}\frac{1}{(1+|\xi|)^{2\sigma_2}}d\xi
    \bigg)
    \| (1+|\xi|)^{\sigma_2} f_{\eta}  \|_{L^2_\xi L^\infty_y}^2
 \leq \frac{1}{\sigma_2-\frac12}
  \|  \pa_y f_{\eta} \|_{H^{\sigma_2, 0}}^2.
\end{aligned}
\end{equation*}
Moreover 
\begin{align*}
    \| (1+|\xi|)^{\sigma_1} f_{\eta} \|_{L^2_\xi L^\infty_y}^2 
    &= 
    \int_{\mathbb{R}}
    (1+|\xi|)^{2\sigma_1} 
    \sup_{y\in (0,1)}
    \Big|
    \smallint_0^y \partial_y f_{\eta} (\xi,z)dz
    \Big|^2\\
    &\leq 
    \int_{\mathbb{R}}
    (1+|\xi|)^{\sigma_1} 
    \sup_{y\in (0,1)}
    y
    \smallint_0^y |\partial_y f_{\eta} (\xi,z)|^2dz\\
    &\leq 
    \int_{\mathbb{R}}
    (1+|\xi|)^{\sigma_1}
    \smallint_0^1 |\partial_y f_{\eta} (\xi,z)|^2dz = 
    \| \partial_y f_{\eta} \|_{H^{\sigma_1, 0}}^2.
\end{align*}
Summarising, we finally deduce that
\beno
\| (f g)_\eta \|_{H^{\sigma_1, 0}}^2 &\leq& \frac{ 2^{2\sigma_1-1}}{\sigma_2-\frac12} \| \partial_y f_{\eta} \|_{H^{\sigma_1, 0}}^2 \|  g_{\eta} \|_{H^{\sigma_2,0}}^2+\frac{ 2^{2\sigma_1-1}}{\sigma_2-\frac12} \|\pa_y f_{\eta} \|_{H^{\sigma_2,0}}^2 \|  g_{\eta} \|_{H^{\sigma_1,0}}^2.
\eeno
By applying the square root to the above inequality, we
conclude the proof of the Lemma.

\end{proof} 

\subsection*{Data Availability Statement} 
Data sharing is not applicable to this article, since no datasets were generated or analysed during the current study.

\subsection*{Conflict of interest} 
The authors declare that they have no conflict of interest. 

\subsection*{Acknowledgment} 
The second author was partially supported by the Bavarian Funding Programme for the Initiation of International Projects (förderkennzeichen: BayIntAn\_UWUE\_2022\_139)

\bibliography{bibliography.bib}{}

\begin{thebibliography}{10}

\bibitem{MR4362378}
N.~Aarach.
\newblock Hydrostatic approximation of the 2{D} {MHD} system in a thin strip
  with a small analytic data.
\newblock {\em J. Math. Anal. Appl.}, 509(2):Paper No. 125949, 55, 2022.

\bibitem{MR3942552}
B.~Abdelhedi.
\newblock Global existence of solutions for hyperbolic {N}avier-{S}tokes
  equations in three space dimensions.
\newblock {\em Asymptot. Anal.}, 112(3-4):213--225, 2019.

\bibitem{Gallagher2D3D}
D.~Ars\'{e}nio and I.~Gallagher.
\newblock Solutions of {N}avier-{S}tokes-{M}axwell systems in large energy
  spaces.
\newblock {\em Trans. Amer. Math. Soc.}, 373(6):3853--3884, 2020.

\bibitem{MR2045417}
Y.~Brenier, R.~Natalini, and M.~Puel.
\newblock On a relaxation approximation of the incompressible {N}avier-{S}tokes
  equations.
\newblock {\em Proc. Amer. Math. Soc.}, 132(4):1021--1028, 2004.

\bibitem{carrassi1972AMN}
M.~Carrassi and A.~Morro.
\newblock A modified navier-stokes equation, and its consequences on sound
  dispersion.
\newblock {\em Il Nuovo Cimento B (1971-1996)}, 9:321--343, 1972.

\bibitem{MR0032898}
C.~Cattaneo.
\newblock Sulla conduzione del calore.
\newblock {\em Atti Sem. Mat. Fis. Univ. Modena}, 3, 1949.

\bibitem{MR95680}
C.~Cattaneo.
\newblock Sur une forme de l'\'{e}quation de la chaleur \'{e}liminant le
  paradoxe d'une propagation instantan\'{e}e.
\newblock {\em C. R. Acad. Sci. Paris}, 247:431--433, 1958.

\bibitem{Coulaud}
O.~Coulaud, I.~Hachicha, and G.~Raugel.
\newblock {Hyperbolic Quasilinear Navier–Stokes Equations in $ \mathbb R^2$}.
\newblock {\em J Dyn Diff Equat}, 2021.

\bibitem{davidson_2001}
P.~A. Davidson.
\newblock {\em An Introduction to Magnetohydrodynamics}.
\newblock Cambridge Texts in Applied Mathematics. Cambridge University Press,
  2001.

\bibitem{MR3925144}
H.~Dietert and D.~G\'{e}rard-Varet.
\newblock Well-posedness of the {P}randtl equations without any structural
  assumption.
\newblock {\em Ann. PDE}, 5(1):Paper No. 8, 51, 2019.

\bibitem{MR3490904}
D.~Donatelli and S.~Spirito.
\newblock Vanishing dielectric constant regime for the {N}avier {S}tokes
  {M}axwell equations.
\newblock {\em NoDEA Nonlinear Differential Equations Appl.}, 23(3):Art. 28,
  19, 2016.

\bibitem{MR2601044}
D.~G\'{e}rard-Varet and E.~Dormy.
\newblock On the ill-posedness of the {P}randtl equation.
\newblock {\em J. Amer. Math. Soc.}, 23(2):591--609, 2010.

\bibitem{MR2952715}
D.~G\'{e}rard-Varet and T.~Nguyen.
\newblock Remarks on the ill-posedness of the {P}randtl equation.
\newblock {\em Asymptot. Anal.}, 77(1-2):71--88, 2012.

\bibitem{MR3657241}
D.~G\'{e}rard-Varet and M.~Prestipino.
\newblock Formal derivation and stability analysis of boundary layer models in
  {MHD}.
\newblock {\em Z. Angew. Math. Phys.}, 68(3):Paper No. 76, 16, 2017.

\bibitem{Masmoudi3D}
P.~Germain, S.~Ibrahim, and N.~Masmoudi.
\newblock Well-posedness of the {N}avier-{S}tokes-{M}axwell equations.
\newblock {\em Proc. Roy. Soc. Edinburgh Sect. A}, 144(1):71--86, 2014.

\bibitem{hartmann1937theory}
J.~Hartmann.
\newblock Theory of laminar flow of an electrically conductive liquid in a
  homogeneous magnetic field.
\newblock {\em Selsk. Mat. Fys. Medd.}, 15(6):1--28, 1937.

\bibitem{Isern_2017}
J.~Isern, E.~Garc{\'{\i}}a-Berro, B.~Külebi, and P.~Lor{\'{e}}n-Aguilar.
\newblock A common origin of magnetism from planets to white dwarfs.
\newblock {\em The Astrophysical Journal}, 836(2):L28, feb 2017.

\bibitem{Ogawa}
S.~Kawashima, R.~Nakasato, and T.~Ogawa.
\newblock Global well-posedness and time-decay of solutions for the
  compressible {H}all-magnetohydrodynamic system in the critical {B}esov
  framework.
\newblock {\em J. Differential Equations}, 328:1--64, 2022.

\bibitem{10.1093/mnras/staa774}
P.~Kumar and Ž. Bošnjak.
\newblock {FRB coherent emission from decay of Alfvén waves}.
\newblock {\em Monthly Notices of the Royal Astronomical Society},
  494(2):2385--2395, 03 2020.

\bibitem{MR4270479}
W.-X. Li and T.~Yang.
\newblock Well-posedness of the {MHD} boundary layer system in {G}evrey
  function space without structural assumption.
\newblock {\em SIAM J. Math. Anal.}, 53(3):3236--3264, 2021.

\bibitem{MR3864769}
C.-J. Liu, F.~Xie, and T.~Yang.
\newblock A note on the ill-posedness of shear flow for the {MHD} boundary
  layer equations.
\newblock {\em Sci. China Math.}, 61(11):2065--2078, 2018.

\bibitem{MR3975147}
C.-J. Liu, F.~Xie, and T.~Yang.
\newblock Justification of {P}randtl ansatz for {MHD} boundary layer.
\newblock {\em SIAM J. Math. Anal.}, 51(3):2748--2791, 2019.

\bibitem{MR4213671}
N.~Liu and P.~Zhang.
\newblock Global small analytic solutions of {MHD} boundary layer equations.
\newblock {\em J. Differential Equations}, 281:199--257, 2021.

\bibitem{Masmoudi2D}
N.~Masmoudi.
\newblock Global well posedness for the {M}axwell-{N}avier-{S}tokes system in
  2{D}.
\newblock {\em J. Math. Pures Appl. (9)}, 93(6):559--571, 2010.

\bibitem{10.1111/j.1365-2966.2004.07898.x}
P.~J. Meintjes.
\newblock {Magnetized fragmented mass transfer in cataclysmic variables: AE
  Aquarii, a trial case}.
\newblock {\em Monthly Notices of the Royal Astronomical Society},
  352(2):416--426, 08 2004.

\bibitem{PhysRevD.64.044009}
G.~Mendell.
\newblock Magnetic effects on the viscous boundary layer damping of the r-modes
  in neutron stars.
\newblock {\em Phys. Rev. D}, 64:044009, Jul 2001.

\bibitem{Popham_2001}
Robert P. and Rashid S.
\newblock Accretion disk boundary layers around neutron stars: X-ray production
  in low-mass x-ray binaries.
\newblock {\em The Astrophysical Journal}, 547(1):355--383, jan 2001.

\bibitem{MR2404054}
M.~Paicu and G.~Raugel.
\newblock Une perturbation hyperbolique des \'{e}quations de {N}avier-{S}tokes.
\newblock In {\em E{SAIM} {P}roceedings.[{J}ourn\'{e}es d'{A}nalyse
  {F}onctionnelle et {N}um\'{e}rique en l'honneur de {M}ichel {C}rouzeix]},
  volume~21 of {\em ESAIM Proc.}, pages 65--87, 2007.

\bibitem{MR4271962}
M.~Paicu and P.~Zhang.
\newblock Global existence and the decay of solutions to the {P}randtl system
  with small analytic data.
\newblock {\em Arch. Ration. Mech. Anal.}, 241(1):403--446, 2021.

\bibitem{MR4429384}
M.~Paicu and P.~Zhang.
\newblock Global hydrostatic approximation of the hyperbolic {N}avier-{S}tokes
  system with small {G}evrey class 2 data.
\newblock {\em Sci. China Math.}, 65(6):1109--1146, 2022.

\bibitem{MR4125518}
M.~Paicu, P.~Zhang, and Z.~Zhang.
\newblock On the hydrostatic approximation of the {N}avier-{S}tokes equations
  in a thin strip.
\newblock {\em Adv. Math.}, 372:107293, 42, 2020.

\bibitem{MR4328431}
M.~Paicu and N.~Zhu.
\newblock Global regularity for the 2{D} {MHD} and tropical climate model with
  horizontal dissipation.
\newblock {\em J. Nonlinear Sci.}, 31(6):Paper No. 99, 39, 2021.

\bibitem{priest_2014}
E.~Priest.
\newblock {\em Magnetohydrodynamics of the Sun}.
\newblock Cambridge University Press, 2014.

\bibitem{MR3085226}
R.~Racke and J.~Saal.
\newblock Hyperbolic {N}avier-{S}tokes equations {II}: {G}lobal existence of
  small solutions.
\newblock {\em Evol. Equ. Control Theory}, 1(1):217--234, 2012.

\bibitem{Voros}
Z.~Vörös, A.~Varsani, E.~Yordanova, Y.~L. Sasunov, O.~W. Roberts, Á. Kis,
  R.~Nakamura, and Y.~Narita.
\newblock Magnetic reconnection within the boundary layer of a magnetic cloud
  in the solar wind.
\newblock {\em Journal of Geophysical Research: Space Physics},
  126(9):e2021JA029415, 2021.

\end{thebibliography}
\bibliographystyle{plain}

\end{document}